\documentclass[11pt, reqno]{amsart}
\usepackage{amsmath, amsthm, amssymb, amsfonts, mathrsfs}
\usepackage[a4paper, left=3.2cm, right=3.2cm, top=3.2cm, bottom=3.2cm]{geometry}
\usepackage[hidelinks]{hyperref}
\numberwithin{equation}{section}

\title[The Hasse principle for random Fano hypersurfaces]{The Hasse principle for random Fano hypersurfaces}

\author{Tim Browning}
\address{IST Austria \\ Am Campus 1 \\ 3400 Klosterneuburg \\ Austria}
\email{tdb@ist.ac.at}

\author{Pierre Le Boudec}
\address{Departement Mathematik und Informatik \\ Fachbereich Mathematik \\ Spiegelgasse 1 \\ 4051 Basel \\ Switzerland}
\email{pierre.leboudec@unibas.ch}

\author{Will Sawin}
\address{Columbia University \\ Department of Mathematics \\ 2990 Broadway \\ New York \\ NY $10027$ \\ USA}
\email{sawin@math.columbia.edu}


\begin{document}

\newtheorem{theorem}{Theorem}[section]
\newtheorem{lemma}[theorem]{Lemma}
\newtheorem{proposition}[theorem]{Proposition}
\newtheorem{corollary}[theorem]{Corollary}
\newtheorem{conjecture}[theorem]{Conjecture}
\theoremstyle{definition}
\newtheorem{definition}[theorem]{Definition}

\newcommand{\Gal}{\operatorname{Gal}}
\newcommand{\rad}{\operatorname{rad}}
\newcommand{\Span}{\operatorname{Span}}
\newcommand{\vol}{\operatorname{vol}}
\newcommand{\GL}{\operatorname{GL}}
\newcommand{\sign}{\operatorname{sgn}}

\begin{abstract}
It is known that the Brauer--Manin obstruction to the Hasse principle is vacuous for smooth Fano hypersurfaces of dimension at least $3$ over any number field. Moreover, for such varieties it follows from a general conjecture of
Colliot-Th\'{e}l\`{e}ne that the Brauer--Manin obstruction to the Hasse principle should be the only one, so that the Hasse principle is expected to hold. Working over the field of rational numbers and ordering Fano hypersurfaces of fixed degree and dimension by height, we prove that almost every such hypersurface satisfies the Hasse principle provided that the dimension is at least $3$. This proves a conjecture of Poonen and Voloch in every case except for cubic surfaces.
\end{abstract}

\maketitle

\thispagestyle{empty}
\setcounter{tocdepth}{1}
\tableofcontents

\section{Introduction}

Let $d, n \geq 2$ be such that $n \geq d$ and let $N_{d,n} = \binom{n+d}{d}$ be the number of monomials of degree
$d$ in $n+1$ variables. Ordering monomials lexicographically, degree $d$ hypersurfaces in $\mathbb{P}^n$ that are defined over $\mathbb{Q}$ are parametrized by $\mathbb{V}_{d,n} = \mathbb{P}^{N_{d,n}-1}(\mathbb{Q})$. It follows from the assumption $n \geq d$ that a generic element of $\mathbb{V}_{d,n}$ is a smooth Fano hypersurface.

We shall order elements of $\mathbb{V}_{d,n}$ using the usual exponential height on projective space. With this in mind, for any $N \geq 1$, let $\mathbb{Z}_{\mathrm{prim}}^N$ be the set of $(c_1, \dots, c_N) \in \mathbb{Z}^N$ such that
$\gcd(c_1, \dots, c_N) =1$ and let $|| \cdot ||$ be the Euclidean norm in $\mathbb{R}^N$. The height of
$V \in \mathbb{V}_{d,n}$ is then defined to be $||\mathbf{a}_V||$ where
$\mathbf{a}_V \in \mathbb{Z}_{\mathrm{prim}}^{N_{d,n}}$ denotes any of the two primitive coefficient vectors associated to $V$. Moreover, for any $A \geq 1$, we let
\begin{equation*}
\mathbb{V}_{d,n}(A) = \left\{ V \in \mathbb{V}_{d,n} : ||\mathbf{a}_V|| \leq A \right\}.
\end{equation*}
The primary goal of this article is to investigate the asymptotic behaviour of the quantity
\begin{equation*}
\varrho_{d,n}(A) = \frac{\# \{ V \in \mathbb{V}_{d,n}(A) : V(\mathbb{Q}) \neq \emptyset \}}{\#\mathbb{V}_{d,n}(A)},
\end{equation*}
as $A \to \infty$. The ratio $\varrho_{d,n}(A)$ is the proportion of degree $d$ hypersurfaces in $\mathbb{P}^n$ which are defined over $\mathbb{Q}$, have height at most $A$, and admit a rational point.

For any $V \in \mathbb{V}_{d,n}$, we let $V(\mathbf{A}_{\mathbb{Q}})$ denote the set of ad\`{e}les of $V$. We introduce the set $\mathbb{V}_{d,n}^{\mathrm{loc}}$ of elements of $\mathbb{V}_{d,n}$ that are everywhere locally soluble, that is
\begin{equation*}
\mathbb{V}_{d,n}^{\mathrm{loc}} = \left\{ V \in \mathbb{V}_{d,n} : V(\mathbf{A}_{\mathbb{Q}}) \neq \emptyset \right\}.
\end{equation*}
We also let
\begin{equation}
\label{Definition Vdnloc}
\mathbb{V}_{d,n}^{\mathrm{loc}}(A) = \mathbb{V}_{d,n}^{\mathrm{loc}} \cap \mathbb{V}_{d,n}(A),
\end{equation}
and we denote the density of the set $\mathbb{V}_{d,n}^{\mathrm{loc}}$ by
\begin{equation*}
\varrho_{d,n}^{\mathrm{loc}} = \lim_{A \to \infty} \frac{\# \mathbb{V}_{d,n}^{\mathrm{loc}}(A)}{\# \mathbb{V}_{d,n}(A)},
\end{equation*}
whenever this limit exists. In the case $(d,n)=(2,2)$, work of Serre \cite[Exemple~$4$]{MR1075658} shows that a typical rational plane conic is not everywhere locally soluble, that is
\begin{equation}
\label{Equality Serre}
\varrho_{2,2}^{\mathrm{loc}}=0.
\end{equation}
Note that a far-reaching interpretation of this phenomenon can be found in recent work of Loughran \cite{MR3852186}. If $(d,n) \neq (2,2)$, Poonen and Voloch prove \cite[Theorem~3.6]{MR2029869} that $\varrho_{d,n}^{\mathrm{loc}}$ exists, is equal to a product of local densities and moreover
\begin{equation}
\label{Lower bound loc}
\varrho_{d,n}^{\mathrm{loc}} > 0.
\end{equation}
Put another way, the proportion of degree $d$ hypersurfaces in $\mathbb{P}^n$ defined over $\mathbb{Q}$, which are everywhere locally soluble, exists and is positive. Furthermore, Poonen and Voloch conjecture
\cite[Conjecture~$2.2$.(ii)]{MR2029869} that $\varrho_{d,n}(A)$ tends to a limit as $A \to \infty$ and
\begin{equation}
\label{Conjecture PV}
\lim_{A \to \infty} \varrho_{d,n}(A) = \varrho_{d,n}^{\mathrm{loc}}.
\end{equation}
They check \cite[Proposition~$3$.$4$]{MR2029869} that their prediction follows from Colliot-Th\'{e}l\`{e}ne's conjecture \cite{MR2011747} that for smooth, proper, geometrically integral and rationally connected varieties, the Brauer--Manin obstruction to the Hasse principle is the only one. Indeed, Colliot-Th\'{e}l\`{e}ne shows in an appendix to their work \cite[Corollary~A.$2$]{MR2029869} that there is no Brauer--Manin obstruction when $n \geq 4$. The remaining case $(d,n)=(3,3)$ of cubic surfaces relies on a result of Swinnerton-Dyer \cite{MR1207510} asserting that the
Brauer--Manin obstruction is vacuous when the action of the Galois group
$\Gal(\overline{\mathbb{Q}}/\mathbb{Q})$ on the $27$ lines is the full Weyl group $W(\mathbf{E}_6)$. The equality \eqref{Conjecture PV} then follows from Hilbert's irreducibility theorem.

We remark that in the case of general type hypersurfaces, that is when $d>n+1$, Poonen and Voloch
\cite[Conjecture~$2$.$2$.(i)]{MR2029869} conjecture that the ratio $\varrho_{d,n}(A)$ should approach $0$ as
$A \to \infty$, but this range of variables lies outside the scope of the present investigation.

The expectation \eqref{Conjecture PV} holds if either $d=2$ or $n \geq (d-1) 2^d$, as it follows respectively from the Hasse--Minkowski theorem and the celebrated work of Birch \cite{MR150129}. In addition, we also note that in the setting of diagonal hypersurfaces, Br\"{u}dern and Dietmann have confirmed in \cite[Theorem~$1$.$3$]{MR3177289} that the analogue of the equality \eqref{Conjecture PV} holds under the assumption $n>3d$.

The following is our main result and only leaves open the case of cubic surfaces in the Poonen--Voloch conjecture for Fano hypersurfaces.

\begin{theorem}
\label{Theorem 1}
Let $d \geq 2$ and $n \geq d$ with $(d,n) \neq (3,3)$. Then we have
\begin{equation*}
\lim_{A \to \infty} \varrho_{d,n}(A) = \varrho_{d,n}^{\mathrm{loc}}.
\end{equation*}
\end{theorem}

In other words, in the Fano range $n \geq d$ and when degree $d$ hypersurfaces in $\mathbb{P}^n$ that are defined over $\mathbb{Q}$ are ordered by height, $100 \%$ of these hypersurfaces satisfy the Hasse principle provided that
$(d,n) \neq (3,3)$. In fact, it transpires from Propositions~\ref{Proposition 1} and \ref{Proposition 2} that
\begin{equation*}
\frac{\# \{V \in \mathbb{V}_{d,n}^{\mathrm{loc}}(A) : V(\mathbb{Q})=\emptyset\}}{\# \mathbb{V}_{d,n}(A)} \ll
\frac1{(\log A)^{1/48n}}.
\end{equation*}
We have not made any effort to optimise the exponent of $\log A$ in this upper bound.

Unfortunately, in the case $(d,n)=(3,3)$ our understanding of the geometry of the lattices involved in our work does not allow us to establish the equality \eqref{Conjecture PV}. However we can still show that a positive proportion of cubic surfaces have a rational point. Indeed, it suffices to consider the more stringent constraint in which the rational points are restricted to lie in one of the coordinate hyperplanes. This reduces the analysis to the case of plane cubic curves and we can therefore appeal to the work of Bhargava~\cite[Theorem~$2$]{Bhargava} to conclude. Recalling the lower bound \eqref{Lower bound loc}, we see that we have the following corollary of Theorem~\ref{Theorem 1}.

\begin{corollary}
\label{Corollary 1}
Let $d \geq 2$ and $n \geq d$ with $(d,n) \neq (2,2)$. Then we have
\begin{equation*}
\liminf_{A \to \infty} \varrho_{d,n}(A) > 0.
\end{equation*}
\end{corollary}

Corollary~\ref{Corollary 1} states that, putting aside the particular case of plane conics, a positive proportion of Fano hypersurfaces of fixed degree and dimension admit a rational point.

Our methods actually allow us to prove a much stronger result than Theorem~\ref{Theorem 1}. Indeed, as stated in Theorem~\ref{Theorem height}, we are able to estimate in an optimal way the smallest height of a rational point on a hypersurface for $100 \%$ of everywhere locally soluble hypersurfaces.
 
The proof of Theorem~\ref{Theorem height} relies upon various arguments coming from the geometry of numbers, together with a careful study of local densities. To establish results such as Theorem~\ref{Theorem height}, it is customary to prove that the number of rational points of bounded height on a hypersurface is on average well-approximated by an adequate quantity, which is traditionally taken to be the main term in the asymptotic formula predicted by the Hardy--Littlewood circle method. Unfortunately in our setting this object is rather complicated to analyse and we shall replace it by a carefully chosen localised counting function, which is designed to approximate the Hardy--Littlewood expectation and yet remain amenable to analysis via the geometry of numbers.

\subsection*{Acknowledgements}

During the preparation of this article the first-named author was supported by EPSRC grant EP/P$026710/1$
and FWF grant P $32428$-N$35$. The research of the second-named author is integrally funded by the Swiss National Science Foundation through the SNSF Professorship number $170565$ awarded to the project
\textit{Height of rational points on algebraic varieties}. Both the financial support of the SNSF and the perfect working conditions provided by the University of Basel are gratefully acknowledged. The research was conducted during the period that the third-named author served as a Clay Research Fellow, and partially conducted during the period he was supported by Dr. Max R\"{o}ssler, the Walter Haefner Foundation and the ETH Zurich Foundation. Finally, the authors are grateful to Menny Aka, Manfred Einsiedler, Roger Heath-Brown, David Masser, Peter Sarnak and Andreas Wieser for interesting conversations, and the organisers of the $2019$ trimester programme
\textit{\`{A} la red\'{e}couverte des points rationnels} at the Institut Henri Poincar\'{e} in Paris, where the project was drawn to a close.

\section{Roadmap of the proof}

Our purpose in this section is to describe our strategy for proving Theorem~\ref{Theorem 1}. Associated to any Fano hypersurface $V \in \mathbb{V}_{d,n}$ is the anticanonical height function $H : V(\mathbb{Q}) \to \mathbb{R}_{> 0}$ metrised by the Euclidean norm $|| \cdot ||$ in $\mathbb{R}^{n+1}$. Thus, for $x \in V(\mathbb{Q})$ we choose
$\mathbf{x}=(x_0, \dots, x_n) \in \mathbb{Z}_{\mathrm{prim}}^{n+1}$ such that $x=(x_0 : \dots : x_n)$ and we set
\begin{equation}
\label{Definition height}
H(x) = ||\mathbf{x}||^{n+1-d}.
\end{equation}
This allows us to define the counting function
\begin{equation}
\label{Definition counting}
N_V(B) = \# \{ x \in V(\mathbb{Q}) : H(x) \leq B \}.
\end{equation}
To tackle Theorem~\ref{Theorem 1} we will first show that $N_V(B)$ is on average well-approximated by a certain localised counting function, and we will then prove that the localised counting function is only rarely smaller than its expected value.

Given $N \geq 1$ and $\mathbf{c} \in \mathbb{Z}^N$, of special importance in our work is the integral lattice
\begin{equation}
\label{Definition lattice}
\Lambda_{\mathbf{c}} = \{ \mathbf{y} \in \mathbb{Z}^N : \langle \mathbf{c}, \mathbf{y} \rangle = 0 \},
\end{equation}
where $\langle \cdot, \cdot \rangle$ denotes the usual Euclidean inner product in $\mathbb{R}^N$. In addition, it will be very convenient to introduce the following notation.

\begin{definition}
\label{Definition Veronese}
Given $d,n \geq 1$, we let $\nu_{d,n} : \mathbb{R}^{n+1} \to \mathbb{R}^{N_{d,n}}$ denote the Veronese embedding, defined by listing all the monomials of degree $d$ in $n+1$ variables using the lexicographical ordering.
\end{definition}

We see that
\begin{equation}
\label{Equality counting}
N_V(B) = \frac1{2} \sum_{\substack{\mathbf{x} \in \Xi_{d,n}(B) \\ \mathbf{a}_V \in \Lambda_{\nu_{d,n}(\mathbf{x})}}} 1,
\end{equation}
where $\mathbf{a}_V \in \mathbb{Z}_{\mathrm{prim}}^{N_{d,n}}$ denotes any of the two primitive coefficient vectors associated to $V$, and where we have set
\begin{equation}
\label{Definition Xi}
\Xi_{d,n}(B) = \left\{ \mathbf{x} \in \mathbb{Z}^{n+1}_{\mathrm{prim}} : ||\mathbf{x}|| \leq B^{1/(n+1-d)} \right\}.
\end{equation}

Manin's conjecture \cite{MR974910} gives a precise prediction for the asymptotic behaviour of $N_V(B)$ as $B \to \infty$ for Fano hypersurfaces $V \in \mathbb{V}_{d,n}$. We remark that most $V \in \mathbb{V}_{d,n}$ do not possess accumulating thin subsets and have Picard group isomorphic to $\mathbb{Z}$. Thus $N_V(B)$ is expected to grow linearly in terms of $B$ whenever $V(\mathbb{Q})$ is Zariski dense in $V$. The localised counting function we work with is chosen to mimic the main term in this expected asymptotic formula.

For any $N \geq 1$, any real $\gamma > 0$ and $\mathbf{v} \in \mathbb{R}^{N}$, we introduce the region
\begin{equation}
\label{Definition Cgamma}
\mathcal{C}_{\mathbf{v}}^{(\gamma)} = \left\{\mathbf{t} \in \mathbb{R}^N :
|\langle \mathbf{v}, \mathbf{t} \rangle |\leq \frac{||\mathbf{v}|| \cdot ||\mathbf{t}||}{2\gamma}\right\},
\end{equation}
and for any $Q \geq 1$ and $\mathbf{c} \in \mathbb{Z}^N$, we define the lattice
\begin{equation}
\label{Definition local lattice}
\Lambda_{\mathbf{c}}^{(Q)} = \left\{\mathbf{y}\in \mathbb{Z}^N : 
\langle \mathbf{c}, \mathbf{y} \rangle \equiv 0 \bmod{Q} \right\}.
\end{equation}
Furthermore, we set
\begin{equation}
\label{Definition alpha}
\alpha = \log B,
\end{equation}
and
\begin{equation}
\label{Definition W}
W = \prod_{p \leq w} p^{\lceil \log w/ \log p \rceil+1},
\end{equation}
where
\begin{equation}
\label{Definition w}
w = \frac{\log B}{\log \log B}.
\end{equation}
Our localised counting function is then defined as
\begin{equation}
\label{Definition local counting}
N_V^{\mathrm{loc}}(B) = \frac1{2} \cdot \frac{\alpha W}{||\mathbf{a}_V||}
\sum_{\substack{\mathbf{x} \in \Xi_{d,n}(B) \\
\mathbf{a}_V \in \Lambda_{\nu_{d,n}(\mathbf{x})}^{(W)} \cap \mathcal{C}_{\nu_{d,n}(\mathbf{x})}^{(\alpha)}}}
\frac1{||\nu_{d,n}(\mathbf{x})||}.
\end{equation}

The main contribution to $\log W$ comes from the primes $p \in (w^{1/2},w)$. Therefore, an application of the prime number theorem reveals that $\log W \sim 3 w$, which implies in particular
\begin{equation}
\label{Upper bound W}
W \ll B^{4/ \log \log B}.
\end{equation}
We thus see that $\alpha$ and $W$ both tend to infinity rather slowly with respect to $B$. This fact together with the observation that $W$ becomes more and more divisible as $B$ grows will turn out to be crucial in our argument.

Our methods not only allow us to prove that $100 \%$ of the everywhere locally soluble Fano hypersurfaces
$V \in \mathbb{V}_{d,n}$ admit a rational point, but we actually obtain an upper bound for the smallest height of a rational point on $V$. Recall the definition \eqref{Definition height} of the anticanonical height $H$. For any
$V \in \mathbb{V}_{d,n}$, it is convenient to define
\begin{equation*}
\mathfrak{M}(V) =
\begin{cases}
\displaystyle{\min_{x \in V(\mathbb{Q})}} H(x), & \textrm{if } V(\mathbb{Q}) \neq \emptyset, \\
\infty, & \textrm{if } V(\mathbb{Q}) = \emptyset.
\end{cases}
\end{equation*}
The interested reader is invited to refer to the introduction of \cite{RandomFano} for a survey of works on the quantity
$\mathfrak{M}(V)$ in the setting of Fano hypersurfaces.

Theorem~\ref{Theorem 1} is an immediate consequence of the following result.

\begin{theorem}
\label{Theorem height}
Let $d \geq 2$ and $n \geq d$ with $(d,n) \neq (3,3)$. Let $\psi : \mathbb{R}_{>0} \to \mathbb{R}_{>0}$ be such that
$\psi(u)/u \to \infty$ as $u \to \infty$. Then we have
\begin{equation*}
\lim_{A \to \infty}
\frac{\# \left\{V \in \mathbb{V}_{d,n}(A) : \mathfrak{M}(V) \leq \psi(||\mathbf{a}_V||) \right\}}{\#\mathbb{V}_{d,n}(A)} = \varrho_{d,n}^{\mathrm{loc}}.
\end{equation*}
\end{theorem}

Combining Theorem~\ref{Theorem height} with work of the second author \cite[Theorem~$1$]{RandomFano}, we deduce that if $(d,n) \notin \{ (2,2), (3,3)\}$ and if $\xi : \mathbb{R}_{>0} \to \mathbb{R}_{>0}$ is any function satisfying $\xi(u) \to \infty$ as $u \to \infty$, then for $100 \%$ of the everywhere locally soluble Fano hypersurfaces
$V \in \mathbb{V}_{d,n}$ we have the optimal inequalities
\begin{equation*}
\frac1{\xi(||\mathbf{a}_V||)} < \frac{\mathfrak{M}(V)}{||\mathbf{a}_V||} \leq \xi(||\mathbf{a}_V||).
\end{equation*}
In the case $d=2$, Cassels' celebrated bound \cite{MR69217} states that for any hypersurface
$V \in \mathbb{V}_{2,n}^{\mathrm{loc}}$ we have
\begin{equation}
\label{Upper bound Cassels}
\mathfrak{M}(V) \ll ||\mathbf{a}_V||^{n(n-1)/2},
\end{equation}
where the implied constant depends at most on $n$. Note that the exponent is known to be optimal thanks to Kneser's example \cite{MR81306}. Theorem~\ref{Theorem height} shows that for typical quadratic hypersurfaces
$V \in \mathbb{V}_{2,n}^{\mathrm{loc}}$, Cassels' upper bound \eqref{Upper bound Cassels} is very far from the truth as soon as $n \geq 3$.

Heuristically we expect the counting functions $N_V(B)$ and $N_V^{\mathrm{loc}}(B)$ to be of exact order $B/A$, for generic $V \in \mathbb{V}_{d,n}(A)$. Our first result shows that when the ratio $B/A$ tends to $\infty$ sufficiently slowly with respect to $A$ then it is rare for $N_V(B)$ not to be well-approximated by $N_V^{\mathrm{loc}}(B)$, as $V$ runs over the set $\mathbb{V}_{d,n}(A)$. We stress that here and throughout Sections~\ref{Section approximating} and \ref{Section localised rarely small}, all the implied constants depend at most on $d$ and $n$ unless specified otherwise.

\begin{proposition}
\label{Proposition 1}
Let $d \geq 2$ and $n \geq d$ with $(d,n) \notin \{ (2,2), (3,3) \}$. Let $\phi : \mathbb{R}_{>0} \to \mathbb{R}_{>1}$ be such that $\phi(A) \leq (\log A)^{1/2}$. Then we have
\begin{equation*}
\frac1{\#\mathbb{V}_{d,n}(A)} \cdot
\# \left\{V \in \mathbb{V}_{d,n}(A) : \left| N_V(A\phi(A)) - N_V^{\mathrm{loc}}(A\phi(A)) \right| > \phi(A)^{2/3} \right\} \ll
\frac1{\phi(A)^{1/3}}.
\end{equation*}
\end{proposition}

Proposition~\ref{Proposition 1} will follow directly from Proposition~\ref{Proposition variance}, which provides a sharp upper bound for the variance
\begin{equation*}
\sum_{V \in \mathbb{V}_{d,n}(A)} \left( N_V(A\phi(A)) - N_V^{\mathrm{loc}}(A\phi(A)) \right)^2.
\end{equation*}
In order to prove Proposition~\ref{Proposition variance} we will start in Section~\ref{Section geometry of numbers} by gathering a series of tools coming from the geometry of numbers. Section~\ref{Section approximating} will then be devoted to the proof of Proposition~\ref{Proposition variance}.

Our second major ingredient in the proof of Theorem~\ref{Theorem height} states that if the ratio $B/A$ tends to
$\infty$ as $A$ tends to $\infty$ and satisfies a mild upper bound, then it is rare for the localised counting function $N_V^{\mathrm{loc}}(B)$ to be smaller than its expected size, as $V$ runs over the set
$\mathbb{V}_{d,n}^{\mathrm{loc}}(A)$.

\begin{proposition}
\label{Proposition 2}
Let $d \geq 2$ and $n \geq d$ with $(d,n) \neq (2,2)$. Let $\phi : \mathbb{R}_{>0} \to \mathbb{R}_{>1}$ be such that
$\phi(A) \leq A^{3/n}$. Then we have
\begin{equation*}
\frac1{\# \mathbb{V}_{d,n}^{\mathrm{loc}}(A)} \cdot
\# \left\{ V \in \mathbb{V}_{d,n}^{\mathrm{loc}}(A) : N_V^{\mathrm{loc}}(A \phi(A)) \leq \phi(A)^{2/3} \right\} \ll
\frac1{\phi(A)^{1/24n}}.
\end{equation*}
\end{proposition}

Proposition~\ref{Proposition 2} will be established in Section~\ref{Section localised rarely small}. The proof will consist in checking that certain non-Archimedean and Archimedean factors that are hidden in our localised counting function $N_V^{\mathrm{loc}}(A \phi(A))$ are rarely small as $V$ runs over $\mathbb{V}_{d,n}^{\mathrm{loc}}(A)$, as stated in Propositions~\ref{Proposition non-Archimedean} and \ref{Proposition Archimedean}. It may be worth noting that Section~\ref{Section localised rarely small} is independent from Sections~\ref{Section geometry of numbers} and
\ref{Section approximating}.

We now proceed to prove that Theorem~\ref{Theorem height} follows from Propositions~\ref{Proposition 1} and \ref{Proposition 2}.

\begin{proof}[Proof of Theorem~\ref{Theorem height}]
The case $(d,n)=(2,2)$ of plane conics is a direct consequence of the equality \eqref{Equality Serre} so we assume that $(d,n) \neq (2,2)$. We set
\begin{equation*}
\mathscr{P}_{d,n}(A) =
\# \left\{V \in \mathbb{V}_{d,n}^{\mathrm{loc}}(A) : \mathfrak{M}(V) > \psi(||\mathbf{a}_V||) \right\},
\end{equation*}
and we observe that our goal is to prove that
\begin{equation}
\label{Goal height}
\lim_{A \to \infty} \frac{\mathscr{P}_{d,n}(A)}{\#\mathbb{V}_{d,n}(A)} = 0.
\end{equation}
Let $\eta \in (0,1)$. Since we clearly have $\mathbb{V}_{d,n}^{\mathrm{loc}}(\eta A) \subset \mathbb{V}_{d,n}(\eta A)$ and $\# \mathbb{V}_{d,n}(\eta A) \ll \eta^{N_{d,n}} A^{N_{d,n}}$, we see that
\begin{equation*}
\mathscr{P}_{d,n}(A) = \# \left\{V \in \mathbb{V}_{d,n}^{\mathrm{loc}}(A) :
\begin{array}{l l}
||\mathbf{a}_V|| > \eta A \\
\mathfrak{M}(V) > \psi(||\mathbf{a}_V||)
\end{array}
\right\} + O \left( \eta^{N_{d,n}} A^{N_{d,n}} \right).
\end{equation*}
By assumption if $A$ is large enough then for any $u \geq \eta A$ we have $\psi(u) \geq u/\eta^2$ and thus
$\psi(u) \geq A/\eta$. We deduce that
\begin{equation}
\label{Upper bound P}
\mathscr{P}_{d,n}(A) \ll \# \left\{V \in \mathbb{V}_{d,n}^{\mathrm{loc}}(A) : \mathfrak{M}(V) > \frac{A}{\eta} \right\} + \eta^{N_{d,n}} A^{N_{d,n}}.
\end{equation}
Next, we note that if a hypersurface $V \in \mathbb{V}_{d,n}^{\mathrm{loc}}(A)$ satisfies the lower bound
$\mathfrak{M}(V) > A/\eta$ then $N_V(A/\eta)=0$, which implies that we have either
\begin{equation}
\label{Inequality 1}
\left|N_V\left(\frac{A}{\eta}\right) - N_V^{\mathrm{loc}}\left(\frac{A}{\eta}\right)\right| > \frac1{\eta^{2/3}},
\end{equation}
or
\begin{equation}
\label{Inequality 2}
N_V^{\mathrm{loc}}\left(\frac{A}{\eta}\right) \leq \frac1{\eta^{2/3}}.
\end{equation}
As a result, taking $\phi(A) = 1/\eta$ in Propositions~\ref{Proposition 1} and \ref{Proposition 2} to bound the number of $V \in \mathbb{V}_{d,n}^{\mathrm{loc}}(A)$ satisfying either the lower bound \eqref{Inequality 1} or the upper bound \eqref{Inequality 2}, we derive
\begin{align}
\nonumber
\frac1{\# \mathbb{V}_{d,n}(A)} \cdot \# \left\{V \in \mathbb{V}_{d,n}^{\mathrm{loc}}(A) : \mathfrak{M}(V) > \frac{A}{\eta} \right\} & \ll \eta^{1/3} + \eta^{1/24n} \frac{\# \mathbb{V}_{d,n}^{\mathrm{loc}}(A)}{\# \mathbb{V}_{d,n}(A)} \\
\label{Upper bound after propositions}
& \ll \eta^{1/24n}.
\end{align}
Note that we have used the fact that $\eta \in (0,1)$. We now remark that we have the lower bound
\begin{equation}
\label{Lower bound Vdn}
\# \mathbb{V}_{d,n}(A) \gg A^{N_{d,n}}.
\end{equation}
Hence, putting together the upper bounds \eqref{Upper bound P} and \eqref{Upper bound after propositions} we obtain
\begin{equation*}
\limsup_{A \to \infty} \frac{\mathscr{P}_{d,n}(A)}{\#\mathbb{V}_{d,n}(A)} \ll \eta^{1/24n}.
\end{equation*}
Since this upper bound holds for any $\eta \in (0,1)$ we see that the equality \eqref{Goal height} follows, which completes the proof of Theorem~\ref{Theorem height}.
\end{proof}

\section{Tools from the geometry of numbers}

\label{Section geometry of numbers}

Our goal in this section is to gather all the geometry of numbers results that we will need to establish
Proposition~\ref{Proposition 1}. In Section~\ref{Section basic facts} we start by recalling classical facts about lattices in $\mathbb{R}^N$ and we prove a series of lattice point counting estimates. We will then calculate the determinants of certain lattices in Section~\ref{Section determinants}. Our next task in
Section~\ref{Section successive minima} will be to investigate the size of the successive minima of the key lattices
$\Lambda_{\nu_{d,n}(\mathbf{x})}$ and $\Lambda_{\nu_{d,n}(\mathbf{x})} \cap \Lambda_{\nu_{d,n}(\mathbf{y})}$ for given linearly independent vectors $\mathbf{x}, \mathbf{y} \in \mathbb{Z}^N$. In Section~\ref{Section typical size} we will then turn to estimating the typical size of some key quantities uncovered in Section~\ref{Section successive minima}. Finally, in Section~\ref{Section quartic threefolds} we will collect some further results that we will require to handle the specific case of quartic threefolds.

\subsection{Basic facts and lattice point counting estimates}

\label{Section basic facts}

Let $N \geq 1$. A lattice $\Lambda \subset \mathbb{R}^N$ is a discrete subgroup of $\mathbb{R}^N$. The dimension of the subspace $\Span_{\mathbb{R}}(\Lambda)$ is called the rank of $\Lambda$. If $\Lambda$ is a rank $R$ lattice
and if $(\mathbf{b}_1, \dots, \mathbf{b}_R)$ is any basis of $\Lambda$ then the determinant of $\Lambda$ is the $R$-dimensional volume of the fundamental parallelepiped spanned by $(\mathbf{b}_1, \dots, \mathbf{b}_R)$, and is thus given by
\begin{equation}
\label{Definition det}
\det(\Lambda) = \sqrt{\det(\mathbf{B}^T \mathbf{B})},
\end{equation}
where $\mathbf B$ is the $N \times R$ matrix whose columns are the vectors $\mathbf{b}_1, \dots, \mathbf{b}_R$.

Let $\Lambda \subset \mathbb{R}^N$ be a lattice and $\Gamma$ be a sublattice of $\Lambda$ such that
$\Span_{\mathbb{R}} (\Gamma) \cap \Lambda = \Gamma$, which is equivalent to saying that any basis of $\Gamma$ can be extended into a basis of $\Lambda$. Letting $\pi : \mathbb{R}^N \to \Span_{\mathbb{R}} (\Gamma)^{\perp}$ denote the orthogonal projection on $\Span_{\mathbb{R}} (\Gamma)^{\perp}$, we define the quotient lattice of
$\Lambda$ by $\Gamma$ by
\begin{equation*}
\Lambda/\Gamma = \pi(\Lambda).
\end{equation*}
It is clear that the rank of $\Lambda$ is equal to the sum of the ranks of $\Gamma$ and $\Lambda/\Gamma$. Moreover, the calculation of the determinant of $\Lambda$ using a basis of $\Lambda$ extending a basis of $\Gamma$ yields the equality
\begin{equation}
\label{Det quotient torsion free}
\det(\Lambda/\Gamma) = \frac{\det(\Lambda)}{\det(\Gamma)}.
\end{equation}

In addition, a lattice $\Lambda \subset \mathbb{R}^N$ is said to be integral if $\Lambda \subset \mathbb{Z}^N$. Moreover, an integral lattice $\Lambda$ of rank $R$ is said to be primitive if it is not properly contained in another integral lattice of rank $R$, that is if $\Span_{\mathbb{R}}(\Lambda) \cap \mathbb{Z}^N = \Lambda$. Given an integral lattice $\Lambda \subset \mathbb{Z}^N$, the lattice $\Lambda^{\perp}$ orthogonal to $\Lambda$ is defined by
\begin{equation*}
\Lambda^{\perp} = \left\{ \mathbf{a} \in \mathbb{Z}^N : \forall \mathbf{z} \in \Lambda \ \langle \mathbf{a}, \mathbf{z} \rangle = 0 \right\}.
\end{equation*}
It is clear that $\Lambda^{\perp}$ is a primitive lattice of rank $N-R$, and we see that
$(\Lambda^{\perp})^{\perp}=\Lambda$ if and only if $\Lambda$ is a primitive lattice. Furthermore, if $\Lambda$ is primitive then (see for example \cite[Corollary of Lemma~$1$]{MR224562}) we have
\begin{equation}
\label{Det orthogonal}
\det(\Lambda^{\perp}) = \det(\Lambda).
\end{equation}

The dual lattice $\Lambda^{\ast}$ of a lattice $\Lambda \subset \mathbb{R}^N$ is defined by
\begin{equation*}
\Lambda^{\ast} = \left\{ \mathbf{a} \in \Span_{\mathbb{R}} (\Lambda) :
\forall \mathbf{z} \in \Lambda \ \langle \mathbf{a}, \mathbf{z} \rangle \in \mathbb{Z} \right\}.
\end{equation*}
It is easy to check (see for instance \cite[Chapter~I, Lemma~$5$]{MR1434478}) that the lattices $\Lambda^{\ast}$ and
$\Lambda$ have equal rank, and moreover
\begin{equation}
\label{Det dual}
\det(\Lambda^{\ast}) = \frac1{\det(\Lambda)}.
\end{equation}
The following result is due to Schmidt \cite[Lemma~$1$]{MR224562}.

\begin{lemma}
\label{Lemma Schmidt quotient}
Let $N \geq 1$ and let $\Lambda \subset \mathbb{Z}^N$ be a primitive lattice. We have
\begin{equation*}
(\Lambda^{\perp})^{\ast} = \mathbb{Z}^N / \Lambda.
\end{equation*}
\end{lemma}

For any $u > 0$ we let
\begin{equation*}
\mathcal{B}_N(u) = \{ \mathbf{y} \in \mathbb{R}^N : ||\mathbf{y}|| \leq u \}
\end{equation*}
be the closed Euclidean ball of radius $u$ in $\mathbb{R}^N$. We now introduce notation for the successive minima of a lattice.

\begin{definition}
\label{Definition successive}
Let $N \geq 1$ and $R \in \{1, \dots, N\}$. Given a lattice $\Lambda \subset \mathbb{R}^N$ of rank $R$, the successive minima $\lambda_1(\Lambda), \dots, \lambda_R(\Lambda)$ of $\Lambda$ with respect to the unit ball
$\mathcal{B}_N(1)$ are defined for $i \in \{1, \dots, R\}$ by
\begin{equation*}
\lambda_i(\Lambda) = \inf \left\{ u \in \mathbb{R}_{> 0} :
\dim \left( \Span_{\mathbb{R}} ( \Lambda \cap \mathcal{B}_N(u) ) \right) \geq i \right\}.
\end{equation*}
\end{definition}

We clearly have $\lambda_1(\Lambda) \leq \dots \leq \lambda_R(\Lambda)$ and moreover Minkowski's second theorem (see for example \cite[Chapter VIII, Theorem V]{MR1434478}) states that
\begin{equation}
\label{Estimates Minkowski}
\det(\Lambda) \leq \lambda_1(\Lambda) \cdots \lambda_R(\Lambda) \ll \det(\Lambda),
\end{equation}
where the implied constant depends at most on $R$.

We now record a result which relates the successive minima of a lattice and those of its dual. The following version is due to Banaszczyk \cite[Theorem~$2.1$]{MR1233487}.

\begin{lemma}
\label{Lemma Banaszczyk}
Let $N \geq 1$ and $R \in \{1, \dots, N\}$. Let $\Lambda \subset \mathbb{R}^N$ be a lattice of rank $R$. For any
$i \in \{1, \dots, R \}$, we have
\begin{equation*}
\lambda_i(\Lambda) \leq \frac{R}{\lambda_{R-i+1}(\Lambda^{\ast})}.
\end{equation*}
\end{lemma}

We use the convention that empty products and empty summations are respectively equal to $1$ and $0$. Recall the definition \eqref{Definition Cgamma} of the region $\mathcal{C}_{\mathbf{v}}^{(\gamma)}$ for given
$\gamma > 0$ and $\mathbf{v} \in \mathbb{R}^{N}$. The following lattice point counting result will prove pivotal in our work.

\begin{lemma}
\label{Lemma lattice pivotal}
Let $N \geq 2$ and $R \in \{1, \dots, N\}$. Let $\Lambda \subset \mathbb{R}^N$ be a lattice of rank $R$. Let
$I \in \{1, \dots, N-1 \}$ and $\mathbf{v}_1, \dots, \mathbf{v}_I \in \mathbb{R}^N$. Let also $\gamma > 0$. Define
\begin{equation*}
\mathcal{R}_{\mathbf{v}_1, \dots, \mathbf{v}_I}(T, \gamma) =
\mathcal{B}_N(T) \cap \mathcal{C}_{\mathbf{v}_1}^{(\gamma)} \cap \cdots \cap \mathcal{C}_{\mathbf{v}_I}^{(\gamma)},
\end{equation*}
and
\begin{equation*}
\mathcal{V}_{\mathbf{v}_1, \dots, \mathbf{v}_I}(\Lambda; \gamma) =
\vol \left( \Span_{\mathbb{R}} ( \Lambda ) \cap \mathcal{R}_{\mathbf{v}_1, \dots, \mathbf{v}_I}(1, \gamma) \right).
\end{equation*}
Let $Y \geq \lambda_R(\Lambda)$. For $T \geq Y$, we have
\begin{equation*}
\# \left( \Lambda \cap \mathcal{R}_{\mathbf{v}_1, \dots, \mathbf{v}_I}(T, \gamma) \right) =
\frac{T^R}{\det (\Lambda)} \left( \mathcal{V}_{\mathbf{v}_1, \dots, \mathbf{v}_I}(\Lambda; \gamma) + O \left( \frac{Y}{T} \right) \right),
\end{equation*}
where the implied constant depends at most on $R$.
\end{lemma}

\begin{proof}
Let $\mathbf{O}$ be an $N \times N$ orthogonal matrix mapping $\Span_{\mathbb{R}} ( \Lambda )$ to the subspace of
$\mathbb{R}^N$ spanned by the first $R$ coordinates, which we simply denote by $\mathbb{R}^R$. Let
$\Gamma = \mathbf{O} \cdot \Lambda$ and $\mathcal{T}_{\mathbf{v}_1, \dots, \mathbf{v}_I}(T, \gamma) =
\mathbf{O} \cdot \mathcal{R}_{\mathbf{v}_1, \dots, \mathbf{v}_I}(T, \gamma)$. We clearly have
\begin{equation*}
\# \left( \Lambda \cap \mathcal{R}_{\mathbf{v}_1, \dots, \mathbf{v}_I}(T, \gamma) \right) = \# \left( \Gamma \cap\mathcal{T}_{\mathbf{v}_1, \dots, \mathbf{v}_I}(T, \gamma)\right).
\end{equation*}
The region $\mathcal{T}_{\mathbf{v}_1, \dots, \mathbf{v}_I}(T, \gamma)$ is a semi-algebraic set and is thus definable in an o-minimal structure. Hence, it follows from work of Barroero and Widmer \cite[Theorem~$1.3$]{MR3264671} that
\begin{equation*}
\# \left( \Gamma \cap \mathcal{T}_{\mathbf{v}_1, \dots, \mathbf{v}_I}(T, \gamma) \right) =
\frac{\vol \left( \mathbb{R}^R \cap \mathcal{T}_{\mathbf{v}_1, \dots, \mathbf{v}_I}(T, \gamma) \right)}
{\det (\Gamma)} + O \left( \sum_{i=1}^R \frac{T^{R-i}}{\lambda_1(\Gamma) \cdots \lambda_{R-i}(\Gamma)} \right).
\end{equation*}
Since the matrix $\mathbf{O}$ is orthogonal we have $\det(\Gamma) = \det(\Lambda)$ and also
$ \lambda_i(\Gamma)=\lambda_i(\Lambda)$ for any $i \in \{1, \dots, R\}$, and moreover
\begin{equation*}
\vol \left( \mathbb{R}^R \cap \mathcal{T}_{\mathbf{v}_1, \dots, \mathbf{v}_I}(T, \gamma) \right) =
T^R \mathcal{V}_{\mathbf{v}_1, \dots, \mathbf{v}_I}(\Lambda; \gamma).
\end{equation*}
We thus obtain
\begin{equation*}
\# \left( \Lambda \cap \mathcal{R}_{\mathbf{v}_1, \dots, \mathbf{v}_I}(T, \gamma) \right) =
\frac{T^R}{\det (\Lambda)} \mathcal{V}_{\mathbf{v}_1, \dots, \mathbf{v}_I}(\Lambda; \gamma) +
O \left( \sum_{i = 1}^R \frac{T^{R-i}}{\lambda_1(\Lambda) \cdots \lambda_{R-i}(\Lambda)} \right).
\end{equation*}
We deduce from Minkowski's theorem \eqref{Estimates Minkowski} that
\begin{equation*}
\# \left( \Lambda \cap \mathcal{R}_{\mathbf{v}_1, \dots, \mathbf{v}_I}(T, \gamma) \right) =
\frac{T^R}{\det (\Lambda)} \left( \mathcal{V}_{\mathbf{v}_1, \dots, \mathbf{v}_I}(\Lambda; \gamma) +
O \left( \sum_{i=1}^{R} \frac{\lambda_{R-i+1}(\Lambda) \cdots \lambda_R(\Lambda)}{T^i} \right) \right).
\end{equation*}
By assumption we have $\lambda_i(\Lambda) \leq Y$ for any $i \in \{1, \dots, R\}$. We deduce that
\begin{equation*}
\# \left( \Lambda \cap \mathcal{R}_{\mathbf{v}_1, \dots, \mathbf{v}_I}(T, \gamma) \right) =
\frac{T^R}{\det (\Lambda)} \left( \mathcal{V}_{\mathbf{v}_1, \dots, \mathbf{v}_I}(\Lambda; \gamma) +
O \left( \sum_{i=1}^{R} \frac{Y^i}{T^i} \right) \right),
\end{equation*}
which completes the proof since $T \geq Y$. 
\end{proof}

Our arguments will make intensive use of the following classical result.

\begin{lemma}
\label{Lemma Schmidt upper bound}
Let $N \geq 1$ and $R \in \{1, \dots, N\}$. Let $\Lambda \subset \mathbb{R}^N$ be a lattice of rank $R$. For
$T \geq 1$, we have
\begin{equation*}
\# \left( \Lambda \cap \mathcal{B}_N(T) \right) \ll
\sum_{i = 0}^R \frac{T^{R-i}}{\lambda_1(\Lambda) \cdots \lambda_{R-i}(\Lambda)}.
\end{equation*}
In particular, if $i_0 \in \{1, \dots, R\}$ then for $T \geq \lambda_{i_0}(\Lambda)$, we have
\begin{equation*}
\# \left( \Lambda \cap \mathcal{B}_N(T) \right) \ll
\frac{T^R}{\lambda_1(\Lambda) \cdots \lambda_{i_0-1}(\Lambda) \lambda_{i_0}(\Lambda)^{R-i_0+1}}.
\end{equation*}
Moreover, the implied constants depend at most on $R$. 
\end{lemma}

\begin{proof}
Appealing to a result of Schmidt \cite[Lemma~$2$]{MR224562} and using the estimates \eqref{Estimates Minkowski} we immediately obtain the first upper bound. Since $\lambda_i \geq \lambda_{i_0}$ for any $i \in \{i_0, \dots, R\}$ we deduce
\begin{equation*}
\# \left( \Lambda \cap \mathcal{B}_N(T) \right) \ll \sum_{i = 0}^{R-i_0}
\frac{T^{R-i}}{\lambda_1(\Lambda) \cdots \lambda_{i_0-1}(\Lambda) \lambda_{i_0}(\Lambda)^{R-i-i_0+1}}
+ \sum_{i = R-i_0+1}^R \frac{T^{R-i}}{\lambda_1(\Lambda) \cdots \lambda_{R-i}(\Lambda)}.
\end{equation*}
The assumption $T \geq \lambda_{i_0}(\Lambda)$ thus completes the proof.
\end{proof}

Lemma~\ref{Lemma lattice pivotal} is satisfactory when the successive minima of the lattice $\Lambda$ are small. Alternatively, if the successive minima of $\Lambda$ are known to be large then we can use
Lemma~\ref{Lemma Schmidt upper bound} to derive a useful upper bound for the cardinality of the set
$\Lambda \cap \mathcal{B}_N(T)$.

\begin{lemma}
\label{Lemma R0}
Let $N \geq 1$ and $R \in \{1, \dots, N\}$. Let $\Lambda \subset \mathbb{R}^N$ be a lattice of rank $R$. Let
$M > 0$ be such that $M < \lambda_1(\Lambda)$ and let $Y \geq \lambda_R(\Lambda)$. For any
$R_0 \in \{0, \dots, R-1\}$ and $T \leq Y$, we have
\begin{equation*}
\# \left( (\Lambda \smallsetminus \{ \boldsymbol{0} \} ) \cap \mathcal{B}_N(T) \right) \ll
\frac{T^{R-R_0} Y^{R_0}}{\det (\Lambda)} + \left(\frac{T}{M}\right)^{R-R_0-1}.
\end{equation*}
Further, let $j_0 \in \{0, \dots, R-1\}$ and $J \geq M$ be such that $J < \lambda_{j_0+1}(\Lambda)$. For any
$R_0 \in \{0, \dots, R-1-j_0\}$ and $T \leq Y$, we have
\begin{equation*}
\# \left( (\Lambda \smallsetminus \{ \boldsymbol{0} \} ) \cap \mathcal{B}_N(T) \right) \ll
\frac{T^{R-R_0} Y^{R_0}}{\det (\Lambda)} +
\left(\frac{T}{M}\right)^{j_0} \left( \left(\frac{T}{J} \right)^{R-R_0-1-j_0} + 1 \right).
\end{equation*}
Moreover, the implied constants depend at most on $R$. 
\end{lemma}

\begin{proof}
By assumption if $T \leq M$ then $(\Lambda \smallsetminus \{ \boldsymbol{0} \} ) \cap \mathcal{B}_N(T) = \emptyset$ so both upper bounds trivially hold and we can assume that $T > M$. It is now clear that the first upper bound follows from the second by taking $j_0=0$ and $J=M$. The estimate \eqref{Estimates Minkowski} and
Lemma~\ref{Lemma Schmidt upper bound} imply that 
\begin{equation*}
\# \left( \Lambda \cap \mathcal{B}_N(T) \right) \ll
\frac{T^R}{\det(\Lambda)} \sum_{i=0}^{R_0} \frac{\lambda_{R-i+1}(\Lambda) \cdots \lambda_R(\Lambda)}{T^i} +
\sum_{i=R_0+1}^R \frac{T^{R-i}}{\lambda_1(\Lambda) \cdots \lambda_{R-i}(\Lambda)}.
\end{equation*}
Since by assumption we have $\lambda_i(\Lambda) \leq Y$ for any $i \in \{ 1, \dots, R\}$ and $T \leq Y$, we see that
\begin{equation*}
\sum_{i=0}^{R_0} \frac{\lambda_{R-i+1}(\Lambda) \cdots \lambda_R(\Lambda)}{T^i} \ll \left( \frac{Y}{T} \right)^{R_0}.
\end{equation*}
Moreover, our assumptions also imply that $\lambda_i(\Lambda) > M$ for any $i \in \{ 1, \dots, j_0\}$ and
$\lambda_i(\Lambda) > J$ for any $i \in \{ j_0+1, \dots, R\}$. We thus have
\begin{equation*}
\sum_{i=R_0+1}^R \frac{T^{R-i}}{\lambda_1(\Lambda) \cdots \lambda_{R-i}(\Lambda)} \ll 
\sum_{i=R_0+1}^{R-1-j_0} \frac{T^{R-i}}{M^{j_0} J^{R-i-j_0}} +
\sum_{i=R-j_0}^R \frac{T^{R-i}}{M^{R-i}}.
\end{equation*}
Since $T > M$ we obtain
\begin{equation*}
\# \left( \Lambda \cap \mathcal{B}_N(T) \right) \ll \frac{T^{R-R_0} Y^{R_0}}{\det (\Lambda)} +
\left( \frac{J}{M} \right)^{j_0} \ \sum_{i=R_0+1}^{R-1-j_0} \left( \frac{T}{J} \right)^{R-i} +
\left( \frac{T}{M} \right)^{j_0}.
\end{equation*}
If $R_0 = R-1-j_0$ then the summation in the right-hand side is empty so in this case the proof is complete. If
$R_0 \leq R-2-j_0$ we get
\begin{align*}
\# \left( \Lambda \cap \mathcal{B}_N(T) \right) & \ll \frac{T^{R-R_0} Y^{R_0}}{\det (\Lambda)} +
\left( \frac{J}{M} \right)^{j_0} \left( \left( \frac{T}{J} \right)^{R-R_0-1} \! + \left( \frac{T}{J} \right)^{j_0+1} \right) +
\left( \frac{T}{M} \right)^{j_0} \\
& \ll \frac{T^{R-R_0} Y^{R_0}}{\det (\Lambda)} +
\left( \frac{T}{M} \right)^{j_0} \left( \left(\frac{T}{J} \right)^{R-R_0-1-j_0} + \frac{T}{J} + 1 \right),
\end{align*}
which completes the proof on noting that $R-R_0-1-j_0 \geq 1$.
\end{proof}

\subsection{The determinant of certain lattices}

\label{Section determinants}

In this section we establish formulae for the determinants of several lattices. The following notation will be very useful.

\begin{definition}
\label{Definition G}
Let $N \geq 1$ and $k \in \{1, \dots, N\}$. Given linearly independent vectors
$\mathbf{c}_1, \dots, \mathbf{c}_k \in \mathbb{Z}^N$ we let $\mathcal{G}(\mathbf{c}_1, \dots, \mathbf{c}_k)$ denote the greatest common divisor of the $k \times k$ minors of the $N \times k$ matrix whose columns are the vectors
$\mathbf{c}_1, \dots, \mathbf{c}_k$.
\end{definition}

Recall the respective definitions \eqref{Definition lattice} and \eqref{Definition local lattice} of the lattices
$\Lambda_{\mathbf{c}}$ and $\Lambda_{\mathbf{c}}^{(Q)}$, for given $\mathbf{c} \in \mathbb{Z}^N$ and $Q \geq 1$. The following lemma can be found in work of the second author \cite[Lemma~$4$]{RandomFano}.

\begin{lemma}
\label{Lemma determinant global}
Let $N \geq 1$ and $k \in \{1, \dots, N-1\}$. Let also $\mathbf{c}_1, \dots, \mathbf{c}_k \in \mathbb{Z}^N$ be linearly independent vectors. We have
\begin{equation*}
\det \left(\Lambda_{\mathbf{c}_1} \cap \cdots \cap \Lambda_{\mathbf{c}_k} \right) =
\frac{\det \left( \mathbb{Z} \mathbf{c}_1 \oplus \cdots \oplus \mathbb{Z} \mathbf{c}_k \right)}
{\mathcal{G}(\mathbf{c}_1, \dots, \mathbf{c}_k)}.
\end{equation*}
\end{lemma}

The next result provides us with formulae for the determinants of two further lattices involved in our work.

\begin{lemma}
\label{Lemma determinant mixed/local}
Let $N \geq 2$ and $Q \geq 1$. Let $\mathbf{c}, \mathbf{d} \in \mathbb{Z}_{\mathrm{prim}}^N$ be two linearly independent vectors. We have
\begin{equation*}
\det \left( \Lambda_{\mathbf{c}}^{(Q)} \cap \Lambda_{\mathbf{d}}^{(Q)} \right) = 
\frac{Q^2}{\gcd(\mathcal{G}(\mathbf{c}, \mathbf{d}), Q)},
\end{equation*}
and
\begin{equation*}
\det \left( \Lambda_{\mathbf{c}} \cap \Lambda_{\mathbf{d}}^{(Q)} \right) =
||\mathbf{c}|| \cdot \frac{Q}{\gcd(\mathcal{G}(\mathbf{c}, \mathbf{d}), Q)}.
\end{equation*}
\end{lemma}

\begin{proof}
Since $\mathbf{c}, \mathbf{d} \in \mathbb{Z}_{\mathrm{prim}}^N$ are linearly independent, the Smith normal form theorem implies that there exist $f \in \mathbb{Z}$ and $\mathbf{T} \in \GL_N(\mathbb{Z})$ such that
\begin{equation}
\label{Equality Smith}
\begin{pmatrix}
\mathbf{c}^T \\
\mathbf{d}^T
\end{pmatrix}
\mathbf{T} =
\begin{pmatrix}
1 & 0 & 0 & \cdots & 0 \\
f & \mathcal{G}(\mathbf{c}, \mathbf{d}) & 0 & \cdots & 0 
\end{pmatrix}.
\end{equation}
By definition we have
\begin{equation*}
\Lambda_{\mathbf{c}}^{(Q)} \cap \Lambda_{\mathbf{d}}^{(Q)} =
\left\{ \mathbf{y} \in \mathbb{Z}^N :
\begin{pmatrix}
\mathbf{c}^T \\
\mathbf{d}^T
\end{pmatrix} \mathbf{y} \equiv \boldsymbol{0} \bmod{Q} \right\}.
\end{equation*}
We thus deduce that
\begin{equation}
\label{Equality after Smith}
\Lambda_{\mathbf{c}}^{(Q)} \cap \Lambda_{\mathbf{d}}^{(Q)} =
\mathbf{T} \cdot \left\{ \mathbf{z} \in \mathbb{Z}^N :
\begin{pmatrix}
1 & 0 & 0 & \cdots & 0 \\
f & \mathcal{G}(\mathbf{c}, \mathbf{d}) & 0 & \cdots & 0 
\end{pmatrix}
\mathbf{z} \equiv \boldsymbol{0} \bmod{Q} \right\},
\end{equation}
and the first part of the lemma follows since $|\det (\mathbf{T})|=1$.

Furthermore, we have
\begin{equation*}
\Lambda_{\mathbf{c}} \cap \Lambda_{\mathbf{d}}^{(Q)} =
\left\{ \mathbf{y} \in \mathbb{Z}^N :
\begin{pmatrix}
\mathbf{c}^T \\
\mathbf{d}^T
\end{pmatrix} \mathbf{y} \in
\begin{pmatrix}
0 \\
Q \mathbb{Z}
\end{pmatrix}
\right\},
\end{equation*}
so we see that
\begin{equation*}
\Lambda_{\mathbf{c}} \cap \Lambda_{\mathbf{d}}^{(Q)} =
\mathbf{T} \cdot
\left\{ \mathbf{z} \in \mathbb{Z}^N :
\begin{pmatrix}
1 & 0 & 0 & \cdots & 0 \\
f & \mathcal{G}(\mathbf{c}, \mathbf{d}) & 0 & \cdots & 0 
\end{pmatrix} \mathbf{z} \in
\begin{pmatrix}
0 \\
Q \mathbb{Z}
\end{pmatrix}
\right\}.
\end{equation*}
Therefore, a basis of the lattice $\Lambda_{\mathbf{c}} \cap \Lambda_{\mathbf{d}}^{(Q)}$ is given by
\begin{equation*}
\left( \frac{Q}{\gcd(\mathcal{G}(\mathbf{c}, \mathbf{d}), Q)} \cdot \mathbf{T} \mathbf{e}_2, \mathbf{T} \mathbf{e}_3, \dots, \mathbf{T} \mathbf{e}_N \right),
\end{equation*}
where $(\mathbf{e}_1, \dots, \mathbf{e}_N)$ denotes the canonical basis of $\mathbb{R}^N$. Using the definition \eqref{Definition det}, we obtain
\begin{equation*}
\det \left(\Lambda_{\mathbf{c}} \cap \Lambda_{\mathbf{d}}^{(Q)} \right) =
\frac{Q}{\gcd(\mathcal{G}(\mathbf{c}, \mathbf{d}), Q)} \cdot \sqrt{\det \left( \mathbf{S} \right)},
\end{equation*}
where
\begin{equation*}
\mathbf{S} = \left( \mathbf{e}_i^T \mathbf{T}^T \mathbf{T} \mathbf{e}_j \right)_{i,j = 2, \dots, N}.
\end{equation*}
The matrix $\mathbf{S}$ is formed by removing from $\mathbf{T}^T \mathbf{T}$ its first line and its first column. Therefore, $\det (\mathbf{S})$ is the cofactor of index $(1,1)$ of $\mathbf{T}^T \mathbf{T}$ and is thus equal to the entry of index $(1,1)$ of the matrix $(\mathbf{T}^T \mathbf{T})^{-1}$ since $\det(\mathbf{T}^T \mathbf{T})=1$. In other words, we have
\begin{equation*}
\det (\mathbf{S}) = \mathbf{e}_1^T \left( \mathbf{T}^T \mathbf{T} \right)^{-1} \mathbf{e}_1.
\end{equation*}
But the equality
\eqref{Equality Smith} gives $\mathbf{c}^T = \mathbf{e}_1^T \mathbf{T}^{-1}$, so that
$\det \left( \mathbf{S} \right) = ||\mathbf{c}||^2$, which completes the proof.
\end{proof}

\subsection{Bounding the successive minima of the key lattices}

\label{Section successive minima}

Recall that for $d,n \geq 1$ the Veronese embedding $\nu_{d,n} : \mathbb{R}^{n+1} \to \mathbb{R}^{N_{d,n}}$ was introduced in Definition~\ref{Definition Veronese}. Given two linearly independent vectors
$\mathbf{x}, \mathbf{y} \in \mathbb{Z}_{\mathrm{prim}}^{n+1}$, the lattices
$\Lambda_{\nu_{d,n}(\mathbf{x})}$ and $\Lambda_{\nu_{d,n}(\mathbf{x})} \cap \Lambda_{\nu_{d,n}(\mathbf{y})}$ respectively have rank $N_{d,n}-1$ and $N_{d,n}-2$. These two lattices play a pivotal role in our arguments and this section is concerned with bounding the size of their successive minima, thereby aligning us for an efficient application of Lemma~\ref{Lemma lattice pivotal}. It is convenient to introduce the following notation.

\begin{definition}
\label{Definition Polynomials}
Let $d,n \geq 1$ and let $\mathbf{X} = (X_0, \dots, X_n)$. We let $\mathbb{R}[\mathbf{X}]^{(d)}$ denote the vector space of homogeneous polynomials of degree $d$ in $n+1$ variables and we let
$\omega_d : \mathbb{R}[\mathbf{X}]^{(d)} \to \mathbb{R}^{N_{d,n}}$ be the isomorphism defined using the lexicographical ordering. We also let $\mathscr{M}_{d,n}$ denote the set of monomials of degree $d$ in $n+1$ variables.
\end{definition}

We start by stating an elementary result that we will use repeatedly in our arguments.

\begin{lemma}
\label{Lemma G}
Let $d,n \geq 1$. Let $\mathbf{x}, \mathbf{y} \in \mathbb{Z}_{\mathrm{prim}}^{n+1}$ be two linearly independent vectors. Then 
\begin{equation*}
\mathcal{G}(\nu_{d,n}(\mathbf{x}), \nu_{d,n}(\mathbf{y})) = \mathcal{G}(\mathbf{x}, \mathbf{y}).
\end{equation*}
\end{lemma}

\begin{proof}
First, it is not hard to check that
$\mathcal{G}(\mathbf{x}, \mathbf{y}) \mid \mathcal{G}(\nu_{d,n}(\mathbf{x}), \nu_{d,n}(\mathbf{y}))$. Indeed, if $q$ is a positive integer such that $x_i y_j \equiv x_j y_i \bmod{q}$ for any $i, j \in \{0, \dots, n\}$, then we clearly have
$P(\mathbf{x}) Q(\mathbf{y}) \equiv P(\mathbf{y}) Q(\mathbf{x}) \bmod{q}$ for any $P, Q \in \mathscr{M}_{d,n}$. Second, let $p$ be a prime divisor of $\mathcal{G}(\nu_{d,n}(\mathbf{x}), \nu_{d,n}(\mathbf{y}))$ and let us show that the $p$-adic valuation of $\mathcal{G}(\nu_{d,n}(\mathbf{x}), \nu_{d,n}(\mathbf{y}))$ is at most the $p$-adic valuation of $x_k y_{\ell} - x_{\ell} y_k$ for any $k, \ell \in \{0, \dots, n\}$. By definition of
$\mathcal{G}(\nu_{d,n}(\mathbf{x}), \nu_{d,n}(\mathbf{y}))$ we have
\begin{equation*}
\mathcal{G}(\nu_{d,n}(\mathbf{x}), \nu_{d,n}(\mathbf{y})) \mid x_j^{d-1} y_j^{d-1} (x_k y_{\ell} - x_{\ell} y_k),
\end{equation*}
for any $j, k, \ell \in \{0, \dots, n\}$. Therefore, it suffices to check that there exists $j \in \{0, \dots, n\}$ such that
$p \nmid x_j y_j$. Otherwise, since $\mathbf{x}$ and $\mathbf{y}$ are primitive vectors, there would exist distinct indices $j_0, j_1 \in \{0, \dots, n\}$ such that $p \nmid x_{j_1} y_{j_0}$ and $p \mid x_{j_0}$. It would follow that
$p \nmid x_{j_0}^d y_{j_1}^d- x_{j_1}^d y_{j_0}^d$, which would contradict the fact that
$p \mid \mathcal{G}(\nu_{d,n}(\mathbf{x}), \nu_{d,n}(\mathbf{y}))$. This completes the proof.
\end{proof}

Our work will make crucial use of the following notion.

\begin{definition}
\label{Definition d}
Let $n \geq 1$ and let $\mathbf{x}, \mathbf{y} \in \mathbb{Z}^{n+1}$ be two linearly independent vectors. For any
$r \in \{2, \dots, n+1\}$ we let $\mathfrak{d}_r(\mathbf{x})$ be the minimum determinant of a rank $r$ sublattice of
$\mathbb{Z}^{n+1}$ containing $\mathbf{x}$, and we let $\mathfrak{d}_r(\mathbf{x},\mathbf{y})$ be the minimum determinant of a rank $r$ sublattice of $\mathbb{Z}^{n+1}$ containing $\mathbf{x}$ and $\mathbf{y}$.
\end{definition}

Let $\mathcal{L}(\mathbf{x}, \mathbf{y})$ denote the unique primitive lattice of $\mathbb{Z}^{n+1}$ of rank $2$ containing $\mathbf{x}$ and $\mathbf{y}$, that is
\begin{equation}
\label{Definition lattice L}
\mathcal{L}(\mathbf{x}, \mathbf{y}) =
\left( \mathbb{R} \mathbf{x} \oplus \mathbb{R} \mathbf{y} \right) \cap \mathbb{Z}^{n+1}.
\end{equation}
By definition we have
\begin{equation}
\label{Equality det L}
\det (\mathcal{L}(\mathbf{x}, \mathbf{y})) = \mathfrak{d}_2(\mathbf{x},\mathbf{y}).
\end{equation}
The next result provides us with a formula for this determinant.

\begin{lemma}
\label{Lemma d2(x,y)}
Let $n \geq 1$ and let $\mathbf{x}, \mathbf{y} \in \mathbb{Z}^{n+1}$ be two linearly independent vectors. We have
\begin{equation*}
\mathfrak{d}_2(\mathbf{x},\mathbf{y}) =
\frac{\left(||\mathbf{x}||^2 ||\mathbf{y}||^2 - \langle \mathbf{x}, \mathbf{y} \rangle^2 \right)^{1/2}}
{\mathcal{G}(\mathbf{x},\mathbf{y})}.
\end{equation*}
\end{lemma}

\begin{proof}
We note that $\mathcal{L}(\mathbf{x}, \mathbf{y})^{\perp} = \Lambda_\mathbf{x} \cap \Lambda_\mathbf{y}$. Therefore, since the lattice $\mathcal{L}(\mathbf{x}, \mathbf{y})$ is primitive, the equality \eqref{Det orthogonal} gives
$\mathfrak{d}_2(\mathbf{x},\mathbf{y}) = \det \left(\Lambda_\mathbf{x}\cap \Lambda_\mathbf{y}\right)$. An application of Lemma~\ref{Lemma determinant global} thus yields
\begin{equation*}
\mathfrak{d}_2(\mathbf{x},\mathbf{y}) =
\frac{\det(\mathbb{Z} \mathbf{x} \oplus \mathbb{Z} \mathbf{y})}{\mathcal{G}(\mathbf{x},\mathbf{y})}.
\end{equation*}
It follows from the definition \eqref{Definition det} of the determinant of a lattice that
\begin{equation*}
\det(\mathbb{Z} \mathbf{x} \oplus \mathbb{Z} \mathbf{y})^2 = ||\mathbf{x}||^2 ||\mathbf{y}||^2 - \langle \mathbf{x}, \mathbf{y} \rangle^2,
\end{equation*}
which completes the proof.
\end{proof}

Given an integer $R \geq 1$ and a lattice $\Lambda$ of rank $R$, we recall that for $i \in \{1, \dots, R \}$ we have defined $\lambda_i(\Lambda)$ to be the $i$-th successive minimum of $\Lambda$, as stated in
Definition~\ref{Definition successive}. We are now ready to reveal the main results of this section. The following result sharpens work of the second author \cite[Lemma~$5$]{RandomFano}, a saving that is key for our application.

\begin{lemma}
\label{Key lemma one vector}
Let $d, n \geq 1$ and let $\mathbf{x} \in \mathbb{Z}_{\mathrm{prim}}^{n+1}$. We have
\begin{equation*}
\lambda_{N_{d,n}-1} \left(\Lambda_{\nu_{d,n}(\mathbf{x})}\right) \leq n \frac{||\mathbf{x}||}{\mathfrak{d}_2(\mathbf{x})}.
\end{equation*}
\end{lemma}

\begin{proof}
We start by dealing with the case $d=1$ and we note that $N_{1,n}=n+1$ and $\nu_{1,n}(\mathbf{x}) = \mathbf{x}$. We aim to apply Lemma~\ref{Lemma Banaszczyk} and we thus let $\mathbf{a} \in \Lambda_{\mathbf{x}}^{\ast}$ be a
non-zero vector. The lattice $\mathbb{Z} \mathbf{x}$ is primitive so we deduce from
Lemma~\ref{Lemma Schmidt quotient} that $\Lambda_{\mathbf{x}}^{\ast} = \mathbb{Z}^{n+1} / \mathbb{Z}\mathbf{x}$. The vector $\mathbf{a}$ can thus be written as $\mathbf{a} = \mathbf{b} + \mathbf{t}$ for some
$\mathbf{b} \in \mathbb{Z}^{n+1}$ and $\mathbf{t} \in \mathbb{R} \mathbf{x}$. Since
$\langle \mathbf{a}, \mathbf{x} \rangle = 0$ and $\mathbf{a}$ is non-zero, the vectors $\mathbf{b}$ and $\mathbf{x}$ are linearly independent and therefore the integral lattice $\mathbb{Z} \mathbf{b} \oplus \mathbb{Z} \mathbf{x}$ has rank $2$. It follows that $\mathfrak{d}_2(\mathbf{x}) \leq \det (\mathbb{Z} \mathbf{b} \oplus \mathbb{Z} \mathbf{x})$. Since $\mathbf{t} \in \mathbb{R} \mathbf{x}$ we have
$\det(\mathbb{Z} \mathbf{b} \oplus \mathbb{Z} \mathbf{x})=\det(\mathbb{Z} \mathbf{a} \oplus \mathbb{Z}\mathbf{x})$ and we thus see that $\mathfrak{d}_2(\mathbf{x}) \leq ||{\mathbf{a}}|| \cdot ||\mathbf{x}||$. This gives
\begin{equation*}
\lambda_1(\Lambda_{\mathbf{x}}^{\ast}) \geq \frac{\mathfrak{d}_2(\mathbf{x})}{||\mathbf{x}||}.
\end{equation*}
An application of Lemma~\ref{Lemma Banaszczyk} thus completes the proof in the case $d=1$.

Assume now that $d \geq 2$. By the case $d=1$, we can pick $n$ linearly independent vectors
$\mathbf{a}_1, \dots, \mathbf{a}_n$ of the lattice $\Lambda_{\mathbf{x}}$ in the ball
$\mathcal{B}_{n+1}(n ||\mathbf{x}||/\mathfrak{d}_2(\mathbf{x}))$. For any $P \in \mathscr{M}_{d-1,n}$ we let
$\Theta_P : \mathbb{R}^{n+1} \to \mathbb{R}^{N_{d,n}}$ be the linear map defined for
$\mathbf{c} \in \mathbb{R}^{n+1}$ by
\begin{equation*}
\Theta_P(\mathbf{c}) = \omega_d \left( P(\mathbf{X}) \cdot \langle \mathbf{c}, \mathbf{X} \rangle \right),
\end{equation*}
and we note that by definition
$\langle \Theta_P(\mathbf{c}), \nu_{d,n}(\mathbf{X}) \rangle=P(\mathbf{X}) \cdot \langle \mathbf{c}, \mathbf{X}\rangle$. For any $i \in \{1, \dots, n\}$ we have $\langle \mathbf{a}_i, \mathbf{x} \rangle=0$ and it follows that
$\Theta_P(\mathbf{a}_i) \in \Lambda_{\nu_{d,n}(\mathbf{x})}$. Also, we see that for any
$\mathbf{c} \in \mathbb{R}^{n+1}$ we have $||\Theta_P(\mathbf{c})|| = ||\mathbf{c}||$ and thus
$\Theta_P(\mathbf{a}_i) \in \mathcal{B}_{N_{d,n}}( n ||\mathbf{x}||/\mathfrak{d}_2(\mathbf{x}))$ for any
$i \in \{1, \dots, n\}$. Therefore, in order to complete the proof it suffices to check that there are $N_{d,n}-1$ linearly independent vectors in the set
\begin{equation*}
\left\{ \Theta_P(\mathbf{a}_i) : 
\begin{array}{l l}
P \in \mathscr{M}_{d-1,n} \\
i \in \{1, \dots, n\}
\end{array}
\right\}.
\end{equation*}
But the real subspace spanned in $\mathbb{R}^{N_{d,n}}$ by this set of vectors contains in particular the image under
$\omega_d$ of the set
\begin{equation*}
\Span_{\mathbb{R}} \left( \left\{ P(\mathbf{X}) \cdot (x_i X_j - x_j X_i) :
\begin{array}{l l}
P \in \mathscr{M}_{d-1,n} \\
i, j \in \{0, \dots, n\}
\end{array}
\right\} \right)
= \left\{ Q \in \mathbb{R}[\mathbf{X}]^{(d)} : Q (\mathbf{x}) = 0 \right\}.
\end{equation*}
Note that we have used the fact that
$\Lambda_{\mathbf{x}} = \mathbb{R} \mathbf{a}_1 \oplus \cdots \oplus \mathbb{R} \mathbf{a}_n$. This implies that
\begin{equation*}
\Span_{\mathbb{R}} \left( \left\{ \Theta_P(\mathbf{a}_i) : 
\begin{array}{l l}
P \in \mathscr{M}_{d-1,n} \\
i \in \{1, \dots, n\}
\end{array}
\right\} \right)
= \left(\mathbb{R} \nu_{d,n}(\mathbf{x})\right)^{\perp},
\end{equation*}
which finishes the proof.
\end{proof}

The proof of our second result follows a similar strategy, but is much more challenging. 

\begin{lemma}
\label{Key lemma two vectors}
Let $d, n \geq 2$ and let $\mathbf{x}, \mathbf{y} \in \mathbb{Z}_{\mathrm{prim}}^{n+1}$ be two linearly independent vectors. We have
\begin{equation*}
\lambda_{N_{d,n}-2} \left( \Lambda_{\nu_{d,n}(\mathbf{x})} \cap \Lambda_{\nu_{d,n}(\mathbf{y})} \right) \leq
3n^2 \max \left\{ \frac{\mathfrak{d}_2(\mathbf{x},\mathbf{y})}{\mathfrak{d}_3(\mathbf{x},\mathbf{y})}, 
\frac{||\mathbf{x}|| \cdot ||\mathbf{y}||}{\mathfrak{d}_2(\mathbf{x},\mathbf{y})^2} \right\}.
\end{equation*}
\end{lemma}

Given two linearly independent vectors $\mathbf{x}, \mathbf{y} \in \mathbb{Z}_{\mathrm{prim}}^{n+1}$, the proof of 
Lemma~\ref{Key lemma two vectors} depends on a close analysis of the primitive lattices
\begin{equation*}
\mathcal{Q}_2(\mathbf{x}, \mathbf{y}) =
\left( \mathbb{R} \nu_{2,n}(\mathbf{x}) \oplus \mathbb{R} \nu_{2,n}(\mathbf{y}) \right) \cap \mathbb{Z}^{N_{2,n}},
\end{equation*}
and
\begin{equation*}
\mathcal{Q}_3(\mathbf{x}, \mathbf{y}) =
\left( \mathbb{R} \nu_{2,n}(\mathbf{x}) \oplus \mathbb{R} \nu_{2,n}(\mathbf{y}) \oplus \mathbb{R} \nu_{2,n}(\mathbf{x} + \mathbf{y}) \right) \cap \mathbb{Z}^{N_{2,n}}.
\end{equation*}
We note that for given $i, j \in \{0, \dots, n\}$, a straightforward calculation yields
\begin{equation}
\label{Equality linear independence}
\det
\begin{pmatrix}
x_i^2 & y_i^2 & (x_i+y_i)^2 \\
x_i x_j & y_i y_j & (x_i+y_i)(x_j+y_j) \\
x_j^2 & y_j^2 & (x_j + y_j)^2 
\end{pmatrix}
= (x_j y_i -x_i y_j)^3.
\end{equation}
If the vectors $\mathbf{x}$ and $\mathbf{y}$ are linearly independent then there exist $i_0, j_0 \in \{0, \dots, n\}$ such that $x_{j_0} y_{i_0} - x_{i_0} y_{j_0} \neq 0$ so we deduce that the vectors $\nu_{2,n}(\mathbf{x})$,
$\nu_{2,n}(\mathbf{y})$ and $\nu_{2,n}(\mathbf{x} + \mathbf{y})$ are linearly independent. It follows that the lattices
$\mathcal{Q}_2(\mathbf{x}, \mathbf{y})$ and $\mathcal{Q}_3(\mathbf{x}, \mathbf{y})$ respectively have rank $2$ and $3$ and thus the quotient lattice
$\mathcal{Q}_3(\mathbf{x}, \mathbf{y}) / \mathcal{Q}_2(\mathbf{x}, \mathbf{y})$ has rank $1$. The following result is concerned with the determinant of this lattice.

\begin{lemma}
\label{Lemma det quotient lattice}
Let $n \geq 2$ and let $\mathbf{x}, \mathbf{y} \in \mathbb{Z}_{\mathrm{prim}}^{n+1}$ be two linearly independent vectors. Then 
\begin{equation*}
\det \left( \mathcal{Q}_3(\mathbf{x}, \mathbf{y}) / \mathcal{Q}_2(\mathbf{x}, \mathbf{y}) \right) \geq \frac1{3} \cdot \frac{\mathfrak{d}_2(\mathbf{x},\mathbf{y})^2}{||\mathbf{x}|| \cdot ||\mathbf{y}||}.
\end{equation*}
\end{lemma}

\begin{proof}
In order to achieve our goal, we employ the identity \eqref{Det quotient torsion free}, giving 
\begin{equation}
\label{Equality det quotient}
\det \left( \mathcal{Q}_3(\mathbf{x}, \mathbf{y}) / \mathcal{Q}_2(\mathbf{x}, \mathbf{y})\right) =
\frac{\det \left( \mathcal{Q}_3(\mathbf{x}, \mathbf{y}) \right)}{\det \left( \mathcal{Q}_2(\mathbf{x}, \mathbf{y}) \right)}.
\end{equation}

We start by proving an upper bound for the determinant of $\mathcal{Q}_2(\mathbf{x}, \mathbf{y})$. Since the lattice
$\mathcal{Q}_2(\mathbf{x}, \mathbf{y})$ is primitive, we see that we have
\begin{equation}
\label{Equality Q2}
\mathcal{Q}_2(\mathbf{x}, \mathbf{y}) =
\left( \Lambda_{\nu_{2,n}(\mathbf{x})} \cap \Lambda_{\nu_{2,n}(\mathbf{y})} \right)^{\perp}.
\end{equation}
It follows from the equality \eqref{Det orthogonal} that
\begin{equation*}
\det(\mathcal{Q}_2(\mathbf{x}, \mathbf{y})) =
\det(\Lambda_{\nu_{2,n}(\mathbf{x})} \cap \Lambda_{\nu_{2,n}(\mathbf{y})}),
\end{equation*}
and Lemma~\ref{Lemma determinant global} thus gives
\begin{equation}
\label{Equality det Q2}
\det \left( \mathcal{Q}_2(\mathbf{x}, \mathbf{y}) \right) =
\frac{\det ( \mathbb{Z} \nu_{2,n}(\mathbf{x}) \oplus \mathbb{Z} \nu_{2,n}(\mathbf{y}))}
{\mathcal{G}(\nu_{2,n}(\mathbf{x}),\nu_{2,n}(\mathbf{y}))}.
\end{equation}
Recalling the definition \eqref{Definition det} of the determinant of a lattice, we see that
\begin{equation*}
\det( \mathbb{Z} \nu_{2,n}(\mathbf{x}) \oplus \mathbb{Z} \nu_{2,n}(\mathbf{y}) )^2 =
||\nu_{2,n}(\mathbf{x})||^2 ||\nu_{2,n}(\mathbf{y})||^2 - \langle \nu_{2,n}(\mathbf{x}), \nu_{2,n}(\mathbf{y}) \rangle^2.
\end{equation*}
This can be rewritten as
\begin{equation*}
\det( \mathbb{Z} \nu_{2,n}(\mathbf{x}) \oplus \mathbb{Z} \nu_{2,n}(\mathbf{y}) )^2 =
\frac1{2} \sum_{P_1,P_2 \in \mathscr{M}_{2,n}} \left( P_1(\mathbf{x}) P_2(\mathbf{y}) - P_2(\mathbf{x}) P_1(\mathbf{y}) \right)^2,
\end{equation*}
which implies in particular that
\begin{equation*}
\det (\mathbb{Z} \nu_{2,n}(\mathbf{x}) \oplus \mathbb{Z} \nu_{2,n}(\mathbf{y}))^2 \leq
\frac1{2} \sum_{i_1, j_1, i_2, j_2 = 0}^n \left(x_{i_1} x_{j_1} y_{i_2} y_{j_2} - x_{i_2} x_{j_2} y_{i_1} y_{j_1} \right)^2.
\end{equation*}
Expanding the square, this leads to
\begin{equation*}
\det (\mathbb{Z} \nu_{2,n}(\mathbf{x}) \oplus \mathbb{Z} \nu_{2,n}(\mathbf{y}))^2 \leq 
||\mathbf{x}||^4||\mathbf{y}||^4 - \langle \mathbf{x}, \mathbf{y} \rangle^4.
\end{equation*}
Recalling the equality \eqref{Equality det Q2} and using Lemma~\ref{Lemma G}, we deduce
\begin{equation}
\label{Upper bound det Q2}
\det \left( \mathcal{Q}_2(\mathbf{x}, \mathbf{y}) \right) \leq
\frac{\left( ||\mathbf{x}||^4||\mathbf{y}||^4 - \langle \mathbf{x}, \mathbf{y} \rangle^4 \right)^{1/2}}{\mathcal{G}(\mathbf{x}, \mathbf{y})}.
\end{equation}

We now follow a similar approach to prove a lower bound for the determinant of
$\mathcal{Q}_3(\mathbf{x}, \mathbf{y})$. Since the lattice $\mathcal{Q}_3(\mathbf{x}, \mathbf{y})$ is primitive we have
\begin{equation*}
\mathcal{Q}_3(\mathbf{x}, \mathbf{y}) =
\left( \Lambda_{\nu_{2,n}(\mathbf{x})} \cap \Lambda_{\nu_{2,n}(\mathbf{y})}
\cap \Lambda_{\nu_{2,n}(\mathbf{x} + \mathbf{y})} \right)^{\perp}.
\end{equation*}
The equality \eqref{Det orthogonal} thus gives
\begin{equation*}
\det(\mathcal{Q}_3(\mathbf{x}, \mathbf{y})) = \det( \Lambda_{\nu_{2,n}(\mathbf{x})} \cap \Lambda_{\nu_{2,n}(\mathbf{y})} \cap \Lambda_{\nu_{2,n}(\mathbf{x} + \mathbf{y})}).
\end{equation*}
Using Lemma~\ref{Lemma determinant global} we deduce
\begin{equation}
\label{Equality det Q3}
\det \left( \mathcal{Q}_3(\mathbf{x}, \mathbf{y}) \right) =
\frac{\det ( \mathbb{Z} \nu_{2,n}(\mathbf{x}) \oplus \mathbb{Z} \nu_{2,n}(\mathbf{y}) \oplus \mathbb{Z} \nu_{2,n}(\mathbf{x}+\mathbf{y}))}{\mathcal{G}(\nu_{2,n}(\mathbf{x}),\nu_{2,n}(\mathbf{y}),\nu_{2,n}(\mathbf{x}+\mathbf{y}))}.
\end{equation}
Applying the definition \eqref{Definition det}, we see that the square of the determinant of the lattice
$\mathbb{Z} \nu_{2,n}(\mathbf{x}) \oplus \mathbb{Z} \nu_{2,n}(\mathbf{y}) \oplus \mathbb{Z} \nu_{2,n}(\mathbf{x}+\mathbf{y})$ is equal to
\begin{equation*}
\det \begin{pmatrix}
||\nu_{2,n}(\mathbf{x})||^2 & \langle \nu_{2,n}(\mathbf{x}), \nu_{2,n}(\mathbf{y}) \rangle& \langle \nu_{2,n}(\mathbf{x}), \nu_{2,n}(\mathbf{x}+\mathbf{y}) \rangle \\
\langle \nu_{2,n}(\mathbf{x}), \nu_{2,n}(\mathbf{y}) \rangle & ||\nu_{2,n}(\mathbf{y})||^2 & \langle \nu_{2,n}(\mathbf{y}), \nu_{2,n}(\mathbf{x}+\mathbf{y}) \rangle \\
\langle \nu_{2,n}(\mathbf{x}), \nu_{2,n}(\mathbf{x}+\mathbf{y}) \rangle & \langle \nu_{2,n}(\mathbf{y}), \nu_{2,n}(\mathbf{x}+\mathbf{y}) \rangle & ||\nu_{2,n}(\mathbf{x}+\mathbf{y})||^2 
\end{pmatrix}.
\end{equation*}
Letting $S_3$ denote the permutation group of the set $\{1, 2, 3\}$, the calculation of this determinant shows that it can be rewritten as
\begin{equation*}
\frac1{6} \sum_{P_1, P_2, P_3 \in \mathscr{M}_{2,n}} \left( \sum_{\sigma \in S_3} \sign(\sigma) P_{\sigma(1)}(\mathbf{x}) P_{\sigma(2)}(\mathbf{y}) P_{\sigma(3)}(\mathbf{x}+\mathbf{y}) \right)^2.
\end{equation*}
Therefore, we note that it is in particular bounded below by
\begin{equation*}
\frac1{48} \sum_{i_1, j_1, i_2, j_2, i_3, j_3=0}^n \left( \sum_{\sigma \in S_3} \sign(\sigma) x_{i_{\sigma(1)}} x_{j_{\sigma(1)}} y_{i_{\sigma(2)}} y_{j_{\sigma(2)}} (x_{i_{\sigma(3)}} + y_{i_{\sigma(3)}}) (x_{j_{\sigma(3)}} + y_{j_{\sigma(3)}}) \right)^2.
\end{equation*}
Expanding the square, a straightforward calculation shows that this quantity equals
\begin{align*}
& \frac1{8} \bigg(
||\mathbf{x}||^4 ||\mathbf{y}||^4 ||\mathbf{x} + \mathbf{y}||^4
+2 \langle \mathbf{x}, \mathbf{y} \rangle^2 \langle \mathbf{x}, \mathbf{x}+\mathbf{y} \rangle^2 \langle \mathbf{y}, \mathbf{x}+\mathbf{y} \rangle^2 \\
& - \langle \mathbf{x}, \mathbf{y} \rangle^4 ||\mathbf{x} + \mathbf{y}||^4
- \langle \mathbf{x}, \mathbf{x}+\mathbf{y} \rangle^4 ||\mathbf{y}||^4
- \langle \mathbf{y}, \mathbf{x}+\mathbf{y} \rangle^4 ||\mathbf{x}||^4 \bigg) =
\frac1{4} \left( ||\mathbf{x}||^2 ||\mathbf{y}||^2 - \langle \mathbf{x}, \mathbf{y} \rangle^2 \right)^3.
\end{align*}
We have thus obtained
\begin{equation*}
\det(\mathbb{Z} \nu_{2,n}(\mathbf{x}) \oplus \mathbb{Z} \nu_{2,n}(\mathbf{y}) \oplus \mathbb{Z} \nu_{2,n}(\mathbf{x}+\mathbf{y}))^2 \geq \frac1{4} \left( ||\mathbf{x}||^2 ||\mathbf{y}||^2 - \langle \mathbf{x}, \mathbf{y} \rangle^2 \right)^3.
\end{equation*}
In addition, the identity \eqref{Equality linear independence} shows that
\begin{equation*}
\mathcal{G}(\nu_{2,n}(\mathbf{x}),\nu_{2,n}(\mathbf{y}),\nu_{2,n}(\mathbf{x}+\mathbf{y})) \leq
\mathcal{G}( \mathbf{x}, \mathbf{y})^3.
\end{equation*}
Therefore, recalling the equality \eqref{Equality det Q3} we eventually derive
\begin{equation*}
\det \left( \mathcal{Q}_3(\mathbf{x}, \mathbf{y}) \right) \geq \frac1{2} \cdot
\frac{\left( ||\mathbf{x}||^2 ||\mathbf{y}||^2 - \langle \mathbf{x}, \mathbf{y} \rangle^2 \right)^{3/2}}
{\mathcal{G}(\mathbf{x}, \mathbf{y})^3}.
\end{equation*}

Putting together the lower bound \eqref{Equality det quotient} and the upper bound \eqref{Upper bound det Q2}, we obtain
\begin{equation*}
\det \left( \mathcal{Q}_3(\mathbf{x}, \mathbf{y}) / \mathcal{Q}_2(\mathbf{x}, \mathbf{y})\right) \geq \frac1{2} \cdot 
\frac{||\mathbf{x}||^2 ||\mathbf{y}||^2 - \langle \mathbf{x}, \mathbf{y} \rangle^2}{\left(||\mathbf{x}||^2 ||\mathbf{y}||^2 + \langle \mathbf{x}, \mathbf{y} \rangle^2\right)^{1/2}\mathcal{G}(\mathbf{x}, \mathbf{y})^2}.
\end{equation*}
Noticing that the Cauchy--Schwarz inequality gives
\begin{equation*}
\left(||\mathbf{x}||^2 ||\mathbf{y}||^2 + \langle \mathbf{x}, \mathbf{y} \rangle^2\right)^{1/2} \leq \sqrt{2} \cdot
||\mathbf{x}|| \cdot ||\mathbf{y}||
\end{equation*}
and $2 \sqrt{2} \leq 3$, we see that an application of Lemma~\ref{Lemma d2(x,y)} completes the proof.
\end{proof}

We are now ready to furnish the proof of Lemma~\ref{Key lemma two vectors}.

\begin{proof}[Proof of Lemma~\ref{Key lemma two vectors}]
We start by proving the result in the case $d=2$. Recalling the identity \eqref{Equality Q2} and aiming to apply
Lemma~\ref{Lemma Banaszczyk} we let
$\mathbf{a} \in \left(\mathcal{Q}_2(\mathbf{x}, \mathbf{y})^{\perp} \right)^{\ast}$ be a non-zero vector. The lattice
$\mathcal{Q}_2(\mathbf{x}, \mathbf{y})$ is primitive so it follows from Lemma~\ref{Lemma Schmidt quotient} that
$\left( \mathcal{Q}_2(\mathbf{x}, \mathbf{y})^{\perp} \right)^{\ast} =
\mathbb{Z}^{N_{2,n}} / \mathcal{Q}_2(\mathbf{x}, \mathbf{y})$. We thus deduce the existence of vectors
$\mathbf{b} \in \mathbb{Z}^{N_{2,n}}$ and $\mathbf{t} \in \Span_{\mathbb{R}}(\mathcal{Q}_2(\mathbf{x},\mathbf{y}))$ such that $\mathbf{a} = \mathbf{b} + \mathbf{t}$.

We are going to distinguish two cases depending on whether the vector $\mathbf{a}$ belongs to
$\Span_{\mathbb{R}}(\mathcal{Q}_3(\mathbf{x}, \mathbf{y}))$ or not. In the former case we see that
$\mathbf{b} \in \mathcal{Q}_3(\mathbf{x}, \mathbf{y})$. Therefore, if
$\pi : \mathbb{R}^{N_{2,n}} \to \Span_{\mathbb{R}}(\mathcal{Q}_2(\mathbf{x}, \mathbf{y}))^{\perp}$ denotes the orthogonal projection on $\Span_{\mathbb{R}}(\mathcal{Q}_2(\mathbf{x}, \mathbf{y}))^{\perp}$ then
$\mathbf{a} = \pi(\mathbf{a}) = \pi(\mathbf{b})$ is a non-zero vector in the quotient lattice
$\mathcal{Q}_3(\mathbf{x}, \mathbf{y})/ \mathcal{Q}_2(\mathbf{x}, \mathbf{y})$. Since this lattice has rank $1$ an application of Lemma~\ref{Lemma det quotient lattice} shows that, in the case where
$\mathbf{a} \in \Span_{\mathbb{R}}(\mathcal{Q}_3(\mathbf{x}, \mathbf{y}))$, we have
\begin{equation}
\label{Lower bound first case}
\lambda_1 \left( \left( \mathcal{Q}_2(\mathbf{x}, \mathbf{y})^{\perp} \right)^{\ast} \right) \geq \frac1{3} \cdot \frac{\mathfrak{d}_2(\mathbf{x},\mathbf{y})^2}{||\mathbf{x}|| \cdot ||\mathbf{y}||}.
\end{equation}

We now deal with the case where $\mathbf{a} \notin \Span_{\mathbb{R}} (\mathcal{Q}_3(\mathbf{x}, \mathbf{y}))$. Indexing the coordinates of $\mathbb{R}^{N_{2,n}}$ using the lexicographical ordering, we introduce a symmetric bilinear form $\sigma : \mathbb{R}^{n+1} \times \mathbb{R}^{n+1} \to \mathbb{R}^{N_{2,n}}$ by defining the coordinate of index $(i,j)$ of $\sigma(\mathbf{u},\mathbf{v})$ as being equal to
\begin{equation*}
\begin{cases}
u_i v_i, & \textrm{if } i=j, \\
u_i v_j + u_j v_i, & \textrm{if } i \neq j,
\end{cases}
\end{equation*}
for any $i \in \{0, \dots, n\}$, $j \in \{i, \dots, n\}$ and
$(\mathbf{u},\mathbf{v}) \in \mathbb{R}^{n+1} \times \mathbb{R}^{n+1}$. Recall the definition \eqref{Definition lattice L} of the lattice $\mathcal{L}(\mathbf{x}, \mathbf{y})$. Our next task is to check that
\begin{equation}
\label{Equality sigma}
\Span_{\mathbb{R}} \left(\sigma \left(\mathbb{Z}^{n+1}\times \mathcal{L}(\mathbf{x}, \mathbf{y})^{\perp} \right)\right) =
\Span_{\mathbb{R}} (\mathcal{Q}_3(\mathbf{x}, \mathbf{y}))^{\perp}.
\end{equation}
We start by noting that for any $\mathbf{z} \in \mathbb{Z}^{n+1}$, we have 
\begin{align*}
\langle \nu_{2,n}(\mathbf{z}),\sigma(\mathbf{u},\mathbf{v}) \rangle & =
\sum_{i=0}^n z_i^2 u_i v_i + \sum_{i=0}^n \sum_{j=i+1}^n z_i z_j (u_i v_j + u_j v_i) \\
& = \sum_{i=0}^n z_i u_i \sum_{j=i}^n z_j v_j + \sum_{i=0}^n \sum_{j=i+1}^n z_i z_j u_j v_i \\
& = z_0 u_0 \langle \mathbf{z},\mathbf{v} \rangle+\sum_{i=1}^n z_i u_i \left(\sum_{j=i}^n z_j v_j +\sum_{j=0}^{i-1} z_jv_j\right) \\
& = \langle \mathbf{z},\mathbf{u} \rangle \cdot \langle \mathbf{z},\mathbf{v} \rangle.
\end{align*}
It follows that $\sigma(\mathbf{u},\mathbf{v}) \in \Span_{\mathbb{R}}( \mathcal{Q}_3(\mathbf{x}, \mathbf{y}))^{\perp}$ when $(\mathbf{u},\mathbf{v}) \in \mathbb{Z}^{n+1} \times \mathcal{L}(\mathbf{x}, \mathbf{y})^{\perp}$. In addition, by examining the values of $\sigma$ at pairs of vectors of the canonical basis of $\mathbb{R}^{n+1}$, we get
\begin{equation*}
\Span_{\mathbb{R}} \left(\sigma \left( \mathbb{R}^{n+1} \times \mathbb{R}^{n+1} \right) \right)=\mathbb{R}^{N_{2,n}}.
\end{equation*}
Moreover, it is clear that any $\mathbf{v} \in \mathbb{R}^{n+1}$ can uniquely be written as
$\mathbf{v} = s \mathbf{x} + t \mathbf{y} + \mathbf{z}$ for some $s,t \in \mathbb{R}$ and
$\mathbf{z} \in \Span_{\mathbb{R}} \left(\mathcal{L}(\mathbf{x}, \mathbf{y})^{\perp} \right)$. We thus have
\begin{equation*}
\sigma \left( \mathbb{R}^{n+1}\times \mathbb{R}^{n+1} \right) =
\left\{ \Sigma_{s_1,t_1,s_2,t_2}^{(\mathbf{x}, \mathbf{y})}(\mathbf{z}_1, \mathbf{z}_2) :
\begin{array}{l l}
s_1, t_1, s_2, t_2 \in \mathbb{R} \\
\mathbf{z}_1, \mathbf{z}_2 \in \Span_{\mathbb{R}} \left(\mathcal{L}(\mathbf{x}, \mathbf{y})^{\perp} \right)
\end{array}
\right\},
\end{equation*}
where
\begin{align*}
\Sigma_{s_1,t_1,s_2,t_2}^{(\mathbf{x}, \mathbf{y})}(\mathbf{z}_1, \mathbf{z}_2) = & \ 
s_1 s_2 \sigma(\mathbf{x},\mathbf{x}) + (s_1 t_2 + s_2 t_1) \sigma(\mathbf{x},\mathbf{y})
+ t_1t_2 \sigma(\mathbf{y},\mathbf{y}) + \sigma(s_1 \mathbf{x} + t_1 \mathbf{y}, \mathbf{z}_2) \\
& + \sigma(s_2 \mathbf{x} + t_2\mathbf{y}, \mathbf{z}_1) + \sigma(\mathbf{z}_1, \mathbf{z}_2).
\end{align*}
Using the bilinearity of $\sigma$ we immediately deduce that for any $s_1, t_1, s_2, t_2 \in \mathbb{R}$ and any
$\mathbf{z}_1, \mathbf{z}_2 \in \Span_{\mathbb{R}} \left(\mathcal{L}(\mathbf{x}, \mathbf{y})^{\perp} \right)$ the vectors $\sigma(s_1\mathbf{x}+t_1\mathbf{y},\mathbf{z}_2)$, $\sigma(s_2\mathbf{x}+t_2\mathbf{y},\mathbf{z}_1)$ and
$\sigma(\mathbf{z}_1,\mathbf{z}_2)$ all belong to the vector space
$\Span_{\mathbb{R}} \left( \sigma \left( \mathbb{Z}^{n+1} \times \mathcal{L}(\mathbf{x},\mathbf{y})^{\perp} \right) \right)$. We thus derive the lower bound
\begin{equation*}
\dim \left( \Span_{\mathbb{R}} \left( \sigma \left( \mathbb{Z}^{n+1} \times \mathcal{L}(\mathbf{x},\mathbf{y})^{\perp} \right) \right) \right) \geq N_{2,n} - 3,
\end{equation*}
which finishes the proof of the identity \eqref{Equality sigma} since we have
\begin{equation*}
\dim \left( \Span_{\mathbb{R}} \left( \mathcal{Q}_3(\mathbf{x}, \mathbf{y}) \right)^{\perp} \right) = N_{2,n}-3.
\end{equation*}

Recall that we are treating the case where
$\mathbf{a} \notin \Span_{\mathbb{R}} (\mathcal{Q}_3(\mathbf{x}, \mathbf{y}))$ and let
$(\mathbf{e}_0,\dots,\mathbf{e}_n)$ be the canonical basis of $\mathbb{Z}^{n+1}$. If we had
$\langle \mathbf{a}, \sigma(\mathbf{e}_{i}, \mathbf{v})\rangle=0$ for every $i \in \{0, \dots, n \}$ and every
$\mathbf{v} \in \mathcal{L}(\mathbf{x}, \mathbf{y})^{\perp}$ then we would have
$\langle \mathbf{a}, \mathbf{w} \rangle = 0$ for any $\mathbf{w} \in
\Span_{\mathbb{R}} \left( \sigma \left( \mathbb{Z}^{n+1} \times \mathcal{L}(\mathbf{x},\mathbf{y})^{\perp} \right)\right)$ and it would follow from the equality \eqref{Equality sigma} that
$\mathbf{a} \in \Span_{\mathbb{R}} (\mathcal{Q}_3(\mathbf{x}, \mathbf{y}))$. We may thus assume that there exist
$i \in \{0, \dots, n\}$ and $\mathbf{v} \in \mathcal{L}(\mathbf{x}, \mathbf{y})^{\perp}$ such that 
$\langle \mathbf{a}, \sigma(\mathbf{e}_{i}, \mathbf{v}) \rangle \neq 0$. We now define a linear map
$r_i : \mathbb{R}^{N_{2,n}} \to \mathbb{R}^{n+1}$ by using the lexicographical ordering of coordinates in
$\mathbb{R}^{N_{2,n}}$. For any $(c_{0,0}, \dots, c_{n,n}) \in \mathbb{R}^{N_{2,n}}$, we set
\begin{equation*}
r_i(c_{0,0}, \dots, c_{n,n}) = (c_{0,i}, \dots, c_{i,i},c_{i,i+1}, \dots, c_{i,n}).
\end{equation*}
We recall that the vector $\mathbf{a}$ can be written as $\mathbf{a} = \mathbf{b} + \mathbf{t}$ for some
$\mathbf{b} \in \mathbb{Z}^{N_{2,n}}$ and
$\mathbf{t} \in \Span_{\mathbb{R}}(\mathcal{Q}_2(\mathbf{x},\mathbf{y}))$. We note that
$r_i(\nu_{2,n}(\mathbf{x})) = x_i \mathbf{x}$ and $r_i(\nu_{2,n}(\mathbf{y})) = y_i \mathbf{y}$ and thus
$r_i(\mathbf{t}) \in \Span_{\mathbb{R}}(\mathcal{L}(\mathbf{x}, \mathbf{y}))$. Since
$\mathbf{v} \in \mathcal{L}(\mathbf{x}, \mathbf{y})^{\perp}$ we deduce that
$\langle r_i(\mathbf{b}), \mathbf{v} \rangle = \langle r_i(\mathbf{a}), \mathbf{v} \rangle$. Furthermore, by definition of the map $r_i$ we see that
\begin{equation*}
\langle r_i(\mathbf{a}), \mathbf{v} \rangle = \langle \mathbf{a}, \sigma(\mathbf{e}_{i}, \mathbf{v}) \rangle \neq 0,
\end{equation*}
and it follows that the integral lattice $\mathbb{Z} r_i(\mathbf{b}) \oplus \mathcal{L}(\mathbf{x},\mathbf{y})$ has rank $3$. Since this lattice contains $\mathbf{x}$ and $\mathbf{y}$ we have
\begin{equation*}
\mathfrak{d}_3(\mathbf{x},\mathbf{y}) \leq \det \left(\mathbb{Z} r_i(\mathbf{b}) \oplus \mathcal{L}(\mathbf{x},\mathbf{y})\right). 
\end{equation*}
Using the fact that $r_i(\mathbf{t}) \in \Span_{\mathbb{R}}(\mathcal{L}(\mathbf{x}, \mathbf{y}))$ we see that
\begin{equation*}
\det \left(\mathbb{Z} r_i(\mathbf{b}) \oplus \mathcal{L}(\mathbf{x},\mathbf{y})\right) =
\det \left(\mathbb{Z} r_i(\mathbf{a}) \oplus \mathcal{L}(\mathbf{x},\mathbf{y})\right).
\end{equation*}
Recalling the equality \eqref{Equality det L} and using the upper bound $||r_i(\mathbf{a})|| \leq ||\mathbf{a}||$ we obtain
\begin{equation*}
\mathfrak{d}_3(\mathbf{x},\mathbf{y}) \leq ||\mathbf{a}|| \cdot \mathfrak{d}_2(\mathbf{x},\mathbf{y}).
\end{equation*}
This eventually shows that, in the case where $\mathbf{a} \notin \Span_{\mathbb{R}}(\mathcal{Q}_3(\mathbf{x}, \mathbf{y}))$, we have
\begin{equation}
\label{Lower bound second case}
\lambda_1 \left( \left( \mathcal{Q}_2(\mathbf{x}, \mathbf{y})^{\perp} \right)^{\ast} \right) \geq
\frac{\mathfrak{d}_3(\mathbf{x},\mathbf{y})}{\mathfrak{d}_2(\mathbf{x},\mathbf{y})}.
\end{equation}

Recalling the equality \eqref{Equality Q2} and combining the lower bounds \eqref{Lower bound first case} and \eqref{Lower bound second case} we deduce that
\begin{equation*}
\lambda_1 \left( \left( \Lambda_{\nu_{2,n}(\mathbf{x})} \cap \Lambda_{\nu_{2,n}(\mathbf{y})} \right)^{\ast} \right) \geq
\frac1{3} \cdot \min \left\{ \frac{\mathfrak{d}_3(\mathbf{x},\mathbf{y})}{\mathfrak{d}_2(\mathbf{x},\mathbf{y})}, 
\frac{\mathfrak{d}_2(\mathbf{x},\mathbf{y})^2}{||\mathbf{x}|| \cdot ||\mathbf{y}||} \right\}.
\end{equation*}
Applying Lemma~\ref{Lemma Banaszczyk} and using the fact that $N_{2,n}-2 \leq n^2$ for any $n \geq 1$ completes the proof in the case $d=2$.

Assume now that $d \geq 3$. By the case $d=2$, we can pick $N_{2,n}-2$ linearly independent vectors
$\mathbf{a}_1, \dots, \mathbf{a}_{N_{2,n}-2}$ of the lattice
$\Lambda_{\nu_{2,n}(\mathbf{x})} \cap \Lambda_{\nu_{2,n}(\mathbf{y})}$ in the ball
$\mathcal{B}_{N_{2,n}}(\mu(\mathbf{x}, \mathbf{y}))$ where $\mu(\mathbf{x}, \mathbf{y})$ denotes the right-hand side of the upper bound stated in Lemma~\ref{Key lemma two vectors}. For any $P \in \mathscr{M}_{d-2,n}$ we let $\Psi_P : \mathbb{R}^{N_{2,n}} \to \mathbb{R}^{N_{d,n}}$ be the linear map defined for
$\mathbf{c} \in \mathbb{R}^{N_{2,n}}$ by
\begin{equation*}
\Psi_P(\mathbf{c}) = \omega_d \left( P(\mathbf{X}) \cdot \langle \mathbf{c}, \nu_{2,n}(\mathbf{X}) \rangle \right),
\end{equation*}
and we note that $\langle \Psi_P(\mathbf{c}), \nu_{d,n}(\mathbf{X}) \rangle = 
P(\mathbf{X}) \cdot \langle \mathbf{c}, \nu_{2,n}(\mathbf{X}) \rangle$. For any $i \in \{1, \dots, N_{2,n}-2\}$ we have
$\langle \mathbf{a}_i, \nu_{2,n}(\mathbf{x}) \rangle = 0$ and $\langle \mathbf{a}_i, \nu_{2,n}(\mathbf{y}) \rangle = 0$ so it follows that $\Psi_P(\mathbf{a}_i) \in \Lambda_{\nu_{d,n}(\mathbf{x})} \cap \Lambda_{\nu_{d,n}(\mathbf{y})}$. Also, for any $\mathbf{c} \in \mathbb{R}^{N_{2,n}}$ we have $||\Psi_P(\mathbf{c})|| = ||\mathbf{c}||$ and thus
$\Psi_P(\mathbf{a}_i) \in \mathcal{B}_{N_{d,n}}(\mu(\mathbf{x}, \mathbf{y}))$ for any
$i \in \{1, \dots, N_{2,n}-2\}$. Therefore, in order to complete the proof it suffices to check that there are $N_{d,n}-2$ linearly independent vectors in the set
\begin{equation*}
\left\{ \Psi_P(\mathbf{a}_i) : 
\begin{array}{l l}
P \in \mathscr{M}_{d-2,n} \\
i \in \{1, \dots, N_{2,n}-2\}
\end{array}
\right\}.
\end{equation*}
The real subspace spanned in $\mathbb{R}^{N_{d,n}}$ by this set of vectors contains in particular the image under
$\omega_d$ of the set
\begin{equation*}
\Span_{\mathbb{R}} \left( \left\{ P(\mathbf{X}) \cdot L_{i,j,k}^{(\mathbf{x},\mathbf{y})}(\mathbf{X}) :
\begin{array}{l l}
P \in \mathscr{M}_{d-2,n} \\
i, j, k \in \{0, \dots, n\}
\end{array}
\right\} \right) = \left\{ Q \in \mathbb{R}[\mathbf{X}]^{(d)} :
\begin{array}{l l}
Q (\mathbf{x}) = 0 \\
Q (\mathbf{y}) = 0
\end{array}
\right\},
\end{equation*}
where
\begin{equation*}
L_{i,j,k}^{(\mathbf{x},\mathbf{y})}(\mathbf{X}) = (x_j y_k - x_k y_j) X_i + (x_k y_i - x_i y_k) X_j + (x_i y_j - x_j y_i) X_k.
\end{equation*}
Note that we have used the fact that
$\Lambda_{\nu_{2,n}(\mathbf{x})} \cap \Lambda_{\nu_{2,n}(\mathbf{y})} =
\mathbb{R}\mathbf{a}_1 \oplus \cdots \oplus \mathbb{R} \mathbf{a}_{N_{2,n}-2}$. We eventually deduce that
\begin{equation*}
\Span_{\mathbb{R}} \left( \left\{ \Psi_P(\mathbf{a}_i) : 
\begin{array}{l l}
P \in \mathscr{M}_{d-2,n} \\
i \in \{1, \dots, N_{2,n}-2\}
\end{array}
\right\} \right)
= \left(\mathbb{R} \nu_{d,n}(\mathbf{x}) \oplus \mathbb{R} \nu_{d,n}(\mathbf{y}) \right)^{\perp},
\end{equation*}
which finishes the proof.
\end{proof}

\subsection{On the typical size of some key quantities}

\label{Section typical size}

Given two linearly independent vectors $\mathbf{x}, \mathbf{y} \in \mathbb{Z}^{n+1}$, recall that the quantities
$\mathfrak{d}_r(\mathbf{x})$ and $\mathfrak{d}_r(\mathbf{x},\mathbf{y})$ were introduced in
Definition \ref{Definition d} and note that the values $\mathfrak{d}_r(\mathbf{x})$ and
$\mathfrak{d}_r(\mathbf{x},\mathbf{y})$ are necessarily attained by primitive lattices. In addition, for any
$r \in \{2, \dots, n+1\}$ we have the trivial upper bounds
\begin{equation}
\label{Upper bound trivial drx}
\mathfrak{d}_r(\mathbf{x}) \leq ||\mathbf{x}||,
\end{equation}
and
\begin{equation}
\label{Upper bound trivial drxy}
\mathfrak{d}_r(\mathbf{x},\mathbf{y}) \leq ||\mathbf{x}|| \cdot ||\mathbf{y}||.
\end{equation}
We shall need to understand how often one can improve upon these upper bounds as the vectors $\mathbf{x}$ and
$\mathbf{y}$ run over $\mathbb{Z}^{n+1}$.

It is well-known that the successive minima of a random lattice are expected to have equal order of magnitude. In order to exploit this fact we will require an upper bound for the number of primitive lattices of given rank and whose successive minima are constrained to lie in dyadic intervals. The following notation will thus be very useful.

\begin{definition}
\label{Definition S}
Let $n \geq 2$ and $r \in \{1, \dots, n+1\}$. Given $s_1, \dots, s_r \geq 1$, we let $S_{r,n}(s_1, \dots, s_r)$ denote the set of primitive lattices $L \subset \mathbb{Z}^{n+1}$ of rank $r$ and such that $\lambda_j(L) \in (s_j/2, s_j]$ for any
$j \in \{1, \dots, r\}$.
\end{definition}

We shall prove the following result.

\begin{lemma}
\label{Lemma number lattices}
Let $n \geq 2$ and $r \in \{1, \dots, n+1 \}$. For $s_1, \dots, s_r \geq 1$, we have 
\begin{equation*}
\#S_{r,n}(s_1, \dots, s_r) \ll s_1^{n+r} s_2^{n+r-2} \cdots s_r^{n-r+2},
\end{equation*}
where the implied constant depends at most on $n$.
\end{lemma}

\begin{proof}
We proceed by induction on the integer $r$ and we start by noting that the case $r=1$ follows from the observation that
\begin{equation*}
\#S_{1,n}(s_1) \leq \# \left( \mathbb{Z}^{n+1} \cap \mathcal{B}_{n+1}(s_1) \right).
\end{equation*}

We now assume that the result holds for some integer $r-1 \in \{1, \dots, n \}$. Given $L_r \in S_{r,n}(s_1, \dots, s_r)$, for each $j \in \{1, \dots, r-1\}$ we pick $\mathbf{b}_j \in L$ such that $||\mathbf{b}_j|| = \lambda_j(L_r)$ and we introduce the primitive lattice
\begin{equation*}
L_r^{(-1)}=\left( \mathbb{R} \mathbf{b}_1 \oplus \dots \oplus \mathbb{R} \mathbf{b}_{r-1} \right) \cap \mathbb{Z}^{n+1}.
\end{equation*}
Note that $L_r^{(-1)}$ depends on $L$ but may also depend on our choice of $\mathbf{b}_1, \dots, \mathbf{b}_{r-1}$. Since $L_r^{(-1)} \in S_{r-1,n}(s_1, \dots, s_{r-1})$, we deduce that
\begin{equation*}
\# S_{r,n}(s_1, \dots, s_r) \leq \sum_{L_{r-1} \in S_{r-1,n}(s_1, \dots, s_{r-1})}
\# \left\{ L_r \in S_{r,n}(s_1, \dots, s_r) : L_r^{(-1)} = L_{r-1} \right\}.
\end{equation*}
The estimates \eqref{Estimates Minkowski} imply that $\det(L_{r-1}) \ll s_1 \cdots s_{r-1}$ and $\det(L_r)/\det(L_{r-1}) \ll s_r$ for any $L_{r-1} \in S_{r-1,n}(s_1, \dots, s_{r-1})$ and $L_r \in S_{r,n}(s_1, \dots, s_r)$. It thus follows from the work of Schmidt \cite[Lemma~$6$]{MR224562} (with $i=1$) that
\begin{equation*}
\# \left\{ L_r \in S_{r,n}(s_1, \dots, s_r) : L_r^{(-1)} = L_{r-1} \right\} \ll s_r^{n-r+2} s_1 \cdots s_{r-1}.
\end{equation*}
Hence we obtain
\begin{equation*}
\# S_{r,n}(s_1, \dots, s_r) \ll s_r^{n-r+2} s_1 \cdots s_{r-1} \cdot \# S_{r-1,n}(s_1, \dots, s_{r-1}).
\end{equation*}
An application of the induction hypothesis completes the proof.
\end{proof}

For given $X, Y, \Delta \geq 1$, let
\begin{equation}
\label{Definition l_r(X)}
\ell_{r,n}(X;\Delta) = \# \left\{ \mathbf{x} \in \mathbb{Z}^{n+1} :
\begin{array}{l l}
0 < ||\mathbf{x}|| \leq X \\
\mathfrak{d}_r(\mathbf{x}) \leq \Delta
\end{array}
\right\},
\end{equation}
and
\begin{equation}
\label{Definition l_r(X,Y)}
\ell_{r,n}(X,Y;\Delta) = \# \left\{ (\mathbf{x}, \mathbf{y}) \in \mathbb{Z}^{n+1} \times \mathbb{Z}^{n+1} :
\begin{array}{l l}
\dim \left( \Span_{\mathbb{R}} (\{ \mathbf{x}, \mathbf{y} \}) \right) = 2 \\
||\mathbf{x}|| \leq X, ~||\mathbf{y}|| \leq Y \\
\mathfrak{d}_r(\mathbf{x},\mathbf{y}) \leq \Delta
\end{array}
\right\}.
\end{equation}
We begin by analysing the first of these quantities.

\begin{lemma}
\label{Lemma l_r(X)}
Let $n \geq 2$ and $r \in \{2, \dots, n+1 \}$. For $X, \Delta \geq 1$, we have
\begin{equation*}
\ell_{r,n}(X;\Delta) \ll X^r \Delta^n \log \Delta,
\end{equation*}
where the implied constant depends at most on $n$.
\end{lemma}

\begin{proof}
We clearly have
\begin{equation*}
\ell_{r,n}(X;\Delta) \ll \sum_{\substack{\mathbf{x} \in \mathbb{Z}^{n+1} \\ 0 < ||\mathbf{x}|| \leq X}} \ 
\sum_{\substack{s_1 \leq \cdots \leq s_r \\s_1 \cdots s_r \ll \Delta}} \ 
\sum_{\substack{L \in S_{r,n}(s_1, \dots, s_r) \\ \mathbf{x} \in L}} 1,
\end{equation*}
where the summations over $s_1, \dots, s_r$ are over dyadic intervals. Moreover, we note that if
$L \subset \mathbb{Z}^{n+1}$ is a lattice containing a non-zero vector $\mathbf{x}$ then
$\lambda_1(L) \leq ||\mathbf{x}||$. We thus have
\begin{equation*}
\ell_{r,n}(X;\Delta) \ll \sum_{s_2 \leq \cdots \leq s_r} \ \sum_{\substack{s_1 \leq X \\ s_1 \cdots s_r \ll \Delta}} \
\sum_{L \in S_{r,n}(s_1, \dots, s_r)} \# \left( L \cap \mathcal{B}_{n+1}(X) \right).
\end{equation*}
Since $X \geq s_1$, Lemma~\ref{Lemma Schmidt upper bound} gives
\begin{equation*}
\sum_{L \in S_{r,n}(s_1, \dots, s_r)} \# \left( L \cap \mathcal{B}_{n+1}(X) \right) \ll
\frac{X^r}{s_1^r} \cdot \# S_{r,n}(s_1, \dots, s_r).
\end{equation*}
An application of Lemma~\ref{Lemma number lattices} thus yields
\begin{equation*}
\ell_{r,n}(X;\Delta) \ll X^r \sum_{s_2 \leq \cdots \leq s_r} \ \sum_{s_1 \ll \Delta/s_2 \cdots s_r}
s_1^n s_2^{n+r-2} s_3^{n+r-4} \cdots s_r^{n-r+2}.
\end{equation*}
Summing over $s_1$ we obtain
\begin{equation*}
\ell_{r,n}(X;\Delta) \ll X^r \Delta^n \sum_{s_2 \leq \cdots \leq s_r \ll \Delta} s_2^{r-2} s_3^{r-4} \cdots s_r^{-r+2},
\end{equation*}
which completes the proof on summing over the remaining variables.
\end{proof}

We now handle the quantity $\ell_{r,n}(X,Y;\Delta)$ and prove the following result.

\begin{lemma}
\label{Lemma l_r(X,Y)}
Let $n \geq 2$ and $r \in \{2, \dots, n+1 \}$. For $X, Y, \Delta \geq 1$, we have
\begin{equation*}
\ell_{r,n}(X,Y;\Delta) \ll X^r Y^r \Delta^{n-1} (\log \Delta)^{2 \min\{r-2,1\}},
\end{equation*}
where the implied constant depends at most on $n$.
\end{lemma}

\begin{proof}
We proceed as in the proof of Lemma~\ref{Lemma l_r(X)} but here we note that if $L \subset \mathbb{Z}^{n+1}$ is a lattice containing two linearly independent vectors $\mathbf{x}$ and $\mathbf{y}$ then
\begin{equation*}
\lambda_1(L) \leq \min \left\{ ||\mathbf{x}||, ||\mathbf{y}|| \right\},
\end{equation*}
and
\begin{equation*}
\lambda_2(L) \leq \max \left\{ ||\mathbf{x}||, ||\mathbf{y}|| \right\}.
\end{equation*}
We can assume by symmetry that $Y \geq X$. We thus obtain
\begin{equation*}
\ell_{r,n}(X,Y;\Delta) \ll \sum_{\substack{s_2\leq \cdots \leq s_r \\ s_2 \leq Y}} \
\sum_{\substack{s_1 \leq \min \{s_2, X\} \\ s_1 \cdots s_r \ll \Delta}} \ \sum_{L \in S_{r, n}(s_1,\dots,s_r)}
\# \left( L \cap \mathcal{B}_{n+1}(X) \right) \# \left( L \cap \mathcal{B}_{n+1}(Y) \right),
\end{equation*}
where the summations over $s_1, \dots, s_r$ are over dyadic intervals.

We first treat the case $r \geq 3$. We apply twice Lemma~\ref{Lemma Schmidt upper bound} using the inequalities $X \geq s_1$ and $Y \geq s_2$. It follows that
\begin{equation*}
\sum_{L \in S_{r,n}(s_1,\dots,s_r)} \# \left( L \cap \mathcal{B}_{n+1}(X) \right)
\# \left( L \cap \mathcal{B}_{n+1}(Y) \right) \ll \frac{X^r}{s_1^r} \cdot \frac{Y^r}{s_1 s_2^{r-1}} \cdot
\# S_{r,n}(s_1,\dots,s_r).
\end{equation*}
Invoking Lemma~\ref{Lemma number lattices}, we thus deduce
\begin{equation*}
\ell_{r,n}(X,Y;\Delta) \ll X^r Y^r \sum_{s_2 \leq \cdots \leq s_r} \ \sum_{s_1 \ll \Delta/s_2 \cdots s_r}
s_1^{n-1} s_2^{n-1} s_3^{n+r-4} s_4^{n+r-6} \cdots s_r^{n-r+2}.
\end{equation*}
The summation over $s_1$ leads to
\begin{equation*}
\ell_{r,n}(X,Y;\Delta) \ll X^r Y^r \Delta^{n-1} \sum_{s_2 \leq \cdots \leq s_r \ll \Delta} s_3^{r-3} s_4^{r-5} \cdots s_r^{-r+3}.
\end{equation*}
Summing over the remaining variables completes the proof in the case $r \geq 3$.

In the case $r=2$, using the inequalities $X \geq s_1$ and $Y \geq s_2 \geq s_1$ we see that
Lemma~\ref{Lemma Schmidt upper bound} gives
\begin{equation*}
\sum_{L \in S_{2,n}(s_1, s_2)} \# \left( L \cap \mathcal{B}_{n+1}(X) \right) \# \left( L \cap \mathcal{B}_{n+1}(Y) \right) \ll \left( \frac{X^2}{s_1 s_2} + \frac{X}{s_1} \right) \frac{Y^2}{s_1 s_2} \cdot \# S_{2,n}(s_1, s_2).
\end{equation*}
Therefore, it follows from Lemma~\ref{Lemma number lattices} that
\begin{equation*}
\ell_{2,n}(X,Y;\Delta) \ll X^2 Y^2 \sum_{\substack{s_1 s_2 \ll \Delta \\ s_1 \leq s_2}} s_1^n s_2^{n-2} +
X Y^2 \sum_{\substack{s_1 s_2 \ll \Delta \\ s_1 \leq X}} s_1^n s_2^{n-1}.
\end{equation*}
Summing over $s_2$ we derive
\begin{equation*}
\ell_{2,n}(X,Y;\Delta) \ll
X^2 Y^2 \Delta^{n-2} \sum_{s_1 \ll \Delta^{1/2}} s_1^2 + X Y^2 \Delta^{n-1} \sum_{s_1 \leq X} s_1,
\end{equation*}
which finishes the proof in the case $r=2$ on summing over $s_1$.
\end{proof}

\subsection{Handling the case of quartic threefolds}

\label{Section quartic threefolds}

In the hardest case $(d,n)=(4,4)$ of Theorem~\ref{Theorem 1}, we will struggle to handle the contribution from choices of linearly independent vectors $\mathbf{x}, \mathbf{y} \in \mathbb{Z}_{\mathrm{prim}}^5$ which produce particularly short vectors in the lattice $\Lambda_{\nu_{4,4}(\mathbf{x})} \cap \Lambda_{\nu_{4,4}(\mathbf{y})}$ and which happen to lie in a lattice of rank $3$ with small determinant. We shall deal with this issue by showing that such vectors are very rare.

Note that for any linearly independent vectors $\mathbf{x}, \mathbf{y} \in \mathbb{Z}^5$ the lattice
$\Lambda_{\nu_{4,4}(\mathbf{x})} \cap \Lambda_{\nu_{4,4}(\mathbf{y})}$ has rank $68$ since $N_{4,4} = 70$. For given $Z, \Delta, M \geq 1$ and $j \in \{1, \dots, 68\}$ we define
\begin{equation}
\label{Definition lj}
\ell^{(j)}(Z; \Delta, M) = \# \left\{ (\mathbf{x}, \mathbf{y}) \in \mathbb{Z}^5 \times \mathbb{Z}^5 :
\begin{array}{l l}
\dim \left( \Span_{\mathbb{R}} (\{ \mathbf{x}, \mathbf{y} \}) \right) = 2 \\
||\mathbf{x}||, ||\mathbf{y}|| \leq Z \\
\mathfrak{d}_3(\mathbf{x},\mathbf{y}) \leq \Delta\\
\lambda_j(\Lambda_{\nu_{4,4}(\mathbf{x})} \cap \Lambda_{\nu_{4,4}(\mathbf{y})}) \leq M
\end{array}
\right\}.
\end{equation}
Lemma~\ref{Lemma l_r(X,Y)} immediately shows that for any $j \in \{1, \dots, 68\}$ we have
\begin{equation}
\label{Upper bound lj}
\ell^{(j)}(Z; \Delta, M) \ll Z^6 \Delta^3 (\log \Delta)^2.
\end{equation}
We now focus on the case $j=1$ and we shall obtain the following result, which improves on this upper bound when $M$ is small.

\begin{lemma}
\label{Lemma l1}
Let $\varepsilon > 0$. For $Z, \Delta, M \geq 1$, we have
\begin{equation*}
\ell^{(1)}(Z; \Delta, M) \ll M^{40} Z^4 \Delta
\left( M^{30} Z^2 \Delta^{\varepsilon} + M^{30} \Delta^{8/3} + Z^2 \Delta \right) (\log \Delta)^2,
\end{equation*}
where the implied constant depends at most on $\varepsilon$.
\end{lemma}

\begin{proof}
Recall that the set $S_{3,4}(s_1,s_2,s_3)$ was introduced in Definition \ref{Definition S}. In a similar way as in the proofs of Lemmas~\ref{Lemma l_r(X)} and \ref{Lemma l_r(X,Y)}, we start by noting that
\begin{equation}
\label{Upper bound l1}
\ell^{(1)}(Z; \Delta, M) \ll \sum_{\substack{\mathbf{c} \in \mathbb{Z}^{N_{4,4}} \\ 0 < ||\mathbf{c}|| \leq M}} \ 
\sum_{\substack{s_1 \leq s_2 \leq \max \{Z, s_3\} \\ s_1 s_2 s_3 \ll \Delta}} \ \sum_{L \in S_{3,4}(s_1,s_2,s_3)}
N_{\mathbf{c}}(Z;L)^2,
\end{equation}
where the summations over $s_1$, $s_2$ and $s_3$ are over dyadic intervals and
\begin{equation}
\label{Definition Nc}
N_{\mathbf{c}}(Z;L) = \# \left\{\mathbf{x} \in L \cap \mathcal{B}_5(Z) :
\langle \nu_{4,4}(\mathbf{x}), \mathbf{c} \rangle = 0 \right\}.
\end{equation}
Given a non-zero vector $\mathbf{c} \in \mathbb{Z}^{N_{4,4}}$, we let
\begin{equation*}
S_{3,4}(s_1,s_2,s_3; \mathbf{c}) = \left\{ L \in S_{3,4}(s_1,s_2,s_3) :
\Span_{\mathbb{R}}(L) \subset \left\{ \mathbf{x} \in \mathbb{R}^5 : \langle \nu_{4,4}(\mathbf{x}), \mathbf{c} \rangle = 0 \right\} \right\},
\end{equation*}
and we also let $T_{3,4}(s_1,s_2,s_3; \mathbf{c})$ be the complement of $S_{3,4}(s_1,s_2,s_3; \mathbf{c})$ in
$S_{3,4}(s_1,s_2,s_3)$.

We first handle the contribution from lattices belonging to the set $S_{3,4}(s_1,s_2,s_3; \mathbf{c})$. We note that if
$L \in S_{3,4}(s_1,s_2,s_3; \mathbf{c})$ then
\begin{equation*}
N_{\mathbf{c}}(Z;L) = \# \left(L \cap \mathcal{B}_5(Z) \right),
\end{equation*}
so the inequality $Z \geq s_2$ gives
\begin{equation}
\label{Upper bound fixed c}
\sum_{L \in S_{3,4}(s_1,s_2,s_3; \mathbf{c})} N_{\mathbf{c}}(Z;L)^2 \ll
\left( \frac{Z}{s_1} \cdot \frac{Z}{s_2} \left( \frac{Z}{s_3} + 1 \right) \right)^2 \# S_{3,4}(s_1,s_2,s_3;\mathbf{c}).
\end{equation}
In addition, for any lattice $L \in S_{3,4}(s_1,s_2,s_3; \mathbf{c})$ we can use \cite[Lemma~$5$]{MR155800} to pick a basis $(\mathbf{b}_1,\mathbf{b}_2,\mathbf{b}_3)$ of $L$ such that for any $j \in \{1, 2, 3\}$, we have
\begin{equation*}
\lambda_j(L) \leq ||\mathbf{b}_j|| \ll \lambda_j(L).
\end{equation*}
We deduce that there exists an absolute constant $C>0$ such that
\begin{equation*}
\# S_{3,4}(s_1,s_2,s_3; \mathbf{c}) \ll \prod_{j=1}^3
\# \left\{\mathbf{b}_j \in \mathcal{B}_5(C s_j) : \langle \nu_{4,4}(\mathbf{b}_j), \mathbf{c} \rangle = 0 \right\}.
\end{equation*}
Since $\mathbf{c}$ is a non-zero vector we trivially have
\begin{equation*}
\# S_{3,4}(s_1,s_2,s_3; \mathbf{c}) \ll s_1^4 s_2^4 s_3^4,
\end{equation*}
where the implied constant is independent of $\mathbf{c}$. Furthermore, if the quartic form
$\langle \nu_{4,4}(\mathbf{u}), \mathbf{c} \rangle$ is irreducible over $\overline{\mathbb{Q}}$ then we can appeal to work of Broberg and Salberger \cite[Theorem~$1$]{MR2029864}. It follows in this case that for any $\varepsilon>0$ we have
\begin{equation}
\label{Upper bound det method}
\# S_{3,4}(s_1,s_2,s_3; \mathbf{c}) \ll s_1^{3+\varepsilon} s_2^{3+\varepsilon} s_3^{3+\varepsilon},
\end{equation}
where the implied constant may depend on $\varepsilon$ but, crucially, not on $\mathbf{c}$. In addition, if the form
$\langle \nu_{4,4}(\mathbf{u}), \mathbf{c} \rangle$ is reducible over $\overline{\mathbb{Q}}$ but irreducible over
$\mathbb{Q}$ then we see that the set
$\left\{ \mathbf{b} \in \mathbb{Z}^5 : \langle \nu_{4,4}(\mathbf{b}), \mathbf{c} \rangle = 0 \right\}$ lies on an affine subvariety of codimension at least $2$, and a trivial estimate directly yields the upper bound
\eqref{Upper bound det method} with $\varepsilon = 0$. Recalling the upper bound \eqref{Upper bound fixed c} we thus see that
\begin{equation}
\label{Upper bound S}
\sum_{L \in S_{3,4}(s_1,s_2,s_3; \mathbf{c})} N_{\mathbf{c}}(Z;L)^2 \ll
Z^6 s_1^{\vartheta_{\mathbf{c}}} s_2^{\vartheta_{\mathbf{c}}} s_3^{\vartheta_{\mathbf{c}}} +
\frac{Z^4}{s_1^2 s_2^2} \cdot \# S_{3,4}(s_1,s_2,s_3;\mathbf{c}),
\end{equation}
where
\begin{equation*}
\vartheta_{\mathbf{c}} =
\begin{cases}
1 + \varepsilon, & \textrm{if the form } \langle \nu_{4,4}(\mathbf{u}), \mathbf{c} \rangle \textrm{ is irreducible over } \mathbb{Q}, \\
2, & \textrm{otherwise}.
\end{cases}
\end{equation*}

We now deal with the contribution from lattices belonging to the set $T_{3,4}(s_1,s_2,s_3; \mathbf{c})$. Given
$L \in T_{3,4}(s_1,s_2,s_3; \mathbf{c})$ we again use \cite[Lemma~$5$]{MR155800} to select a basis
$(\mathbf{b}_1,\mathbf{b}_2,\mathbf{b}_3)$ of $L$ with the property that if $\mathbf{x} \in L$ is given by
$\mathbf{x} = t_1 \mathbf{b}_1 + t_2 \mathbf{b}_2 + t_3 \mathbf{b}_3$ for some $(t_1, t_2, t_3) \in \mathbb{Z}^3$ then $t_j \ll ||\mathbf{x}|| / s_j$ for any $j \in \{1, 2, 3 \}$. We thus have
\begin{equation}
\label{Definition Set t_i}
N_{\mathbf{c}}(Z;L) \ll \# \left\{ (t_1,t_2,t_3) \in \mathbb{Z}^3 :
\begin{array}{l l}
t_j \ll Z/s_j, ~j \in \{1, 2, 3\} \\
\langle \nu_{4,4}(t_1 \mathbf{b}_1 + t_2 \mathbf{b}_2 + t_3 \mathbf{b}_3), \mathbf{c} \rangle = 0
\end{array}
\right\}.
\end{equation}
Since $L \in T_{3,4}(s_1,s_2,s_3; \mathbf{c})$ the polynomial function
$\langle \nu_{4,4}(u_1 \mathbf{b}_1 + u_2 \mathbf{b}_2 + u_3 \mathbf{b}_3), \mathbf{c} \rangle$ is not identically equal to $0$ as $(u_1,u_2,u_3)$ runs over $\mathbb{R}^3$. It follows that there exists $j \in \{1, 2, 3\}$ such that the coordinate $t_j$ of the elements of the set in the right-hand side of the upper bound \eqref{Definition Set t_i} can assume at most $4$ values when the two other coordinates are fixed. Therefore, since $Z \geq s_2$ we have
\begin{equation}
\label{Upper bound Nc}
N_{\mathbf{c}}(Z;L) \ll \frac{Z}{s_1} \cdot \frac{Z}{s_2},
\end{equation}
where the implies constant does not depend on $\mathbf{c}$. This trivially yields
\begin{equation}
\label{Upper bound T}
\sum_{L \in T_{3,4}(s_1,s_2,s_3;\mathbf{c})} N_{\mathbf{c}}(Z;L)^2 \ll \frac{Z^4}{s_1^2 s_2^2} \cdot \# T_{3,4}(s_1,s_2,s_3; \mathbf{c}).
\end{equation}

Combining the upper bounds \eqref{Upper bound S} and \eqref{Upper bound T} we deduce that
\begin{equation*}
\sum_{L \in S_{3,4}(s_1,s_2,s_3)} N_{\mathbf{c}}(Z;L)^2 \ll
Z^6 s_1^{\vartheta_{\mathbf{c}}} s_2^{\vartheta_{\mathbf{c}}} s_3^{\vartheta_{\mathbf{c}}} +
\frac{Z^4}{s_1^2 s_2^2} \cdot \# S_{3,4}(s_1,s_2,s_3).
\end{equation*}
Recall that $N_{4,4}=70$. Using the upper bound \eqref{Upper bound l1} and Lemma~\ref{Lemma number lattices} we thus derive
\begin{equation*}
\ell^{(1)}(Z; \Delta, M) \ll M^{70} \sum_{\substack{s_1 \leq s_2 \leq s_3 \\ s_1 s_2 s_3 \ll \Delta}}
\left( Z^6 s_1^{1+\varepsilon} s_2^{1+\varepsilon} s_3^{1+\varepsilon} + Z^4 s_1^5 s_2^3 s_3^3 \right) +
M^{40} Z^6 \sum_{\substack{s_1 \leq s_2 \leq s_3 \\ s_1 s_2 s_3 \ll \Delta}} s_1^2 s_2^2 s_3^2.
\end{equation*}
Note that we have used the fact that the number of $\mathbf{c} \in \mathbb{Z}^{N_{4,4}}$ satisfying
$||\mathbf{c}|| \leq M$ and such that the form $\langle \nu_{4,4}(\mathbf{u}), \mathbf{c} \rangle$ is reducible over
$\mathbb{Q}$ is trivially bounded by an absolute constant times
\begin{equation*}
M^{\max \left\{ \binom{5}{1}+\binom{7}{3}, 2 \times \binom{6}{2} \right\}} = M^{40}.
\end{equation*}
We finally remark that it follows from the inequalities $s_1 \leq s_2 \leq s_3$ and $s_1 s_2 s_3 \ll \Delta$ that
$s_1^5 s_2^3 s_3^3 \ll \Delta^{11/3}$, which finishes the proof.
\end{proof}

We now combine Lemmas~\ref{Lemma l_r(X,Y)} and \ref{Lemma l1} to deduce an upper bound which will prove to be very convenient in the proof of Lemma~\ref{Lemma D44}.

\begin{lemma}
\label{Lemma l1 convenient}
For $Z, \Delta, M \geq 1$, we have
\begin{equation*}
\ell^{(1)}(Z; \Delta, M) \ll Z^{9/2} \Delta^2 \left( Z^{3/2} + \Delta^{3/2} \right) (\log \Delta)^2 \min \left\{ \Delta, M^{53} \right\},
\end{equation*}
where the implied constant is absolute.
\end{lemma}

\begin{proof}
If $\Delta \leq M^{53}$ then the upper bound \eqref{Upper bound lj} provides the desired result. In the case where
$\Delta > M^{53}$, we see that $M^{30} Z^2 \Delta^{\varepsilon} \leq Z^2 \Delta$ if $\varepsilon = 1/3$, say. Therefore, using the upper bound \eqref{Upper bound lj} and Lemma~\ref{Lemma l1} we deduce that
\begin{equation*}
\ell^{(1)}(Z; \Delta, M) \ll Z^4 \Delta (\log \Delta)^2 \min \left\{ Z^2 \Delta^2, M^{40} \Delta \left( M^{30} \Delta^{5/3} + Z^2 \right) \right\}.
\end{equation*}
First, if $Z^2 > M^{30} \Delta^{5/3}$ then we obtain
\begin{equation*}
\ell^{(1)}(Z; \Delta, M) \ll Z^6 \Delta^2 (\log \Delta)^2 M^{40},
\end{equation*}
which is satisfactory. Next, if $Z^2 \leq M^{30} \Delta^{5/3}$ we get
\begin{equation*}
\ell^{(1)}(Z; \Delta, M) \ll Z^4 \Delta^3 (\log \Delta)^2 \min \left\{ Z^2, M^{70} \Delta^{2/3} \right\}.
\end{equation*}
Using the inequality
\begin{equation*}
\min \left\{ Z^2, M^{70} \Delta^{2/3} \right\} \leq \left(Z^2\right)^{1/4} \left( M^{70} \Delta^{2/3} \right)^{3/4},
\end{equation*}
we derive
\begin{equation*}
\ell^{(1)}(Z; \Delta, M) \ll Z^{9/2} \Delta^{7/2} (\log \Delta)^2 M^{105/2},
\end{equation*}
which completes the proof.
\end{proof}

In the worst situation, that is when $M$ and $\Delta$ are both small, we see that Lemma~\ref{Lemma l1 convenient} is not much stronger than the upper bound \eqref{Upper bound lj}. In this case we will thus require a different argument. Given a rank $3$ lattice $L \subset \mathbb{Z}^5$ we define the lattice
\begin{equation}
\label{Definition V(L)}
\mathscr{V}(L) = \Span_{\mathbb{Q}}(\nu_{4,4}(L)) \cap \mathbb{Z}^{N_{4,4}},
\end{equation}
and we note that the dimension of the subspace $\Span_{\mathbb{Q}}(\nu_{4,4}(L))$ is equal to $\binom{6}{4} = 15$ and therefore $\mathscr{V}(L)$ has rank $15$. We shall prove the following result, which states that the determinant of the lattice $\mathscr{V}(L)$ can be controlled in terms of the determinant of $L$.

\begin{lemma}
\label{Lemma det44}
Let $L \subset \mathbb{Z}^5$ be a lattice of rank $3$. Then
\begin{equation*}
\det(\mathscr{V}(L)) \ll \det(L)^{20},
\end{equation*}
where the implied constant is absolute.
\end{lemma}

\begin{proof}
We start by using \cite[Lemma~$5$]{MR155800} to select a basis $(\mathbf{b}_1, \mathbf{b}_2, \mathbf{b}_3)$ of $L$ 
satisfying
\begin{equation}
\label{Lower bound det L}
||\mathbf{b}_1|| \cdot ||\mathbf{b}_2|| \cdot ||\mathbf{b}_3|| \ll \det(L).
\end{equation}
By definition we have
\begin{equation*}
\mathscr{V}(L) = \Span_{\mathbb{Q}} \left(
\left\{ \nu_{4,4} \left( \ell_1 \mathbf{b}_1 + \ell_2 \mathbf{b}_2 + \ell_3 \mathbf{b}_3 \right) : (\ell_1, \ell_2, \ell_3) \in \mathbb{Z}^3 \right\} \right) \cap \mathbb{Z}^{N_{4,4}}.
\end{equation*}
Letting $(b_{j,0}, \dots, b_{j,4})$ be the coordinates of the vector $\mathbf{b}_j$ for any $j \in \{1, 2, 3\}$, we see that the vector $\nu_{4,4} \left( X_1 \mathbf{b}_1 + X_2 \mathbf{b}_2 + X_3 \mathbf{b}_3 \right)$ has coordinates 
\begin{equation*}
\prod_{j=1}^4 ( X_1 b_{1,i_j} + X_2 b_{2,i_j} + X_3 b_{3,i_j}),
\end{equation*}
that are indexed by the $(i_1, \dots,i_4) \in \mathbb{Z}^4$ satisfying $0 \leq i_1 \leq i_2 \leq i_3 \leq i_4 \leq 4$. As a result, by using the polynomial identity
\begin{equation*}
\nu_{4,4} \left( X_1 \mathbf{b}_1 + X_2 \mathbf{b}_2 + X_3 \mathbf{b}_3 \right) =
\sum_{\substack{0 \leq e_1,e_2,e_3 \leq 4 \\ e_1+e_2+e_3 = 4}} X_1^{e_1} X_2^{e_2} X_3^{e_3} \mathbf{v}_{e_1,e_2,e_3},
\end{equation*}
we define $15$ vectors $\mathbf{v}_{e_1,e_2,e_3} \in \mathbb{Z}^{N_{4,4}}$ indexed by the
$(e_1,e_2,e_3) \in \{0, \dots, 4\}^3$ such that $e_1+e_2+e_3=4$, and we note that
\begin{equation}
\label{Upper bound norm v}
||\mathbf{v}_{e_1,e_2,e_3}|| \ll ||\mathbf{b}_1||^{e_1} ||\mathbf{b}_2||^{e_2} || \mathbf{b}_3||^{e_3}.
\end{equation}
Since the lattice $\mathscr{V}(L)$ has rank $15$, we see that the $15$ vectors $\mathbf{v}_{e_1,e_2,e_3}$ are linearly independent. Therefore, the lattice
\begin{equation*}
\bigoplus_{\substack{0 \leq e_1,e_2,e_3 \leq 4 \\ e_1+e_2+e_3 = 4}} \mathbb{Z} \mathbf{v}_{e_1,e_2,e_3}
\end{equation*}
is a sublattice of $\mathscr{V}(L)$ of finite index, whence
\begin{equation*}
\det (\mathscr{V}(L)) \leq \prod_{\substack{0 \leq e_1,e_2,e_3 \leq 4 \\ e_1+e_2+e_3 = 4}} ||\mathbf{v}_{e_1,e_2,e_3}||.
\end{equation*}
Using the upper bound \eqref{Upper bound norm v} and the identity
\begin{equation*}
\sum_{\substack{0 \leq e, f, g \leq 4 \\ e + f + g = 4}} e = 20,
\end{equation*}
we eventually obtain
\begin{equation*}
\det (\mathscr{V}(L)) \ll ||\mathbf{b}_1||^{20} ||\mathbf{b}_2||^{20}||\mathbf{b}_3||^{20},
\end{equation*}
which completes the proof on recalling the upper bound \eqref{Lower bound det L}.
\end{proof}

Recall the definition \eqref{Definition lj} of $\ell^{(j)}(Z; \Delta, M)$ for a given integer $j \in \{1, \dots, 68\}$. We now use Lemma~\ref{Lemma det44} to derive an upper bound for the quantity $\ell^{(56)}(Z; \Delta, M)$.

\begin{lemma}
\label{Lemma l56}
For $Z, \Delta, J \geq 1$, we have
\begin{equation*}
\ell^{(56)}(Z; \Delta, J) \ll Z^4 \Delta^{24} J^{15},
\end{equation*}
where the implied constant is absolute.
\end{lemma}

\begin{proof}
We start by noting that
\begin{equation*}
\ell^{(56)}(Z; \Delta, J) \ll
\sum_{\substack{s_1 \leq s_2 \leq \max \{Z, s_3\} \\ s_1 s_2 s_3 \ll \Delta}} \ \sum_{L \in S_{3,4}(s_1,s_2,s_3)} 
\# \mathcal{L}(Z,J;L),
\end{equation*}
where
\begin{equation*}
\mathcal{L}(Z,J;L) = \left\{ (\mathbf{x}, \mathbf{y}) \in \left( L \cap \mathcal{B}_5(Z) \right)^2 :
\begin{array}{l l}
\dim \left( \Span_{\mathbb{R}} (\{ \mathbf{x}, \mathbf{y} \}) \right) = 2 \\
\lambda_{56}(\Lambda_{\nu_{4,4}(\mathbf{x})} \cap \Lambda_{\nu_{4,4}(\mathbf{y})}) \leq J
\end{array}
\right\}.
\end{equation*}
Let $(\mathbf{x}, \mathbf{y}) \in \mathcal{L}(Z,J;L)$. Recall the definition \eqref{Definition V(L)} of the lattice
$\mathscr{V}(L)$ and recall that its rank is equal to $15$. The orthogonal lattice $\mathscr{V}(L)^{\perp}$ has rank $70-15=55$ so the inequality $\lambda_{56}(\Lambda_{\nu_{4,4}(\mathbf{x})} \cap \Lambda_{\nu_{4,4}(\mathbf{y})}) \leq J$ implies that there exists $\mathbf{a} \in \Lambda_{\nu_{4,4}(\mathbf{x})}\cap\Lambda_{\nu_{4,4}(\mathbf{y})}$ such that $||\mathbf{a}|| \leq J$ and $\mathbf{a} \notin \mathscr{V}(L)^{\perp}$. Furthermore, by definition the lattice
$\mathscr{V}(L)$ is primitive so $(\mathscr{V}(L)^{\perp})^{\perp} = \mathscr{V}(L)$ and we deduce from
Lemma~\ref{Lemma Schmidt quotient} that $\mathscr{V}(L)^{\ast}=\mathbb{Z}^{70}/\mathscr{V}(L)^{\perp}$. Hence, if $\pi : \mathbb{R}^{70} \to \Span_{\mathbb{R}}(\mathscr{V}(L))$ denotes the orthogonal projection on
$\Span_{\mathbb{R}}(\mathscr{V}(L))$ then the vector $\mathbf{b} = \pi(\mathbf{a})$ belongs to
$\mathscr{V}(L)^{\ast}$. In addition, we have $||\mathbf{b}|| \leq ||\mathbf{a}|| \leq J$ and $\mathbf{b}$ is non-zero since $\mathbf{a} \notin \Span_{\mathbb{R}}(\mathscr{V}(L))^{\perp}$. Finally we note that
\begin{equation}
\label{Equalities 0}
\langle \mathbf{b}, \nu_{4,4}(\mathbf{x}) \rangle = \langle \mathbf{b}, \nu_{4,4}(\mathbf{y}) \rangle=0.
\end{equation}
Indeed, by definition of $\mathbf{b}$ we have $\mathbf{a}-\mathbf{b} \in \Span_{\mathbb{R}}(\mathscr{V}(L))^{\perp}$, which implies that
\begin{equation*}
\langle \mathbf{a}-\mathbf{b}, \nu_{4,4}(\mathbf{x}) \rangle =
\langle \mathbf{a}-\mathbf{b}, \nu_{4,4}(\mathbf{y}) \rangle = 0,
\end{equation*}
since $\mathbf{x},\mathbf{y} \in L$. Moreover we also have
$\mathbf{a} \in \Lambda_{\nu_{4,4}(\mathbf{x})} \cap \Lambda_{\nu_{4,4}(\mathbf{y})}$ so the equalities \eqref{Equalities 0} follow. As a result, recalling the definition \eqref{Definition Nc} of $N_{\mathbf{b}}(Z;L)$ we deduce that
\begin{equation}
\label{Upper bound intermediate l56}
\ell^{(56)}(Z; \Delta, J) \ll
\sum_{\substack{s_1 \leq s_2 \leq \max \{Z, s_3\} \\ s_1 s_2 s_3 \ll \Delta}} \ \sum_{L \in S_{3,4}(s_1,s_2,s_3)} \ 
\sum_{\substack{\mathbf{b} \in \mathscr{V}(L)^{\ast} \\ 0 < ||\mathbf{b}|| \leq J}} N_{\mathbf{b}}(Z;L)^2.
\end{equation}
If we had
$\Span_{\mathbb{R}}(L) \subset \left\{ \mathbf{x} \in \mathbb{R}^5 : \langle \nu_{4,4}(\mathbf{x}), \mathbf{b} \rangle = 0 \right\}$ then it would follow that $\langle \mathbf{b}, \mathbf{b} \rangle = 0$ since
$\mathbf{b} \in \Span_{\mathbb{R}}(\mathscr{V}(L))$. This is impossible as $\mathbf{b}$ is non-zero . Therefore, since we have $Z \geq s_2$, we can employ the upper bound \eqref{Upper bound Nc} to bound $N_{\mathbf{b}}(Z;L)$ independently of the vector $\mathbf{b}$. This yields
\begin{equation}
\label{Upper bound sum b}
\sum_{\substack{\mathbf{b} \in \mathscr{V}(L)^{\ast} \\ 0 < ||\mathbf{b}|| \leq J}} N_{\mathbf{b}}(Z;L)^2 \ll
\frac{Z^4}{s_1^2s_2^2} \cdot \# \left( \mathscr{V}(L)^{\ast} \cap \mathcal{B}_5(J) \right).
\end{equation}
In addition, the lattice $\mathscr{V}(L)$ is integral, so for any $i \in \{1, \dots, 15\}$ we have
$\lambda_i(\mathscr{V}(L)) \geq 1$. Lemma~\ref{Lemma Banaszczyk} thus implies that we also have
$\lambda_i(\mathscr{V}(L)^{\ast}) \leq 15$ for any $i \in \{1, \dots, 15\}$. Therefore, it follows from
Lemma~\ref{Lemma Schmidt upper bound} and the estimates \eqref{Estimates Minkowski} that
\begin{align*}
\# \left( \mathscr{V}(L)^{\ast} \cap \mathcal{B}_5(J) \right) & \ll
\frac{J^{15}}{\det (\mathscr{V}(L)^{\ast})} \left( \sum_{i = 0}^{15} \frac{\lambda_{16-i}(\mathscr{V}(L)^{\ast}) \cdots \lambda_{15}(\mathscr{V}(L)^{\ast})}{J^i} \right) \\
& \ll \frac{J^{15}}{\det (\mathscr{V}(L)^{\ast})}.
\end{align*}
We now use the equality \eqref{Det dual} and apply Lemma~\ref{Lemma det44} to derive the upper bound
\begin{equation}
\label{Upper bound dual lattice}
\# \left( \mathscr{V}(L)^{\ast} \cap \mathcal{B}_5(J) \right) \ll J^{15} \det(L)^{20}.
\end{equation}
Using again Minkowski's estimates \eqref{Estimates Minkowski} and putting together the upper bounds
\eqref{Upper bound intermediate l56}, \eqref{Upper bound sum b} and \eqref{Upper bound dual lattice}, we get
\begin{equation*}
\ell^{(56)}(Z; \Delta, J) \ll Z^4 J^{15}
\sum_{\substack{s_1 \leq s_2 \leq s_3 \\ s_1 s_2 s_3 \ll \Delta}} s_1^{18} s_2^{18} s_3^{20} \cdot \# S_{3,4}(s_1,s_2,s_3).
\end{equation*}
An application of Lemma~\ref{Lemma number lattices} eventually gives
\begin{equation*}
\ell^{(56)}(Z; \Delta, J) \ll Z^4 J^{15}
\sum_{\substack{s_1 \leq s_2 \leq s_3 \\ s_1 s_2 s_3 \ll \Delta}} s_1^{25} s_2^{23} s_3^{23},
\end{equation*}
which completes the proof on noting that the inequalities $s_1 \leq s_2 \leq s_3$ and $s_1 s_2 s_3 \ll \Delta$ imply that $s_1^{25} s_2^{23} s_3^{23} \ll \Delta^{71/3}$.
\end{proof}

\section{The global and localised counting functions are rarely apart}

\label{Section approximating}

In Section~\ref{Section variance} we check that Proposition~\ref{Proposition 1} is a direct consequence of a certain variance upper bound, as stated in Proposition~\ref{Proposition variance}. The remainder of
Section~\ref{Section approximating} is then devoted to the proof of Proposition~\ref{Proposition variance}. With this goal in mind we will start by estimating the volume of certain regions in $\mathbb{R}^N$ in
Section~\ref{Section volumes}. In Section~\ref{Section inverses determinants} we will then produce tight bounds for the average of the inverse of the determinant of the lattice
$\Lambda_{\nu_{d,n}(\mathbf{x})} \cap \Lambda_{\nu_{d,n}(\mathbf{y})}$, as one varies the linearly independent vectors $\mathbf{x}, \mathbf{y} \in \mathbb{Z}^{n+1}$. Our next task in Section~\ref{Section first} will be to prove an upper bound for the first moment of the counting function $N_V(B)$. In Section~\ref{Section second} we will turn to proving estimates for second moments involving both $N_V(B)$ and our localised counting function $N_V^{\mathrm{loc}}(B)$. We will finally combine all these estimates to prove Proposition~\ref{Proposition variance} in
Section~\ref{Section proof variance}.

\subsection{The key variance upper bound}

Recall the respective definitions \eqref{Definition counting} and \eqref{Definition local counting} of our two counting functions $N_V(B)$ and $N_V^{\mathrm{loc}}(B)$. The following result is the culmination of our work in
Section~\ref{Section approximating}.

\label{Section variance}

\begin{proposition}
\label{Proposition variance}
Let $d \geq 2$ and $n \geq d$ with $(d,n) \notin \{(2,2), (3,3) \}$. Assume that $B/(\log B)^{1/2} \leq A \leq B^2$. Then we have
\begin{equation*} 
\frac1{\# \mathbb{V}_{d,n}(A)} \sum_{V \in \mathbb{V}_{d,n}(A)} \left( N_V(B) - N_V^{\mathrm{loc}}(B) \right)^2 \ll \frac{B}{A}.
\end{equation*}
\end{proposition}

We now proceed to prove that Proposition~\ref{Proposition 1} follows from Proposition~\ref{Proposition variance}.

\begin{proof}[Proof of Proposition~\ref{Proposition 1}]
It is convenient to set
\begin{equation*}
\mathscr{L}_{\phi}(A) = \frac1{\#\mathbb{V}_{d,n}(A)} \cdot
\# \left\{ V \in \mathbb{V}_{d,n}(A) : \left| N_V(A \phi(A)) - N_V^{\mathrm{loc}}(A \phi(A)) \right| > \phi(A)^{2/3} \right\}.
\end{equation*}
We observe that 
\begin{equation*}
\mathscr{L}_{\phi}(A) \leq \frac1{\#\mathbb{V}_{d,n}(A)} \sum_{V \in \mathbb{V}_{d,n}(A)}
\left( \frac{N_V(A \phi(A)) - N_V^{\mathrm{loc}}(A \phi(A))}{\phi(A)^{2/3}} \right)^2.
\end{equation*}
By assumption we have $\phi(A) \leq (\log A)^{1/2}$ so we are in position to apply Proposition~\ref{Proposition variance}, which immediately completes the proof of Proposition~\ref{Proposition 1}.
\end{proof}

\subsection{Volume estimates}

\label{Section volumes}

For any $N \geq 1$ we let $V_N$ denote the volume of the unit ball $\mathcal{B}_N(1)$ in $\mathbb{R}^N$. For
$\mathbf{w}, \mathbf{z} \in \mathbb{R}^{N}$, we introduce the $(N-1)$-dimensional volume
\begin{equation}
\label{Definition volume mixed}
\mathcal{I}(\mathbf{w}, \mathbf{z}) = \vol \left( \left\{ \mathbf{t} \in (\mathbb{R} \mathbf{w})^{\perp} :
|\langle\mathbf{z}, \mathbf{t} \rangle| \leq ||\mathbf{t}|| \leq 1 \right\} \right),
\end{equation}
and we put
\begin{equation*}
\delta_{\mathbf{w},\mathbf{z}} = ||\mathbf{w}||^2 ||\mathbf{z}||^2 - \langle\mathbf{w}, \mathbf{z} \rangle^2.
\end{equation*}
We start by proving the following result.

\begin{lemma}
\label{Lemma volume mixed}
Let $N \geq 3$ and let $\mathbf{w}, \mathbf{z} \in \mathbb{R}^N$ be two linearly independent vectors. Then 
\begin{equation*}
\mathcal{I}(\mathbf{w}, \mathbf{z}) = 
2 \frac{N-2}{N-1} V_{N-2} \frac{||\mathbf{w}||}{\delta_{\mathbf{w},\mathbf{z}}^{1/2}}
\left( 1 + O \left( \min \left\{ 1, \frac{||\mathbf{w}||^2}{\delta_{\mathbf{w},\mathbf{z}}} \right\} \right) \right),
\end{equation*}
where the implied constant depends at most on $N$.
\end{lemma}

\begin{proof}
Let $(\mathbf{g}_3, \dots, \mathbf{g}_N)$ be an orthonormal basis of
$\left( \mathbb{R} \mathbf{w} \oplus \mathbb{R} \mathbf{z} \right)^{\perp}$ and set
\begin{equation*}
\mathbf{g}_2 =
\frac{||\mathbf{w}||^2 \mathbf{z}-\langle \mathbf{w}, \mathbf{z} \rangle \mathbf{w}}
{\delta_{\mathbf{w}, \mathbf{z}}^{1/2} \cdot ||\mathbf{w}||}.
\end{equation*}
We have $\mathbf{g}_2 \in \mathbb{R} \mathbf{w} \oplus \mathbb{R} \mathbf{z}$,
$\langle \mathbf{g}_2, \mathbf{w} \rangle =0$ and $||\mathbf{g}_2|| = 1$ so the family
$(\mathbf{g}_2, \dots, \mathbf{g}_N)$ is an orthonormal basis of $(\mathbb{R}\mathbf{w})^{\perp}$. It follows that
\begin{equation*}
\mathcal{I}(\mathbf{w}, \mathbf{z}) = \vol \left( \left\{ (u_2,\mathbf{u}) \in \mathbb{R}^{N-1} :
\frac{\delta_{\mathbf{w}, \mathbf{z}}}{||\mathbf{w}||^2} \cdot u_2^2 \leq u_2^2 + ||\mathbf{u}||^2 \leq 1
\right\} \right).
\end{equation*}
If $||\mathbf{w}||^2 \geq \delta_{\mathbf{w}, \mathbf{z}}$ then $\mathcal{I}(\mathbf{w}, \mathbf{z}) = V_{N-1}$ and the claimed estimate holds. We now handle the case where $||\mathbf{w}||^2 < \delta_{\mathbf{w}, \mathbf{z}}$. Integrating over $\mathbf{u}$ we obtain
\begin{align*}
\mathcal{I}(\mathbf{w}, \mathbf{z}) & = 2 V_{N-2}
\int_{0}^{||\mathbf{w}||/\delta_{\mathbf{w}, \mathbf{z}}^{1/2}} \left( \left( 1-u_2^2 \right)^{(N-2)/2} -
\left( \frac{\delta_{\mathbf{w}, \mathbf{z}}}{||\mathbf{w}||^2} - 1 \right)^{(N-2)/2} u_2^{N-2} \right) \mathrm{d} u_2 \\
& = 2 V_{N-2} \int_{0}^{||\mathbf{w}||/\delta_{\mathbf{w}, \mathbf{z}}^{1/2}}
\left( 1 - \left( \frac{\delta_{\mathbf{w}, \mathbf{z}}}{||\mathbf{w}||^2}\right)^{(N-2)/2} u_2^{N-2} \right)
\left( 1 +O \left( \frac{||\mathbf{w}||^2}{\delta_{\mathbf{w}, \mathbf{z}}} \right) \right) \mathrm{d} u_2 \\
& = 2 \frac{N-2}{N-1} V_{N-2} \frac{||\mathbf{w}||}{\delta_{\mathbf{w}, \mathbf{z}}^{1/2}}
\left( 1 + O \left( \frac{||\mathbf{w}||^2}{\delta_{\mathbf{w}, \mathbf{z}}} \right) \right),
\end{align*}
which completes the proof.
\end{proof}

Next, for $\mathbf{w}, \mathbf{z} \in \mathbb{R}^N$, we let
\begin{equation}
\label{Definition volume loc}
\mathcal{J}(\mathbf{w}, \mathbf{z}) = \vol \left( \left\{ \mathbf{t} \in \mathbb{R}^N :
|\langle \mathbf{w}, \mathbf{t} \rangle|, |\langle\mathbf{z}, \mathbf{t} \rangle| \leq ||\mathbf{t}|| \leq 1 \right\} \right).
\end{equation}
We shall establish the following estimate for this quantity.

\begin{lemma}
\label{Lemma volume loc}
Let $N \geq 3$ and let $\mathbf{w}, \mathbf{z} \in \mathbb{R}^N$ be two linearly independent vectors. Then 
\begin{equation*}
\mathcal{J}(\mathbf{w}, \mathbf{z}) = 4 \frac{N-2}{N} V_{N-2}
\frac1{\delta_{\mathbf{w},\mathbf{z}}^{1/2}} \left( 1 + O \left( \min \left\{ 1, \frac{(||\mathbf{w}||+ ||\mathbf{z}||)^2}{\delta_{\mathbf{w},\mathbf{z}}} \right\} \right) \right),
\end{equation*}
where the implied constant depends at most on $N$.
\end{lemma}

\begin{proof}
Let $(\mathbf{f}_3, \dots, \mathbf{f}_{N})$ be an orthonormal basis of
$\left( \mathbb{R} \mathbf{w} \oplus \mathbb{R} \mathbf{z} \right)^{\perp}$ and let $\mathbf{C}$ be the $N \times N$ matrix
\begin{equation*}
\mathbf{C} = 
\begin{pmatrix}
\dfrac{||\mathbf{z}||^2 \mathbf{w}-\langle \mathbf{w}, \mathbf{z} \rangle \mathbf{z}}{\delta_{\mathbf{w},\mathbf{z}}} & \dfrac{-\langle \mathbf{w}, \mathbf{z} \rangle \mathbf{w}+||\mathbf{w}||^2\mathbf{z}}{\delta_{\mathbf{w},\mathbf{z}}}
& \mathbf{f}_3 & \cdots & \mathbf{f}_{N}
\end{pmatrix}.
\end{equation*}
Letting $\boldsymbol{0}$ be the zero vector of size $N-2$ and $\mathbf{I}_{N-2}$ be the identity matrix of size $N-2$, we note that $\mathbf{C}^T \mathbf{C}$ is a block diagonal matrix given by
\begin{equation*}
\mathbf{C}^T \mathbf{C} = \frac1{\delta_{\mathbf{w},\mathbf{z}}}
\begin{pmatrix}
||\mathbf{z}||^2 & - \langle \mathbf{w}, \mathbf{z} \rangle & \boldsymbol{0}^T \\
- \langle \mathbf{w}, \mathbf{z} \rangle & ||\mathbf{w}||^2 & \boldsymbol{0}^T \\
\boldsymbol{0} & \boldsymbol{0} & \delta_{\mathbf{w},\mathbf{z}} \mathbf{I}_{N-2}
\end{pmatrix}.
\end{equation*}
This yields in particular $|\det(\mathbf{C})| = 1/\delta_{\mathbf{w},\mathbf{z}}^{1/2}$. Making the change of variables
$\mathbf{t} = \mathbf{C} \mathbf{u}$ and letting $\mathbf{v}$ be the vector whose coordinates are the $N-2$ final coordinates of $\mathbf{u}$, we find that
\begin{equation*}
\mathcal{J}(\mathbf{w}, \mathbf{z}) = \frac1{\delta_{\mathbf{w},\mathbf{z}}^{1/2}} \cdot
\vol \left( \left\{ (u_1,u_2,\mathbf{v}) \in \mathbb{R}^N :
\max \left\{ u_1^2, u_2^2 \right\} \leq q_{\mathbf{w},\mathbf{z}}(u_1,u_2) + ||\mathbf{v}||^2 \leq 1 \right\} \right),
\end{equation*}
where we have introduced the positive definite quadratic form
\begin{equation*}
q_{\mathbf{w},\mathbf{z}}(u_1,u_2) = \frac{||u_1 \mathbf{z} - u_2 \mathbf{w}||^2}{\delta_{\mathbf{w},\mathbf{z}}}.
\end{equation*}
We see that $\mathcal{J}(\mathbf{w}, \mathbf{z}) \ll 1/{\delta_{\mathbf{w},\mathbf{z}}^{1/2}}$, so if
$(||\mathbf{w}||+||\mathbf{z}||)^2 \geq \delta_{\mathbf{w},\mathbf{z}}$ then the claimed estimate holds. We now deal with the case where $(||\mathbf{w}||+||\mathbf{z}||)^2 < \delta_{\mathbf{w},\mathbf{z}}$. We first note that if
$(u_1, u_2) \in \mathbb{R}^2$ satisfies the conditions $u_1^2, u_2^2 \leq 1$ then
\begin{equation*}
q_{\mathbf{w},\mathbf{z}}(u_1,u_2) \leq \frac{(||\mathbf{w}||+||\mathbf{z}||)^2}{\delta_{\mathbf{w},\mathbf{z}}}.
\end{equation*}
We thus deduce that
\begin{equation}
\label{Inequalities J}
\frac{\mathcal{K}^{(1)}(\mathbf{w}, \mathbf{z})}{\delta_{\mathbf{w},\mathbf{z}}^{1/2}} \leq
\mathcal{J}(\mathbf{w}, \mathbf{z}) \leq
\frac{\mathcal{K}^{(2)}(\mathbf{w}, \mathbf{z})}{\delta_{\mathbf{w},\mathbf{z}}^{1/2}},
\end{equation}
where
\begin{equation*}
\mathcal{K}^{(1)}(\mathbf{w}, \mathbf{z}) = \vol \left( \left\{ (u_1,u_2,\mathbf{v}) \in \mathbb{R}^N :
\max \left\{ u_1^2, u_2^2 \right\} \leq ||\mathbf{v}||^2 \leq 1 - \frac{(||\mathbf{w}||+||\mathbf{z}||)^2}{\delta_{\mathbf{w},\mathbf{z}}} \right\} \right),
\end{equation*}
and
\begin{equation*}
\mathcal{K}^{(2)}(\mathbf{w}, \mathbf{z}) = \vol \left( \left\{ (u_1,u_2,\mathbf{v}) \in \mathbb{R}^N :
\max \left\{u_1^2, u_2^2 \right\} - \frac{(||\mathbf{w}||+||\mathbf{z}||)^2}{\delta_{\mathbf{w},\mathbf{z}}} \leq ||\mathbf{v}||^2 \leq 1 \right\} \right).
\end{equation*}
Integrating over $u_1$ and $u_2$ we easily obtain that for any $i \in \{1, 2\}$ we have
\begin{equation*}
\mathcal{K}^{(i)}(\mathbf{w}, \mathbf{z}) =
4 \left( \int_{\mathcal{B}_{N-2}(1)} ||\mathbf{v}||^2 \mathrm{d} \mathbf{v} \right)
\left( 1 + O \left( \frac{(||\mathbf{w}||+||\mathbf{z}||)^2}{\delta_{\mathbf{w},\mathbf{z}}} \right) \right).
\end{equation*}
Recalling the inequalities \eqref{Inequalities J} we see that
\begin{equation*}
\mathcal{J}(\mathbf{w}, \mathbf{z}) =
\frac{4}{\delta_{\mathbf{w},\mathbf{z}}^{1/2}} \left( \int_{\mathcal{B}_{N-2}(1)} ||\mathbf{v}||^2 \mathrm{d} \mathbf{v} \right) \left( 1 + O \left( \frac{(||\mathbf{w}||+||\mathbf{z}||)^2}{\delta_{\mathbf{w},\mathbf{z}}} \right) \right).
\end{equation*}
Finally, a straightforward calculation involving spherical coordinates shows that
\begin{equation*}
\int_{\mathcal{B}_{N-2}(1)} ||\mathbf{v}||^2 \mathrm{d} \mathbf{v} = \frac{N-2}{N} V_{N-2},
\end{equation*}
which completes the proof.
\end{proof}
 
\subsection{Inverses of lattice determinants on average}

\label{Section inverses determinants}

Let $\mathbf{x}, \mathbf{y} \in \mathbb{Z}_{\mathrm{prim}}^{n+1}$ be two linearly independent vectors. The lattice
$\Lambda_{\nu_{d,n}(\mathbf{x})} \cap \Lambda_{\nu_{d,n}(\mathbf{y})}$ features heavily in our work and we shall need to be able to control the inverse of its determinant on average. Recall that the quantities
$\mathcal{G}(\mathbf{x},\mathbf{y})$ and $\mathfrak{d}_2(\mathbf{x},\mathbf{y})$ were respectively introduced in Definitions~\ref{Definition G} and \ref{Definition d}. It follows from Lemmas~\ref{Lemma determinant global} and \ref{Lemma G} that
\begin{equation}
\label{Identity determinant}
\det(\Lambda_{\nu_{d,n}(\mathbf{x})}\cap \Lambda_{\nu_{d,n}(\mathbf{y})})=
\frac{\det\left(\mathbb{Z}\nu_{d,n}(\mathbf{x})\oplus \mathbb{Z}\nu_{d,n}(\mathbf{y})\right)}{\mathcal{G}(\mathbf{x},\mathbf{y})}.
\end{equation}
We start by proving the following pointwise bounds.

\begin{lemma}
\label{Lemma upper and lower bounds det}
Let $d,n \geq 2$ and let $\mathbf{x}, \mathbf{y} \in \mathbb{Z}_{\mathrm{prim}}^{n+1}$ be two linearly independent vectors. We have
\begin{equation*} 
\mathfrak{d}_2(\mathbf{x},\mathbf{y}) \cdot ||\mathbf{x}||^{d-1} ||\mathbf{y}||^{d-1} \ll
\det(\Lambda_{\nu_{d,n}(\mathbf{x})}\cap \Lambda_{\nu_{d,n}(\mathbf{y})}) \ll ||\mathbf{x}||^d ||\mathbf{y}||^d.
\end{equation*}
\end{lemma}

\begin{proof}
Since $\mathcal{G}(\mathbf{x},\mathbf{y})\geq 1$, the equality \eqref{Identity determinant} implies that
\begin{equation*}
\det(\Lambda_{\nu_{d,n}(\mathbf{x})}\cap \Lambda_{\nu_{d,n}(\mathbf{y})}) \leq
||\nu_{d,n}(\mathbf{x})|| \cdot ||\nu_{d,n}(\mathbf{y})||,
\end{equation*}
and the claimed upper bound follows.

Recall that the set $\mathscr{M}_{d,n}$ of monomials of degree $d$ in $n+1$ variables was introduced in
Definition \ref{Definition Polynomials}. By the definition \eqref{Definition det} of the determinant of a lattice we have
\begin{equation}
\label{Equality lattice}
\det( \mathbb{Z} \nu_{d,n}(\mathbf{x}) \oplus \mathbb{Z} \nu_{d,n}(\mathbf{y}) )^2 =
||\nu_{d,n}(\mathbf{x})||^2 ||\nu_{d,n}(\mathbf{y})||^2 - \langle \nu_{d,n}(\mathbf{x}), \nu_{d,n}(\mathbf{y}) \rangle^2,
\end{equation}
which can be rewritten as
\begin{equation*}
\det( \mathbb{Z} \nu_{d,n}(\mathbf{x}) \oplus \mathbb{Z} \nu_{d,n}(\mathbf{y}) )^2 = \frac1{2}
\sum_{P_1,P_2 \in \mathscr{M}_{d,n}} \left( P_1(\mathbf{x}) P_2(\mathbf{y}) - P_2(\mathbf{x}) P_1(\mathbf{y}) \right)^2.
\end{equation*}
We thus see that
\begin{equation*}
\det( \mathbb{Z} \nu_{d,n}(\mathbf{x}) \oplus \mathbb{Z} \nu_{d,n}(\mathbf{y}) )^2 \geq
\frac1{4} \sum_{i_1, j_1, i_2, j_2 = 0}^n
\left(x_{i_1} x_{j_1}^{d-1} y_{i_2} y_{j_2}^{d-1} - x_{i_2} x_{j_2}^{d-1} y_{i_1} y_{j_1}^{d-1} \right)^2.
\end{equation*}
It follows that
\begin{equation*}
\det( \mathbb{Z} \nu_{d,n}(\mathbf{x}) \oplus \mathbb{Z} \nu_{d,n}(\mathbf{y}) )^2 \geq
\frac1{2} \left( ||\mathbf{x}||^2 ||\mathbf{y}||^2 ||\psi_d(\mathbf{x})||^2 ||\psi_d(\mathbf{y})||^2 -
\langle \mathbf{x}, \mathbf{y} \rangle^2 \langle \psi_d(\mathbf{x}), \psi_d(\mathbf{y}) \rangle^2 \right),
\end{equation*}
where the map $\psi_d : \mathbb{R}^{n+1} \to \mathbb{R}^{n+1}$ is defined for $\mathbf{z} \in \mathbb{R}^{n+1}$ by
\begin{equation*}
\psi_d(\mathbf{z}) = \left(z_0^{d-1}, \dots, z_n^{d-1} \right).
\end{equation*}
The Cauchy--Schwarz inequality
\begin{equation*}
\langle \psi_d(\mathbf{x}), \psi_d(\mathbf{y}) \rangle^2 \leq ||\psi_d(\mathbf{x})||^2 ||\psi_d(\mathbf{y})||^2
\end{equation*}
and the lower bounds $||\psi_d(\mathbf{x})|| \gg ||\mathbf{x}||^{d-1}$ and $||\psi_d(\mathbf{y})|| \gg ||\mathbf{y}||^{d-1}$ eventually yield
\begin{equation*}
\det( \mathbb{Z} \nu_{d,n}(\mathbf{x}) \oplus \mathbb{Z} \nu_{d,n}(\mathbf{y}) )^2 \gg
(||\mathbf{x}||^2 ||\mathbf{y}||^2 - \langle \mathbf{x}, \mathbf{y} \rangle^2) \cdot
||\mathbf{x}||^{2(d-1)} ||\mathbf{y}||^{2(d-1)}.
\end{equation*}
An application of Lemma~\ref{Lemma d2(x,y)} completes the proof.
\end{proof}

We define
\begin{equation}
\label{Definition E}
E_{d,n}(B) = \sum_{(\mathbf{x}, \mathbf{y}) \in \Omega_{d,n}(B)} 
\frac1{\det ( \Lambda_{ \nu_{d,n}(\mathbf{x})}\cap \Lambda_{\nu_{d,n}(\mathbf{y})})},
\end{equation}
where
\begin{equation}
\label{Definition Omega}
\Omega_{d,n}(B) =
\left\{ (\mathbf{x}, \mathbf{y}) \in \mathbb{Z}_{\mathrm{prim}}^{n+1} \times \mathbb{Z}_{\mathrm{prim}}^{n+1} :
\begin{array}{l l}
||\mathbf{x}||, ||\mathbf{y}|| \leq B^{1/(n+1-d)} \\ 
\mathbf{x} \neq \pm \mathbf{y}
\end{array}
\right\},
\end{equation}
and we shall prove the following sharp upper and lower bounds.

\begin{lemma}
\label{Lemma E}
Let $d \geq 2$ and $n \geq d$ with $(d,n) \neq (2,2)$. We have
\begin{equation*}
B^2 \ll E_{d,n}(B) \ll B^2.
\end{equation*}
\end{lemma}

\begin{proof}
The upper bound in Lemma~\ref{Lemma upper and lower bounds det} yields
\begin{equation*}
E_{d,n}(B) \gg \sum_{(\mathbf{x}, \mathbf{y}) \in \Omega_{d,n}(B)} \frac1{||\mathbf{x}||^d ||\mathbf{y}||^d},
\end{equation*}
which immediately gives $E_{d,n}(B) \gg B^2$. Recall the definition \eqref{Definition l_r(X,Y)} of the quantity
$\ell_{2,n}(X,Y;\Delta_2)$. We proceed to break the sizes of $||\mathbf{x}||$, $||\mathbf{y}||$ and
$\mathfrak{d}_2(\mathbf{x},\mathbf{y})$ into dyadic intervals. Recalling that we have the upper bound
\eqref{Upper bound trivial drxy} and using the lower bound in Lemma~\ref{Lemma upper and lower bounds det} we get
\begin{equation*}
E_{d,n}(B) \ll \sum_{X, Y \ll B^{1/(n+1-d)}} \ \sum_{\Delta_2 \ll XY} \frac1{\Delta_2 (XY)^{d-1}} \cdot
\ell_{2,n}(X,Y;\Delta_2).
\end{equation*}
Applying Lemma~\ref{Lemma l_r(X,Y)} we deduce that
\begin{align*}
E_{d,n}(B) & \ll \sum_{X, Y \ll B^{1/(n+1-d)}} \frac1{(XY)^{d-3}} \sum_{\Delta_2 \ll XY} \Delta_2^{n-2} \\
& \ll \sum_{X, Y \ll B^{1/(n+1-d)}} (XY)^{n+1-d},
\end{align*}
since $n \geq 3$. The upper bound $E_{d,n}(B) \ll B^2$ follows, which completes the proof.
\end{proof}

Recall the respective definitions \eqref{Definition alpha}, \eqref{Definition W} and \eqref{Definition w} of $\alpha$, $W$ and $w$. We let $\rad(W)$ denote the radical of the integer $W$, that is
\begin{equation}
\label{Definition radical}
\rad(W) = \prod_{p \leq w} p.
\end{equation}
Given two linearly independent vectors $\mathbf{x}, \mathbf{y} \in \mathbb{Z}^{n+1}$ we put
\begin{equation}
\label{Definition Delta}
\Delta(\mathbf{x}, \mathbf{y}) = \frac{||\nu_{d,n}(\mathbf{x})|| \cdot ||\nu_{d,n}(\mathbf{y})||}
{\det ( \mathbb{Z} \nu_{d,n}(\mathbf{x}) \oplus \mathbb{Z} \nu_{d,n}(\mathbf{y}))},
\end{equation}
and
\begin{equation}
\label{Definition localised error}
\mathcal{E}_{\mathbf{x},\mathbf{y}} (B) = \min \left\{ 1, \frac{\Delta(\mathbf{x}, \mathbf{y})^2}{\alpha^2} \right\} + \boldsymbol{1}_{\mathcal{G}(\mathbf{x}, \mathbf{y}) \nmid W/\rad(W)}.
\end{equation}
Bearing this in mind, we let
\begin{equation}
\label{Definition F}
F_{d,n}(B) = \sum_{(\mathbf{x}, \mathbf{y}) \in \Omega_{d,n}(B)}
\frac{\mathcal{E}_{\mathbf{x},\mathbf{y}}(B)}
{\det ( \Lambda_{ \nu_{d,n}(\mathbf{x})}\cap \Lambda_{\nu_{d,n}(\mathbf{y})})},
\end{equation}
and we shall seek a saving over the trivial upper bound $F_{d,n}(B) \ll B^2$ that follows from taking
$\mathcal{E}_{\mathbf{x},\mathbf{y}}(B) \leq 2$ and applying Lemma~\ref{Lemma E}.

\begin{lemma}
\label{Lemma F}
Let $d \geq 2$ and $n \geq d$ with $(d,n) \neq (2,2)$. We have
\begin{equation*}
F_{d,n}(B) \ll \frac{B^2}{(\log B)^{1/2}}.
\end{equation*}
\end{lemma}

\begin{proof}
We let 
\begin{equation*}
F^{(1)}_{d,n}(B) = \sum_{(\mathbf{x}, \mathbf{y}) \in \Omega_{d,n}(B)} \frac1
{\det ( \Lambda_{ \nu_{d,n}(\mathbf{x})}\cap \Lambda_{\nu_{d,n}(\mathbf{y})})} \cdot
\min \left\{ 1, \frac{\Delta(\mathbf{x}, \mathbf{y})^2}{\alpha^2} \right\},
\end{equation*}
and
\begin{equation*}
F^{(2)}_{d,n}(B) = \sum_{(\mathbf{x}, \mathbf{y}) \in \Omega_{d,n}(B)} 
\frac1{\det ( \Lambda_{ \nu_{d,n}(\mathbf{x})}\cap \Lambda_{\nu_{d,n}(\mathbf{y})})} \cdot
\boldsymbol{1}_{\mathcal{G}(\mathbf{x}, \mathbf{y}) \nmid W/\rad(W)},
\end{equation*}
so that
\begin{equation}
\label{Equality F}
F_{d,n}(B) = F^{(1)}_{d,n}(B) + F^{(2)}_{d,n}(B).
\end{equation}

We start by proving an upper bound for the sum $F^{(1)}_{d,n}(B)$. It follows from the equality
\eqref{Identity determinant} and the lower bound in Lemma~\ref{Lemma upper and lower bounds det} that for any linearly independent vectors $\mathbf{x}, \mathbf{y} \in \mathbb{Z}_{\mathrm{prim}}^{n+1}$, we have
\begin{equation*}
\Delta(\mathbf{x}, \mathbf{y}) \ll \frac{||\mathbf{x}|| \cdot ||\mathbf{y}||}{\mathfrak{d}_2(\mathbf{x},\mathbf{y})}.
\end{equation*}
Recall the definition \eqref{Definition l_r(X,Y)} of the quantity $\ell_{2,n}(X,Y;\Delta_2)$. Breaking the sizes of
$||\mathbf{x}||$, $||\mathbf{y}||$ and $\mathfrak{d}_2(\mathbf{x},\mathbf{y})$ into dyadic intervals and using again Lemma~\ref{Lemma upper and lower bounds det} we deduce that
\begin{equation*}
F^{(1)}_{d,n}(B) \ll \sum_{X, Y \ll B^{1/(n+1-d)}} \ \sum_{\Delta_2 \ll XY} \frac1{\Delta_2 (XY)^{d-1}}
\left( \min \left\{ 1, \frac{(XY)^2}{\Delta_2^2\alpha^2} \right\} \right) \ell_{2,n}(X,Y;\Delta_2).
\end{equation*}
Writing that
\begin{equation*}
\min \left\{ 1, \frac{(XY)^2}{\Delta_2^2\alpha^2} \right\} \leq \frac{(XY)^{1/2}}{\Delta_2^{1/2}\alpha^{1/2}},
\end{equation*}
appealing to Lemma~\ref{Lemma l_r(X,Y)} and using the fact that $n \geq 3$, we obtain
\begin{align}
\nonumber
F^{(1)}_{d,n}(B) & \ll \frac1{\alpha^{1/2}} \sum_{X, Y \ll B^{1/(n+1-d)}} \frac1{(XY)^{d-7/2}} \sum_{\Delta_2 \ll XY}
\Delta_2^{n-5/2} \\
\nonumber
& \ll \frac1{\alpha^{1/2}} \sum_{X, Y \ll B^{1/(n+1-d)}} (XY)^{n+1-d} \\
\label{Upper bound F^1}
& \ll \frac{B^2}{\alpha^{1/2}}.
\end{align}

We now consider the sum $F^{(2)}_{d,n}(B)$. For any prime number $p$ and any $m \geq 1$, we let $v_p(m)$ denote the $p$-adic valuation of $m$. Suppose that $m$ is a positive integer which does not divide $W/\rad(W)$. In view of the definition \eqref{Definition W} of $W$ this means that either there is a prime $p>w$ which divides $m$, or else there is a prime $p \leq w$ such that 
\begin{equation*}
v_p(m) > v_p\left(\frac{W}{\rad(W)}\right) = \left\lceil \frac{\log w}{\log p}\right\rceil \geq \frac{\log w}{\log p}.
\end{equation*}
In either case we deduce that $m>w$. As a result, Lemma~\ref{Lemma d2(x,y)} shows that for any linearly independent vectors $\mathbf{x}, \mathbf{y} \in \mathbb{Z}_{\mathrm{prim}}^{n+1}$ such that
$\mathcal{G}(\mathbf{x}, \mathbf{y}) \nmid W/\rad(W)$, we have
\begin{equation*}
\mathfrak{d}_2(\mathbf{x},\mathbf{y}) \leq \frac{||\mathbf{x}|| \cdot ||\mathbf{y}||}{w}.
\end{equation*}
Breaking the sizes of $||\mathbf{x}||$, $||\mathbf{y}||$ and $\mathfrak{d}_2(\mathbf{x},\mathbf{y})$ into dyadic intervals and using Lemma~\ref{Lemma upper and lower bounds det}, we thus see that
\begin{equation*}
F^{(2)}_{d,n}(B) \ll \sum_{X, Y \ll B^{1/(n+1-d)}} \ \sum_{\Delta_2 \ll XY/w} \frac1{\Delta_2 (XY)^{d-1}} \cdot
\ell_{2,n}(X,Y;\Delta_2).
\end{equation*}
Applying Lemma~\ref{Lemma l_r(X,Y)} and using the fact that $n \geq 3$ we derive
\begin{align}
\nonumber
F^{(2)}_{d,n}(B) & \ll \sum_{X, Y \ll B^{1/(n+1-d)}} \frac1{(XY)^{d-3}} \sum_{\Delta_2 \ll XY/w} \Delta_2^{n-2} \\
\nonumber
& \ll \frac1{w^{n-2}} \sum_{X, Y \ll B^{1/(n+1-d)}} (XY)^{n+1-d} \\
\label{Upper bound F^2}
& \ll \frac{B^2}{w^{n-2}}.
\end{align}
Recalling the respective definitions \eqref{Definition alpha} and \eqref{Definition w} of $\alpha$ and $w$, we see that combining the equality \eqref{Equality F} and the upper bounds \eqref{Upper bound F^1} and \eqref{Upper bound F^2} completes the proof.
\end{proof}

\subsection{A first moment bound}

\label{Section first}

Recall the definition \eqref{Definition counting} of the counting function $N_V(B)$. A result of the second author \cite[Theorem~$3$]{RandomFano} implies in particular that for $A \geq B^{1/(n+1-d)}$, we have
\begin{equation*}
\frac1{\# \mathbb{V}_{d,n}(A)} \sum_{V \in \mathbb{V}_{d,n}(A)} N_V(B) \ll \frac{B}{A}.
\end{equation*}
Unfortunately, we will require the allowable range of $A$ to be greater in the critical case $n=d$. We shall achieve this by using our work from Section~\ref{Section geometry of numbers}. In fact, although we elect not to do so here, it would be straightforward to obtain an asymptotic formula that improves upon \cite[Theorem~$3$]{RandomFano}.

Given an integer $R \geq 1$ and a lattice $\Lambda$ of rank $R$, recall that the successive minima
$\lambda_1(\Lambda), \dots, \lambda_R(\Lambda)$ of $\Lambda$ were introduced in
Definition~\ref{Definition successive}. We will employ this notation without further notice throughout
Sections~\ref{Section first} and \ref{Section second}.

\begin{lemma}
\label{Lemma first moment}
Let $d \geq 2$ and $n \geq d$. Assume that $A \geq B^{4/5}$. Then we have 
\begin{equation*}
\frac1{\# \mathbb{V}_{d,n}(A)} \sum_{V \in \mathbb{V}_{d,n}(A)} N_V(B) \ll \frac{B}{A}.
\end{equation*}
\end{lemma}

\begin{proof}
To begin with we clearly have 
\begin{equation*}
\sum_{V \in \mathbb{V}_{d,n}(A)} N_V(B) = \sum_{\substack{x \in \mathbb{P}^n(\mathbb{Q})\\ H(x)\leq B}}
\# \{ V \in \mathbb{V}_{d,n}(A) : x \in V(\mathbb{Q}) \}.
\end{equation*}
Recall the definition \eqref{Definition Xi} of the set $\Xi_{d,n}(B)$. We see that
\begin{equation*}
\sum_{V \in \mathbb{V}_{d,n}(A)} N_V(B) = \frac1{4} \sum_{\mathbf{x} \in \Xi_{d,n}(B)}
\# \left(\Lambda_{\nu_{d,n}(\mathbf{x})} \cap \mathbb{Z}_{\mathrm{prim}}^{N_{d,n}} \cap \mathcal{B}_{N_{d,n}}(A) \right).
\end{equation*}
Recall that the quantity $\mathfrak{d}_2(\mathbf{x})$ was introduced in Definition \ref{Definition d}. For
$\mathbf{x} \in \mathbb{Z}_{\mathrm{prim}}^{n+1}$, we define
\begin{equation}
\label{Definition mu(x)}
\mu(\mathbf{x}) = n \frac{||\mathbf{x}||}{\mathfrak{d}_2(\mathbf{x})},
\end{equation}
and we note that Lemma~\ref{Key lemma one vector} states that
\begin{equation}
\label{Replacement key lemma one vector}
\lambda_{N_{d,n}-1} \left(\Lambda_{\nu_{d,n}(\mathbf{x})}\right) \leq \mu(\mathbf{x}).
\end{equation}
It is convenient to set
\begin{equation*}
M_{d,n}^{(1)}(A,B) =
\frac1{4} \sum_{\substack{\mathbf{x} \in \Xi_{d,n}(B) \\ \mu(\mathbf{x}) \leq A}}
\# \left(\Lambda_{\nu_{d,n}(\mathbf{x})} \cap \mathbb{Z}_{\mathrm{prim}}^{N_{d,n}} \cap
\mathcal{B}_{N_{d,n}}(A) \right),
\end{equation*}
and
\begin{equation}
\label{Definition M2}
M_{d,n}^{(2)}(A,B) = \sum_{V \in \mathbb{V}_{d,n}(A)} N_V(B) - M_{d,n}^{(1)}(A,B).
\end{equation}

We first deal with the sum $M_{d,n}^{(1)}(A,B)$. We have the inequalities \eqref{Replacement key lemma one vector} and $A \geq \mu(\mathbf{x})$ so we are in position to apply Lemma~\ref{Lemma lattice pivotal} with $I=1$ and
$\gamma=1/2$, say. This allows us to derive the upper bound 
\begin{equation*}
\# \left(\Lambda_{\nu_{d,n}(\mathbf{x})} \cap \mathbb{Z}_{\mathrm{prim}}^{N_{d,n}} \cap
\mathcal{B}_{N_{d,n}}(A) \right) \ll \frac{A^{N_{d,n}-1}}{\det(\Lambda_{\nu_{d,n}(\mathbf{x})})}.
\end{equation*}
Moreover Lemma~\ref{Lemma determinant global} gives
$\det(\Lambda_{\nu_{d,n}(\mathbf{x})}) = ||\nu_{d,n}(\mathbf{x})|| \gg ||\mathbf{x}||^d$ so we immediately deduce that
\begin{equation}
\label{Upper bound M1}
M_{d,n}^{(1)}(A,B) \ll A^{N_{d,n}-1} B.
\end{equation}

We now handle the sum $M_{d,n}^{(2)}(A,B)$. We have the inequalities \eqref{Replacement key lemma one vector} and $A < \mu(\mathbf{x})$ so we can apply the first part of Lemma~\ref{Lemma R0} with $M=1/2$ and $R_0=n-1$. This yields
\begin{equation*}
\# \left(\Lambda_{\nu_{d,n}(\mathbf{x})} \cap \mathbb{Z}_{\mathrm{prim}}^{N_{d,n}} \cap
\mathcal{B}_{N_{d,n}}(A) \right) \ll \frac{A^{N_{d,n}-n}}{||\mathbf{x}||^d}
\left( \frac{||\mathbf{x}||}{\mathfrak{d}_2(\mathbf{x})} \right)^{n-1} + A^{N_{d,n}-n-1}.
\end{equation*}
Recall the definition \eqref{Definition l_r(X)} of the quantity $\ell_{2,n}(X;\Delta_0)$. Breaking the sizes of
$||\mathbf{x}||$ and $\mathfrak{d}_2(\mathbf{x})$ into dyadic intervals we see that we have
\begin{equation*}
M_{d,n}^{(2)}(A,B) \ll A^{N_{d,n}-1} \sum_{X \ll B^{1/(n+1-d)}} \ \sum_{\Delta_0 \ll X/A} 
\left( \frac{X^{n-d-1}}{A^{n-1} \Delta_0^{n-1}} + \frac1{A^n} \right) \ell_{2,n}(X;\Delta_0).
\end{equation*}
It follows from Lemma~\ref{Lemma l_r(X)} that
\begin{align*}
M_{d,n}^{(2)}(A,B) & \ll A^{N_{d,n}-1} (\log B) \sum_{X \ll B^{1/(n+1-d)}} \ \sum_{\Delta_0 \ll X/A} 
\left( \frac{X^{n+1-d} \Delta_0}{A^{n-1}} + \frac{X^2 \Delta_0^n}{A^n} \right) \\
& \ll A^{N_{d,n}-1} (\log B) \sum_{X \ll B^{1/(n+1-d)}} \left( \frac{X^{n+2-d}}{A^n} + \frac{X^{n+2}}{A^{2n}} \right) \\
& \ll A^{N_{d,n}-1} \left( \frac{B^{1+1/(n+1-d)}}{A^n} + \frac{B^{(n+2)/(n+1-d)}}{A^{2n}} \right) \log B.
\end{align*}
As a result, we see that the assumption $A \geq B^{4/5}$ implies in particular that for any $d \geq 2$ and $n \geq d$ we have
\begin{equation}
\label{Upper bound M2}
M_{d,n}^{(2)}(A,B) \ll A^{N_{d,n}-1} B.
\end{equation}
We complete the proof by putting together the equality \eqref{Definition M2} and the upper bounds
\eqref{Upper bound M1} and \eqref{Upper bound M2}, and by using the lower bound \eqref{Lower bound Vdn}.
\end{proof}

\subsection{Second moment estimates}

\label{Section second}

We now turn to the proof of three second moment estimates. We start by setting some notation. Recall the expression \eqref{Equality counting} for the global counting function $N_V(B)$ and the definition
\eqref{Definition local counting} of the localised counting function $N_V^{\mathrm{loc}}(B)$. We introduce the second moment of $N_V(B)$ with its diagonal contribution removed, that is
\begin{equation*}
D_{d,n}(A,B) = \sum_{ V \in \mathbb{V}_{d,n}(A)} N_V(B)^2 - \sum_{ V \in \mathbb{V}_{d,n}(A)} N_V(B).
\end{equation*}
Similarly, we also define the mixed moment and the second moment of $N_V^{\mathrm{loc}}(B)$ with their respective diagonal contributions removed. We thus set
\begin{equation}
\label{Definition D mixed}
D_{d,n}^{\mathrm{mix}}(A,B) = \sum_{V \in \mathbb{V}_{d,n}(A)} N_V(B) N_V^{\mathrm{loc}}(B) -
\sum_{V \in \mathbb{V}_{d,n}(A)} \Delta_V^{\mathrm{mix}}(B),
\end{equation}
where
\begin{equation}
\label{Definition Delta mixed}
\Delta_V^{\mathrm{mix}}(B) = \frac1{2} \cdot \frac{\alpha W}{||\mathbf{a}_V||}
\sum_{\substack{\mathbf{x} \in \Xi_{d,n}(B) \\
\mathbf{a}_V \in \Lambda_{\nu_{d,n}(\mathbf{x})}}}
\frac1{||\nu_{d,n}(\mathbf{x})||},
\end{equation}
and
\begin{equation}
\label{Definition D loc}
D_{d,n}^{\mathrm{loc}}(A,B) = \sum_{V \in \mathbb{V}_{d,n}(A)} N_V^{\mathrm{loc}}(B)^2 -
\sum_{V \in \mathbb{V}_{d,n}(A)} \Delta_V^{\mathrm{loc}}(B),
\end{equation}
where
\begin{equation}
\label{Definition Delta loc}
\Delta_V^{\mathrm{loc}}(B)= \frac1{2} \cdot \frac{\alpha^2 W^2}{||\mathbf{a}_V||^2}
\sum_{\substack{\mathbf{x} \in \Xi_{d,n}(B) \\
\mathbf{a}_V \in \Lambda_{\nu_{d,n}(\mathbf{x})}^{(W)} \cap \mathcal{C}_{\nu_{d,n}(\mathbf{x})}^{(\alpha)}}}
\frac1{||\nu_{d,n}(\mathbf{x})||^2}.
\end{equation}
We shall use our work in Section~\ref{Section geometry of numbers} to establish asymptotic formulae for the quantities $D_{d,n}(A,B)$, $D_{d,n}^{\mathrm{mix}}(A,B)$ and $D_{d,n}^{\mathrm{loc}}(A,B)$.

For any lattice $\Lambda \subset \mathbb{Z}^{N_{d,n}}$, any bounded region
$\mathcal{R} \subset \mathbb{R}^{N_{d,n}}$ and any integer $k \geq 0$, it is convenient to set
\begin{equation}
\label{Definition Sk star}
\mathcal{S}_k^{\ast} \left(\Lambda; \mathcal{R} \right) =
\sum_{\mathbf{a} \in \Lambda \cap \mathbb{Z}^{N_{d,n}}_{\mathrm{prim}} \cap \mathcal{R}} \frac1{||\mathbf{a}||^k},
\end{equation}
and
\begin{equation}
\label{Definition Sk}
\mathcal{S}_k \left(\Lambda; \mathcal{R} \right) =
\sum_{\mathbf{a} \in (\Lambda \smallsetminus \boldsymbol{0}) \cap \mathcal{R}} \frac1{||\mathbf{a}||^k}.
\end{equation}
In addition, we recall that $V_N$ denotes the volume of the unit ball $\mathcal{B}_N(1)$ in $\mathbb{R}^N$ for any
$N \geq 1$. Moreover, for any $d,n \geq 2$ we set
\begin{equation}
\label{Definition iota}
\iota_{d,n} = \frac{V_{N_{d,n}-2}}{8 \zeta(N_{d,n}-2)}.
\end{equation}
Finally, recall the definition \eqref{Definition E} of the quantity $E_{d,n}(B)$.

We start by handling the quantity $D_{d,n}(A,B)$ in the case $(d,n) \neq (4,4)$.

\begin{lemma} 
\label{Lemma D}
Let $d \geq 2$ and $n \geq d$ with $(d,n) \notin \{(2,2), (3,3), (4,4)\}$. Assume that $A \geq B/(\log B)$. Then we have 
\begin{equation*}
D_{d,n}(A,B) = \iota_{d,n} A^{N_{d,n}-2} E_{d,n}(B) \left( 1 + O \left( \frac{(\log A)^{9/2}}{A^{1/2}} \right) \right).
\end{equation*}
\end{lemma}

\begin{proof}
Throughout the proof we allow $(d,n)=(4,4)$ unless stated otherwise. Recall the definition \eqref{Definition lattice} of the lattice $\Lambda_{\nu_{d,n}(\mathbf{x})}$. We introduce the lattice
\begin{equation*}
\Gamma_{\mathbf{x}, \mathbf{y}} = \Lambda_{\nu_{d,n}(\mathbf{x})} \cap \Lambda_{\nu_{d,n}(\mathbf{y})}.
\end{equation*}
Recall the definition \eqref{Definition Omega} of the set $\Omega_{d,n}(B)$. We start by noting that
\begin{align*}
D_{d,n}(A,B) & = \sum_{\substack{x,y \in \mathbb{P}^n(\mathbb{Q}) \\ x \neq y \\ H(x), H(y) \leq B}}
\# \{ V \in \mathbb{V}_{d,n}(A) : x, y \in V(\mathbb{Q}) \} \\
& = \frac1{8} \sum_{(\mathbf{x}, \mathbf{y}) \in \Omega_{d,n}(B)}
\mathcal{S}^{\ast}_0(\Gamma_{\mathbf{x}, \mathbf{y}}; \mathcal{B}_{N_{d,n}}(A)).
\end{align*}
Recall that the quantities $\mathfrak{d}_2(\mathbf{x},\mathbf{y})$ and $\mathfrak{d}_3(\mathbf{x},\mathbf{y})$ were introduced in Definition \ref{Definition d}. Given two linearly independent vectors
$\mathbf{x}, \mathbf{y} \in \mathbb{Z}_{\mathrm{prim}}^{n+1}$, we define
\begin{equation}
\label{Definition mu}
\mu(\mathbf{x},\mathbf{y}) = 3 n^2
\max \left\{ \frac{\mathfrak{d}_2(\mathbf{x},\mathbf{y})}{\mathfrak{d}_3(\mathbf{x},\mathbf{y})} , 
\frac{||\mathbf{x}|| \cdot ||\mathbf{y}||}{\mathfrak{d}_2(\mathbf{x},\mathbf{y})^2} \right\},
\end{equation}
and we note that Lemma~\ref{Key lemma two vectors} states that
\begin{equation}
\label{Replacement key lemma two vectors}
\lambda_{N_{d,n}-2} \left( \Gamma_{\mathbf{x}, \mathbf{y}} \right) \leq \mu(\mathbf{x},\mathbf{y}).
\end{equation}
We then set
\begin{equation*}
\Sigma_{d,n}^{(1)}(A,B) = \frac1{8}
\sum_{\substack{(\mathbf{x}, \mathbf{y}) \in \Omega_{d,n}(B) \\ \mu(\mathbf{x},\mathbf{y}) \leq A}}
\mathcal{S}^{\ast}_0(\Gamma_{\mathbf{x}, \mathbf{y}}; \mathcal{B}_{N_{d,n}}(A)),
\end{equation*}
and
\begin{equation}
\label{Definition Sigma2(A,B)}
\Sigma_{d,n}^{(2)}(A,B) = D_{d,n}(A,B) - \Sigma_{d,n}^{(1)}(A,B).
\end{equation}

We first handle the sum $\Sigma_{d,n}^{(1)}(A,B)$. The lattice $\Gamma_{\mathbf{x}, \mathbf{y}}$ is primitive and we have the inequalities \eqref{Replacement key lemma two vectors} and $A \geq \mu(\mathbf{x},\mathbf{y})$ so we are in position to apply \cite[Lemma~$3$]{RandomFano}. We deduce that
\begin{equation*}
\mathcal{S}^{\ast}_0(\Gamma_{\mathbf{x}, \mathbf{y}}; \mathcal{B}_{N_{d,n}}(A)) =
8 \iota_{d,n} \frac{A^{N_{d,n}-2}}{\det ( \Gamma_{\mathbf{x}, \mathbf{y}} )}
\left( 1 + O \left(\frac{\mu(\mathbf{x},\mathbf{y})}{A } \right) \right) + O(A \log B).
\end{equation*}
Using the upper bound in Lemma~\ref{Lemma upper and lower bounds det}, the trivial lower bound
$\mu(\mathbf{x},\mathbf{y}) \geq 1$, and the assumption $B \ll A \log A$ it is easy to check that for any
$(\mathbf{x}, \mathbf{y}) \in \Omega_{d,n}(B)$ we have
\begin{equation*}
\frac{A^{N_{d,n}-3} \mu(\mathbf{x},\mathbf{y})}{\det ( \Gamma_{\mathbf{x}, \mathbf{y}} )} \gg A \log B.
\end{equation*}
We thus obtain
\begin{equation*}
\Sigma_{d,n}^{(1)}(A,B) = \iota_{d,n} A^{N_{d,n}-2}
\sum_{\substack{(\mathbf{x}, \mathbf{y}) \in \Omega_{d,n}(B) \\ \mu(\mathbf{x},\mathbf{y}) \leq A}}
\frac1{\det (\Gamma_{\mathbf{x}, \mathbf{y}})} \left( 1 + O \left(\frac{\mu(\mathbf{x},\mathbf{y})}{A} \right) \right).
\end{equation*}
Note that we have the obvious inequality
\begin{equation*}
\sum_{\substack{(\mathbf{x}, \mathbf{y}) \in \Omega_{d,n}(B) \\ \mu(\mathbf{x},\mathbf{y}) > A}}
\frac1{\det (\Gamma_{\mathbf{x}, \mathbf{y}})} \leq
\frac1{A} \sum_{\substack{(\mathbf{x}, \mathbf{y}) \in \Omega_{d,n}(B) \\ \mu(\mathbf{x},\mathbf{y}) > A}}
\frac{\mu(\mathbf{x},\mathbf{y})}{\det (\Gamma_{\mathbf{x}, \mathbf{y}})}.
\end{equation*}
Therefore, using Lemma~\ref{Lemma E} we deduce that
\begin{equation}
\label{Estimate Sigma1(A,B)}
\Sigma_{d,n}^{(1)}(A,B) = \iota_{d,n} A^{N_{d,n}-2} E_{d,n}(B) \left( 1 + O \left(\mathcal{E}_{d,n}^{(1)}(A,B) \right) \right),
\end{equation}
where
\begin{equation*}
\mathcal{E}_{d,n}^{(1)}(A,B) = \frac1{AB^2} \sum_{(\mathbf{x}, \mathbf{y}) \in \Omega_{d,n}(B)}
\frac{\mu(\mathbf{x},\mathbf{y})}{\det (\Gamma_{\mathbf{x}, \mathbf{y}})}.
\end{equation*}
We shall handle the quantity $\mathcal{E}_{d,n}^{(1)}(A,B)$ by breaking the sizes of $||\mathbf{x}||$, $||\mathbf{y}||$,
$\mathfrak{d}_2(\mathbf{x},\mathbf{y})$ and $\mathfrak{d}_3(\mathbf{x},\mathbf{y})$ into dyadic intervals. Recall that we have the upper bound \eqref{Upper bound trivial drxy}. Moreover, we note that the inequality
$\mathfrak{d}_2(\mathbf{x},\mathbf{y}) \geq \mathfrak{d}_3(\mathbf{x},\mathbf{y})$ implies that
\begin{equation*}
\mu(\mathbf{x},\mathbf{y}) \leq 3 n^2 \frac{||\mathbf{x}|| \cdot ||\mathbf{y}||}{\mathfrak{d}_3(\mathbf{x},\mathbf{y})}.
\end{equation*}
Using the lower bound in Lemma~\ref{Lemma upper and lower bounds det}, we thus deduce
\begin{equation*}
\frac{\mu(\mathbf{x},\mathbf{y})}{\det (\Gamma_{\mathbf{x}, \mathbf{y}})} \ll
\frac1{\mathfrak{d}_2(\mathbf{x},\mathbf{y}) \cdot \mathfrak{d}_3(\mathbf{x},\mathbf{y}) \cdot
||\mathbf{x}||^{d-2} ||\mathbf{y}||^{d-2}}.
\end{equation*}
Recalling the definition \eqref{Definition l_r(X,Y)} of the quantity $\ell_{r,n}(X,Y;\Delta_r)$, we see that we have the upper bound
\begin{equation*}
\mathcal{E}_{d,n}^{(1)}(A,B) \ll \frac1{A B^2} \sum_{X,Y \ll B^{1/(n+1-d)}} \
\sum_{\Delta_2, \Delta_3 \ll XY} \frac1{\Delta_2 \Delta_3 (XY)^{d-2}} \cdot \min_{r \in \{2,3\}} \ell_{r,n}(X,Y;\Delta_r).
\end{equation*} 
It follows from Lemma~\ref{Lemma l_r(X,Y)} that 
\begin{equation*}
\min_{r \in \{2,3\}} \ell_{r,n}(X,Y;\Delta_r) \ll (\log B)^2 (XY)^2 \min \left\{ \Delta_2^{n-1}, \Delta_3^{n-1} X Y \right\}.
\end{equation*}
We now make use of the inequality
\begin{equation*}
\min \left\{ \Delta_2^{n-1}, \Delta_3^{n-1} X Y \right\} \leq \left(\Delta_2^{n-1}\right)^{1-1/(n-1)}
\left(\Delta_3^{n-1} X Y \right)^{1/(n-1)}.
\end{equation*}
This yields
\begin{align}
\nonumber
\mathcal{E}_{d,n}^{(1)}(A,B) & \ll \frac{(\log B)^2}{A B^2} \sum_{X,Y \ll B^{1/(n+1-d)}} \
\sum_{\Delta_2, \Delta_3 \ll XY} \frac{\Delta_2^{n-3}}{(XY)^{d-4-1/(n-1)}} \\
\nonumber
& \ll \frac{(\log B)^4}{A B^2} \sum_{X,Y \ll B^{1/(n+1-d)}} (XY)^{n+1-d+1/(n-1)} \\
\label{Upper bound intermediate E1(A,B)}
& \ll (\log B)^4 \frac{B^{2/(n-1)(n+1-d)}}{A}.
\end{align}
Since $(d,n) \notin \{(2,2), (3,3), (4,4) \}$ we have $(n-1)(n+1-d) \geq 4$, so the assumption $B \ll A \log A$ gives
\begin{equation*}
\mathcal{E}_{d,n}^{(1)}(A,B) \ll \frac{(\log A)^{9/2}}{A^{1/2}}.
\end{equation*}
Recalling the estimate \eqref{Estimate Sigma1(A,B)}, we see that we have obtained
\begin{equation}
\label{Estimate Sigma_1(A,B) final}
\Sigma_{d,n}^{(1)}(A,B) = \iota_{d,n} A^{N_{d,n}-2} E_{d,n}(B) \left( 1 + O \left( \frac{(\log A)^{9/2}}{A^{1/2}} \right) \right).
\end{equation}

We now handle the sum $\Sigma_{d,n}^{(2)}(A,B)$. We start by noting that the trivial upper bound
$\mu(\mathbf{x},\mathbf{y}) \leq 3 n^2 ||\mathbf{x}|| \cdot ||\mathbf{y}||$ and the condition
$A < \mu(\mathbf{x},\mathbf{y})$ together give $A < 3n^2 B^{2/(n+1-d)}$. Therefore, the assumption $B \ll A \log A$ implies that we have either $n=d+1$ or $n=d$. In addition, we have the inequalities \eqref{Replacement key lemma two vectors} and $A < \mu(\mathbf{x},\mathbf{y})$ so we can apply the first part of Lemma~\ref{Lemma R0} with $M=1/2$. This yields
\begin{equation*}
\mathcal{S}^{\ast}_0(\Gamma_{\mathbf{x}, \mathbf{y}}; \mathcal{B}_{N_{d,n}}(A)) \ll \min_{R_0 \in \{0, \dots, N_{d,n}-3\}} \left( \frac{A^{N_{d,n}-2-R_0} \mu(\mathbf{x},\mathbf{y})^{R_0}}{\det (\Gamma_{\mathbf{x}, \mathbf{y}})} + A^{N_{d,n}-3-R_0} \right).
\end{equation*}
As a result, appealing to the lower bound in Lemma~\ref{Lemma upper and lower bounds det} and to
Lemma~\ref{Lemma E} we see that
\begin{equation}
\label{Upper bound Sigma2(A,B)}
\Sigma_{d,n}^{(2)}(A,B) \ll A^{N_{d,n}-2} E_{d,n}(B) \mathcal{E}_{d,n}^{(2)}(A,B),
\end{equation}
where
\begin{equation*}
\mathcal{E}_{d,n}^{(2)}(A,B) = \frac1{B^2}
\sum_{\substack{(\mathbf{x}, \mathbf{y}) \in \Omega_{d,n}(B) \\ \mu(\mathbf{x},\mathbf{y}) > A}}
\min_{R_0 \in \{0, \dots, N_{d,n}-3\}} \left(
\frac{\mu(\mathbf{x},\mathbf{y})^{R_0}}
{A^{R_0} \mathfrak{d}_2(\mathbf{x},\mathbf{y}) \cdot ||\mathbf{x}||^{d-1} ||\mathbf{y}||^{d-1}} + \frac1{A^{R_0+1}} \right).
\end{equation*}
We let $\mathcal{F}_{d,n}(A,B)$ be the contribution to $\mathcal{E}_{d,n}^{(2)}(A,B)$ coming from the
$(\mathbf{x}, \mathbf{y}) \in \Omega_{d,n}(B)$ satisfying
\begin{equation}
\label{Inequality cases mu}
\frac{\mathfrak{d}_2(\mathbf{x},\mathbf{y})}{\mathfrak{d}_3(\mathbf{x},\mathbf{y})} \leq
\frac{||\mathbf{x}|| \cdot ||\mathbf{y}||}{\mathfrak{d}_2(\mathbf{x},\mathbf{y})^2},
\end{equation}
and we also let
\begin{equation}
\label{Definition I}
\mathcal{I}_{d,n}(A,B) = \mathcal{E}_{d,n}^{(2)}(A,B) - \mathcal{F}_{d,n}(A,B).
\end{equation}
We first handle the quantity $\mathcal{F}_{d,n}(A,B)$. Recalling the definition \eqref{Definition mu} of
$\mu(\mathbf{x}, \mathbf{y})$, we see that
\begin{equation*}
\mathcal{F}_{d,n}(A,B) = \frac1{B^2} \! \!
\sum_{\substack{(\mathbf{x}, \mathbf{y}) \in \Omega_{d,n}(B) \\
\mathfrak{d}_2(\mathbf{x},\mathbf{y})^2 < 3 n^2 ||\mathbf{x}|| \cdot ||\mathbf{y}||/A}} \! \!
\min_{R_0 \in \{0, \dots, N_{d,n}-3\}} \left(
\frac{||\mathbf{x}||^{R_0-d+1} ||\mathbf{y}||^{R_0-d+1}}
{A^{R_0} \mathfrak{d}_2(\mathbf{x},\mathbf{y})^{2R_0+1}} + \frac1{A^{R_0+1}} \right).
\end{equation*}
Choosing $R_0 =n-1$ and breaking the size of $\mathfrak{d}_2(\mathbf{x},\mathbf{y})$ into dyadic intervals we get
\begin{align*}
\mathcal{F}_{d,n}(A,B) \ll & \ \frac1{B^2} \sum_{\Delta_2 \ll B^{1/(n+1-d)}/A^{1/2}}
\left( \frac{B^{2(n-d)/(n+1-d)}}{A^{n-1}\Delta_2^{2n-1}} + \frac1{A^n} \right) \\
& \times \ell_{2,n} \left(B^{1/(n+1-d)},B^{1/(n+1-d)};\Delta_2\right).
\end{align*}
Using Lemma~\ref{Lemma l_r(X,Y)} we deduce that
\begin{align*}
\mathcal{F}_{d,n}(A,B) & \ll \sum_{\Delta_2 \ll B^{1/(n+1-d)}/A^{1/2}}
\left( \frac{B^{2/(n+1-d)}}{A^{n-1} \Delta_2^n} + \frac{B^{4/(n+1-d)-2} \Delta_2^{n-1}}{A^n} \right) \\
& \ll \frac{B^{2/(n+1-d)}}{A^{n-1}} + \frac{B^{(2d-n+1)/(n+1-d)}}{A^{(3n-1)/2}}.
\end{align*}
Recall that we have either $n=d+1$ or $n=d$. For any $n \geq 3$, we have proved that
\begin{equation*}
\mathcal{F}_{n-1,n}(A,B) \ll \frac{B}{A^{n-1}} + \frac{B^{(n-1)/2}}{A^{(3n-1)/2}},
\end{equation*}
and
\begin{equation}
\label{Upper bound intermediate F}
\mathcal{F}_{n,n}(A,B) \ll \frac{B^2}{A^{n-1}} + \frac{B^{n+1}}{A^{(3n-1)/2}}.
\end{equation}
Therefore, using the assumptions $(d,n) \notin \{ (2,2), (3,3), (4,4)\}$ and $B \ll A \log A$ we conclude that
\begin{equation}
\label{Upper bound F}
\mathcal{F}_{d,n}(A,B) \ll \frac{(\log A)^6}{A}.
\end{equation}

We finally deal with the quantity $\mathcal{I}_{d,n}(A,B)$. We have
\begin{equation*}
\mathcal{I}_{d,n}(A,B) \leq \frac1{B^2}
\sum_{\substack{(\mathbf{x}, \mathbf{y}) \in \Omega_{d,n}(B) \\
\mathfrak{d}_3(\mathbf{x},\mathbf{y}) <3 n^2 ||\mathbf{x}|| \cdot ||\mathbf{y}||/A}}
\min_{R_0 \in \{d, \dots, N_{d,n}-3\}} \left(
\frac{||\mathbf{x}||^{R_0-d} ||\mathbf{y}||^{R_0-d}}{A^{R_0} \mathfrak{d}_3(\mathbf{x},\mathbf{y})^{R_0}} + \frac1{A^{R_0+1}} \right).
\end{equation*}
Note that we have restricted the minimum to $R_0 \geq d$ and we have then applied the upper bound
\eqref{Upper bound trivial drxy} with $r=2$. Breaking the size of $\mathfrak{d}_3(\mathbf{x},\mathbf{y})$ into dyadic intervals we see that
\begin{align*}
\mathcal{I}_{d,n}(A,B) \ll & \ \frac1{B^2} \sum_{\Delta_3 \ll B^{2/(n+1-d)}/A}
\left( \min_{R_0 \in \{d, \dots, N_{d,n}-3\}} \left( \frac{ B^{(2R_0-2d)/(n+1-d)}}{A^{R_0} \Delta_3^{R_0}} + \frac1{A^{R_0+1}} \right) \right) \\
& \times \ell_{3,n}\left(B^{1/(n+1-d)},B^{1/(n+1-d)};\Delta_3\right).
\end{align*}
Using Lemma~\ref{Lemma l_r(X,Y)} we get
\begin{equation*}
\mathcal{I}_{d,n}(A,B) \ll (\log B)^2 \! \! \! \! \sum_{\Delta_3 \ll B^{2/(n+1-d)}/A} \!
\min_{R_0} \left( \frac{B^{(2R_0-2n+4)/(n+1-d)}}{A^{R_0} \Delta_3^{R_0-n+1}} +
\frac{B^{6/(n+1-d)-2} \Delta_3^{n-1}}{A^{R_0+1}} \right) \! ,
\end{equation*}
where the minimum is over $R_0 \in \{ d, \dots, N_{d,n}-3\}$. We recall that we have either $n=d+1$ or $n=d$ and we first handle the case $n=d+1$. Choosing $R_0=n-1$ we deduce that
\begin{equation*}
\mathcal{I}_{n-1,n}(A,B) \ll (\log B)^2 \sum_{\Delta_3 \ll B/A} \left( \frac{B}{A^{n-1}} + \frac{B \Delta_3^{n-1}}{A^n} \right).
\end{equation*}
Using the assumption $B \ll A \log A$ we conclude that for any $n \geq 3$ we have
\begin{equation}
\label{Upper bound I n=d+1}
\mathcal{I}_{n-1,n}(A,B) \ll \frac{(\log A)^4}{A^{n-2}}.
\end{equation}
We now treat the case $n=d$. The assumption $B \ll A \log A$ yields
\begin{equation*}
\mathcal{I}_{n,n}(A,B) \ll (\log A)^{2N_{n,n}-2n} \sum_{\Delta_3 \ll A (\log A)^2}
\min_{R_0 \in \{ n, \dots, N_{n,n}-3\}} \left( \frac{A^{R_0-2n+4}}{\Delta_3^{R_0-n+1}} + \frac{\Delta_3^{n-1}}{A^{R_0-3}} \right).
\end{equation*}
Changing $R_0$ in $R_0+1$ we see that the right-hand side can grow at most by $(\log A)^2$ when $n$ increases by $1$. For any $n \geq 5$ we thus have
\begin{equation*}
\mathcal{I}_{n,n}(A,B) \ll (\log A)^{2 N_{n,n}-10} \sum_{\Delta_3 \ll A (\log A)^2}
\min_{R_0 \in \{5, \dots, N_{5,5}-3\}} \left( \frac{A^{R_0-6}}{\Delta_3^{R_0-4}} + \frac{\Delta_3^4}{A^{R_0-3}} \right).
\end{equation*}
Taking successively $R_0=5, \dots, R_0 = 8$ as $\Delta_3$ increases we obtain
\begin{align*}
\mathcal{I}_{n,n}(A,B) \ll & \ (\log A)^{2 N_{n,n}-10} \Bigg( \sum_{\Delta_3 \leq A^{1/3}} \left( \frac1{A \Delta_3} + \frac{\Delta_3^4}{A^2} \right) + \sum_{A^{1/3} < \Delta_3 \leq A^{4/7}} \left( \frac1{\Delta_3^2} + \frac{\Delta_3^4}{A^3} \right) \\
& + \sum_{A^{4/7} < \Delta_3 \leq A^{3/4}} \left( \frac{A}{\Delta_3^3} + \frac{\Delta_3^4}{A^4} \right) + \sum_{A^{3/4} < \Delta_3 \ll A (\log A)^2} \left( \frac{A^2}{\Delta_3^4} + \frac{\Delta_3^4}{A^5} \right) \Bigg).
\end{align*}
It thus follows that for any $n \geq 5$ we have
\begin{equation}
\label{Upper bound I n=d}
\mathcal{I}_{n,n}(A,B) \ll \frac{(\log A)^{2 N_{n,n}-10}}{A^{2/3}}.
\end{equation}
Combining the equality \eqref{Definition I} with the upper bounds \eqref{Upper bound F},
\eqref{Upper bound I n=d+1} and \eqref{Upper bound I n=d} we deduce that
\begin{equation*}
\mathcal{E}_{d,n}^{(2)}(A,B) \ll \frac{(\log A)^{2 N_{n,n}-10}}{A^{2/3}}.
\end{equation*}
Recalling the upper bound \eqref{Upper bound Sigma2(A,B)} we conclude that
\begin{equation}
\label{Upper bound Sigma2(A,B) final}
\Sigma_{d,n}^{(2)}(A,B) \ll A^{N_{d,n}-2} E_{d,n}(B) \cdot \frac{(\log A)^{2 N_{n,n}-10}}{A^{2/3}}.
\end{equation}
Putting together the equality \eqref{Definition Sigma2(A,B)}, the estimate \eqref{Estimate Sigma_1(A,B) final} and the upper bound \eqref{Upper bound Sigma2(A,B) final} completes the proof.
\end{proof}

We now use our work in Section~\ref{Section quartic threefolds} to handle separately the quantity $D_{4,4}(A,B)$. Recall the definition \eqref{Definition iota} of $\iota_{4,4}$.

\begin{lemma} 
\label{Lemma D44}
Assume that $A \geq B/(\log B)$. Then we have 
\begin{equation*}
D_{4,4}(A,B) = \iota_{4,4} A^{N_{4,4}-2} E_{4,4}(B) \left( 1 + O \left( \frac1{A^{1/21}} \right) \right).
\end{equation*}
\end{lemma}

\begin{proof}
Combining the equality \eqref{Definition Sigma2(A,B)}, the estimate \eqref{Estimate Sigma1(A,B)} and the upper bound \eqref{Upper bound intermediate E1(A,B)}, and using the assumption $B \ll A \log A$ we obtain
\begin{equation}
\label{Estimate D44}
D_{4,4}(A,B) = \iota_{4,4} A^{N_{4,4}-2} E_{4,4}(B)
\left( 1 + O \left( \frac{(\log A)^{11/3}}{A^{1/3}} \right) \right) + \Sigma_{4,4}^{(2)}(A,B).
\end{equation}
Let $\delta \in (0,1/15)$ to be selected in due course. Let $\mathscr{S}_{>\delta}(A,B)$ denote the contribution to
$\Sigma_{4,4}^{(2)}(A,B)$ from the $(\mathbf{x},\mathbf{y}) \in \Omega_{4,4}(B)$ satisfying
$\mathfrak{d}_3(\mathbf{x},\mathbf{y}) > A^{\delta}$ and set
\begin{equation}
\label{Definition S<delta}
\mathscr{S}_{\leq \delta}(A,B) = \Sigma_{4,4}^{(2)}(A,B) - \mathscr{S}_{>\delta}(A,B).
\end{equation}

We first handle the quantity $\mathscr{S}_{>\delta}(A,B)$. We are going to use the fact that the first successive minimum of the lattice $\Gamma_{\mathbf{x},\mathbf{y}}$ is usually quite large as $(\mathbf{x},\mathbf{y})$ runs over
$\Omega_{4,4}(B)$. We have the inequalities \eqref{Replacement key lemma two vectors} and
$A < \mu(\mathbf{x},\mathbf{y})$ so we can apply the first part of Lemma~\ref{Lemma R0} with
$M=\lambda_1(\Gamma_{\mathbf{x},\mathbf{y}})-1/2$. We deduce
\begin{equation*}
\mathcal{S}^{\ast}_0(\Gamma_{\mathbf{x}, \mathbf{y}}; \mathcal{B}_{N_{4,4}}(A)) \ll
\min_{R_0 \in \{0, \dots, N_{4,4}-3\}} \left(
\frac{A^{N_{4,4}-2-R_0} \mu(\mathbf{x},\mathbf{y})^{R_0}}{\det (\Gamma_{\mathbf{x}, \mathbf{y}})} +
\left( \frac{A}{\lambda_1(\Gamma_{\mathbf{x},\mathbf{y}})}\right)^{N_{4,4}-3-R_0} \right).
\end{equation*}
Noting that $N_{4,4}=70$ and appealing to the lower bound in Lemma~\ref{Lemma upper and lower bounds det} and to Lemma~\ref{Lemma E} we see that
\begin{equation*}
\mathscr{S}_{>\delta}(A,B) \ll A^{N_{4,4}-2} E_{4,4}(B) \mathscr{E}_{>\delta}(A,B),
\end{equation*}
where
\begin{equation*}
\mathscr{E}_{>\delta}(A,B) = \frac1{B^2}
\sum_{\substack{(\mathbf{x}, \mathbf{y}) \in \Omega_{4,4}(B) \\
\mu(\mathbf{x},\mathbf{y}) > A \\ \mathfrak{d}_3(\mathbf{x},\mathbf{y}) > A^{\delta}}}
\min_{R_0} \left(
\frac{\mu(\mathbf{x},\mathbf{y})^{R_0}}
{A^{R_0} \mathfrak{d}_2(\mathbf{x},\mathbf{y}) \cdot ||\mathbf{x}||^3 ||\mathbf{y}||^3} +
\frac1{A^{R_0+1} \lambda_1(\Gamma_{\mathbf{x},\mathbf{y}})^{67-R_0}} \right),
\end{equation*}
where the minimum is over $R_0 \in \{0, \dots, 67\}$. Since $\lambda_1(\Gamma_{\mathbf{x},\mathbf{y}}) \geq 1$ the contribution to $\mathscr{E}_{>\delta}(A,B)$ coming from the $(\mathbf{x}, \mathbf{y}) \in \Omega_{4,4}(B)$ satisfying the inequality \eqref{Inequality cases mu} is at most $\mathcal{F}_{4,4}(A,B)$. Using the upper bound
\eqref{Upper bound intermediate F} we deduce that
\begin{equation}
\label{Definition Fdelta}
\mathscr{S}_{>\delta}(A,B) \ll A^{N_{4,4}-2} E_{4,4}(B) \left( \mathscr{F}_{>\delta}(A,B) + \frac{(\log A)^5}{A^{1/2}} \right),
\end{equation}
where
\begin{equation*}
\mathscr{F}_{>\delta}(A,B) = \frac1{B^2}
\sum_{\substack{(\mathbf{x}, \mathbf{y}) \in \Omega_{4,4}(B) \\
A^{\delta} < \mathfrak{d}_3(\mathbf{x},\mathbf{y}) < ||\mathbf{x}|| \cdot ||\mathbf{y}||/ A}} \!
\min_{R_0} \left( \frac{||\mathbf{x}||^{R_0-4} ||\mathbf{y}||^{R_0-4}}
{A^{R_0} \mathfrak{d}_3(\mathbf{x},\mathbf{y})^{R_0}} +
\frac1{A^{R_0+1} \lambda_1(\Gamma_{\mathbf{x},\mathbf{y}})^{67-R_0}} \right),
\end{equation*}
where the minimum is over $R_0 \in \{1, \dots, 67\}$. Note that we have restricted the minimum to $R_0 \geq 1$ and we have then applied the upper bound \eqref{Upper bound trivial drxy} with $r=2$. Since the lattice
$\Gamma_{\mathbf{x}, \mathbf{y}}$ has rank $68$ it follows from Minkowski's estimate \eqref{Estimates Minkowski} and the upper bound in Lemma~\ref{Lemma upper and lower bounds det} that
\begin{equation*}
\lambda_1(\Gamma_{\mathbf{x},\mathbf{y}}) \ll ||\mathbf{x}||^{1/17} ||\mathbf{y}||^{1/17}.
\end{equation*}
Recall the definition \eqref{Definition lj} of the quantity $\ell^{(1)}(B;\Delta_3,M)$. Breaking the sizes of
$\mathfrak{d}_3(\mathbf{x},\mathbf{y})$ and $\lambda_1(\Gamma_{\mathbf{x},\mathbf{y}})$ into dyadic intervals we see that
\begin{equation*}
\mathscr{F}_{>\delta}(A,B) \ll
\frac1{B^2} \sum_{\substack{A^{\delta} \ll \Delta_3 \ll A (\log A)^2 \\ M \ll (A \log A)^{2/17}}} \!
\left( \min_{R_0} \left( \frac{B^{2R_0-8}}{A^{R_0}\Delta_3^{R_0}} + \frac1{A^{R_0+1} M^{67-R_0}} \right) \right)
\ell^{(1)}(B;\Delta_3,M),
\end{equation*}
where the minimum is over $R_0 \in \{4, \dots, 67\}$. Using Lemma~\ref{Lemma l1 convenient} and the assumption
$B/(\log B) \leq A$ we get
\begin{equation*}
\ell^{(1)}(B;\Delta_3,M) \ll B^{9/2} \Delta_3^2 A^{3/2} (\log A)^5 \min \left\{ \Delta_3, M^{53} \right\}.
\end{equation*}
It follows that
\begin{equation}
\label{Upper bound Fdelta}
\mathscr{F}_{>\delta}(A,B) \ll (\log A)^{134} \mathscr{G}_{>\delta}(A),
\end{equation}
where
\begin{equation*}
\mathscr{G}_{>\delta}(A) = \sum_{\substack{A^{\delta} \ll \Delta_3 \ll A (\log A)^2 \\ M \ll (A \log A)^{2/17}}}
\left( \min_{R_0 \in \{4, \dots, 67\}} \left( \frac{A^{R_0-4}}{\Delta_3^{R_0-2}} + \frac{\Delta_3^2}{A^{R_0-3} M^{67-R_0}} \right) \right) \min \left\{ \Delta_3, M^{53} \right\}.
\end{equation*}
Taking successively $R_0=4, \dots, R_0=6$ as $\Delta_3$ increases we obtain
\begin{align*}
\mathscr{G}_{>\delta}(A) \ll & \
\sum_{\substack{A^{\delta} \ll \Delta_3 \leq A^{2/5} M^{63/5} \\ M \ll (A \log A)^{2/17}}}
\left( \frac1{\Delta_3} + \frac{\Delta_3^2 \min \left\{ \Delta_3, M^{53} \right\}}{A M^{63}} \right) \\
& + \sum_{\substack{A^{2/5} M^{63/5} < \Delta_3 \leq A^{2/3} M^{31/3} \\ M \ll (A \log A)^{2/17}}}
\left( \frac{A \min \left\{ \Delta_3, M^{53} \right\}}{\Delta_3^3} + \frac{\Delta_3^2 \min \left\{ \Delta_3, M^{53} \right\}}{A^2 M^{62}} \right) \\
& + \sum_{\substack{A^{2/3} M^{31/3}< \Delta_3 \ll A (\log A)^2 \\ M \ll (A \log A)^{2/17}}}
\left( \frac{A^2 \min \left\{ \Delta_3, M^{53} \right\}}{\Delta_3^4} + \frac{\Delta_3^2}{A^3 M^8} \right).
\end{align*}
The summation over $\Delta_3$ leads to
\begin{align*}
\mathscr{G}_{>\delta}(A) & \ll \sum_{M \ll A^{1/8}}
\left( \frac1{A^{\delta}} + \min \left\{ \frac{A^{1/5}}{M^{126/5}}, \frac{M^{76/5}}{A^{1/5}} \right\}
+ \min \left\{ \frac1{M^{31}}, \frac{M^{35/3}}{A^{2/3}} \right\} + \frac{(\log A)^4}{A M^8} \right) \\
& \ll \frac{\log A}{A^{\delta}} + \! \sum_{M \leq A^{1/101}} \frac{M^{76/5}}{A^{1/5}}
+ \! \sum_{M > A^{1/101}} \frac{A^{1/5}}{M^{126/5}} + \sum_{M \leq A^{1/64}} \frac{M^{35/3}}{A^{2/3}}
+ \sum_{M > A^{1/64}} \frac1{M^{31}} \\
& \ll \frac{\log A}{A^{\delta}} + \frac1{A^{5/101}}.
\end{align*}
Recalling the upper bounds \eqref{Definition Fdelta} and \eqref{Upper bound Fdelta}, we see that we have proved
\begin{equation}
\label{Upper bound Sdelta}
\mathscr{S}_{>\delta}(A,B) \ll A^{N_{4,4}-2} E_{4,4}(B)
\left( \frac{(\log A)^{135}}{A^{\delta}} + \frac{(\log A)^{134}}{A^{5/101}} \right).
\end{equation}

We now handle the quantity $\mathscr{S}_{\leq \delta}(A,B)$. We are going to take advantage of the fact that it is rare for the lattice $\Gamma_{\mathbf{x},\mathbf{y}}$ to have $56$ small successive minima as
$(\mathbf{x},\mathbf{y})$ runs over $\Omega_{4,4}(B)$. We have the inequalities
\eqref{Replacement key lemma two vectors} and $A < \mu(\mathbf{x},\mathbf{y})$ so we can make use of the second part of Lemma~\ref{Lemma R0} with $M=1/2$, $j_0 = 55$ and
$J = \lambda_{56}(\Gamma_{\mathbf{x},\mathbf{y}})-1/2$. We obtain
\begin{equation*}
\mathcal{S}^{\ast}_0(\Gamma_{\mathbf{x}, \mathbf{y}}; \mathcal{B}_{N_{4,4}}(A)) \ll
\min_{R_0 \in \{0, \dots, 12\}} \left(
\frac{A^{N_{4,4}-2-R_0} \mu(\mathbf{x},\mathbf{y})^{R_0}}{\det (\Gamma_{\mathbf{x}, \mathbf{y}})} +
\frac{A^{67-R_0}}{\lambda_{56}(\Gamma_{\mathbf{x},\mathbf{y}})^{12-R_0}} \right) + A^{55}.
\end{equation*}
The lower bound in Lemma~\ref{Lemma upper and lower bounds det} and Lemma~\ref{Lemma E} thus give
\begin{equation}
\label{Definition Edelta}
\mathscr{S}_{\leq \delta}(A,B) \ll A^{N_{4,4}-2} E_{4,4}(B)
\left( \mathscr{E}_{\leq \delta}(A,B) + \frac{B^8}{A^{13}} \right),
\end{equation}
where
\begin{equation*}
\mathscr{E}_{\leq \delta}(A,B) = \frac1{B^2}
\sum_{\substack{(\mathbf{x}, \mathbf{y}) \in \Omega_{4,4}(B) \\
\mu(\mathbf{x},\mathbf{y}) > A \\ \mathfrak{d}_3(\mathbf{x},\mathbf{y}) \leq A^{\delta}}}
\min_{R_0} \left(
\frac{\mu(\mathbf{x},\mathbf{y})^{R_0}}
{A^{R_0} \mathfrak{d}_2(\mathbf{x},\mathbf{y}) \cdot ||\mathbf{x}||^3 ||\mathbf{y}||^3} +
\frac1{A^{R_0+1} \lambda_{56}(\Gamma_{\mathbf{x},\mathbf{y}})^{12-R_0}} \right),
\end{equation*}
and where the minimum is over $R_0 \in \{0, \dots, 12\}$. We use again the observation that the trivial lower bound
$\lambda_{56}(\Gamma_{\mathbf{x},\mathbf{y}}) \geq 1$ implies that the contribution to
$\mathscr{E}_{\leq \delta}(A,B)$ coming from the $(\mathbf{x}, \mathbf{y}) \in \Omega_{4,4}(B)$ satisfying the inequality \eqref{Inequality cases mu} is at most $\mathcal{F}_{4,4}(A,B)$. Hence the upper bounds
\eqref{Upper bound intermediate F} and \eqref{Definition Edelta} and the assumption $B \ll A \log A$ give
\begin{equation}
\label{Definition F<delta}
\mathscr{S}_{\leq \delta}(A,B) \ll A^{N_{4,4}-2} E_{4,4}(B) \left( \mathscr{F}_{\leq \delta}(A,B) + \frac{(\log A)^5}{A^{1/2}} \right),
\end{equation}
where
\begin{equation*}
\mathscr{F}_{\leq \delta}(A,B) = \frac1{B^2}
\sum_{\substack{(\mathbf{x}, \mathbf{y}) \in \Omega_{4,4}(B) \\ \mathfrak{d}_3(\mathbf{x},\mathbf{y}) \leq A^{\delta}}}
\min_{R_0 \in \{1, \dots, 12\}} \left( \frac{||\mathbf{x}||^{R_0-4} ||\mathbf{y}||^{R_0-4}}
{A^{R_0} \mathfrak{d}_3(\mathbf{x},\mathbf{y})^{R_0}} +
\frac1{A^{R_0+1} \lambda_{56}(\Gamma_{\mathbf{x},\mathbf{y}})^{12-R_0}} \right).
\end{equation*}
Note that we have used the upper bound \eqref{Upper bound trivial drxy} with $r=2$ after restricting the minimum to $R_0 \geq 1$. Since the lattice $\Gamma_{\mathbf{x}, \mathbf{y}}$ has rank $68$ it follows from Minkowski's estimate \eqref{Estimates Minkowski} and the upper bound in Lemma~\ref{Lemma upper and lower bounds det} that
\begin{equation*}
\lambda_{56}(\Gamma_{\mathbf{x},\mathbf{y}}) \ll ||\mathbf{x}||^{4/13} ||\mathbf{y}||^{4/13}.
\end{equation*}
Recall the respective definitions \eqref{Definition l_r(X,Y)} and \eqref{Definition lj} of the quantities
$\ell_{3,4}(X, Y;\Delta_3)$ and $\ell^{(56)}(B;\Delta_3,J)$. Breaking the sizes of $||\mathbf{x}||$, $||\mathbf{y}||$,
$\mathfrak{d}_3(\mathbf{x},\mathbf{y})$ and $\lambda_{56}(\Gamma_{\mathbf{x},\mathbf{y}})$ into dyadic intervals we see that
\begin{align*}
\mathscr{F}_{\leq \delta}(A,B) \ll & \ \frac1{B^2}
\sum_{\substack{X, Y \ll B \\ \Delta_3 \ll A^{\delta} \\ J \ll (A \log A)^{8/13}}}
\left( \min_{R_0 \in \{1, \dots, 12\}} \left( \frac{(XY)^{R_0-4}}{A^{R_0}\Delta_3^{R_0}} + \frac1{A^{R_0+1} J^{12-R_0}} \right) \right) \\
& \times \min \left\{ \ell_{3,4}(X, Y;\Delta_3), \ell^{(56)}(B;\Delta_3,J) \right\}.
\end{align*}
Applying Lemmas~\ref{Lemma l_r(X,Y)} and \ref{Lemma l56} and using the assumption $B \ll A \log A$ we deduce that
\begin{equation*}
\mathscr{F}_{\leq \delta}(A,B) \ll (\log A)^6 \! \! \! \!
\sum_{\substack{X, Y \ll B \\ \Delta_3 \ll A^{\delta} \\ J \ll (A \log A)^{8/13}}} \! \! \! \min_{R_0 \in \{1, \dots, 12\}} \left( \frac{(XY)^{R_0-1}}{A^{R_0}B^2\Delta_3^{R_0-3}} +
\frac{\Delta_3^3 \min \left\{ A^2, \Delta_3^{21} J^{15} \right\}}{A^{R_0-1} J^{12-R_0}} \right).
\end{equation*}
We take $R_0=3$ and make use of the inequality
\begin{equation*}
\min \left\{ A^2, \Delta_3^{21} J^{15} \right\} \leq \left( A^2 \right)^{2/5} \left( \Delta_3^{21} J^{15} \right)^{3/5}.
\end{equation*}
We thus derive
\begin{align*}
\mathscr{F}_{\leq \delta}(A) & \ll
(\log A)^6 \sum_{X,Y \ll B} \sum_{\substack{\Delta_3 \ll A^{\delta} \\ J \ll (A \log A)^{8/13}}}
\left( \frac{(XY)^2}{A^3 B^2} + \frac{\Delta_3^{78/5}}{A^{6/5}} \right) \\
& \ll (\log A)^{10} \left( \frac1{A} + \frac1{A^{6/5-78 \delta/5}} \right).
\end{align*}
Recalling the upper bound \eqref{Definition F<delta} we see that we have obtained
\begin{equation}
\label{Upper bound S<delta}
\mathscr{S}_{\leq \delta}(A,B) \ll
A^{N_{4,4}-2} E_{4,4}(B) \left( \frac{(\log A)^5}{A^{1/2}} + \frac{(\log A)^{10}}{A^{6/5-78 \delta/5}} \right).
\end{equation}

Putting together the equality \eqref{Definition S<delta} and the upper bounds \eqref{Upper bound Sdelta} and \eqref{Upper bound S<delta} and choosing for instance $\delta= 1/20$ we deduce that
\begin{equation*}
\Sigma_{4,4}^{(2)}(A,B) \ll A^{N_{4,4}-2} E_{4,4}(B) \cdot \frac{(\log A)^{134}}{A^{5/101}}.
\end{equation*}
Recalling the estimate \eqref{Estimate D44}, we see that this completes the proof.
\end{proof}

Recall the definition \eqref{Definition iota} of $\iota_{d,n}$. Our next task is to establish an estimate for the quantity $D_{d,n}^{\mathrm{mix}}(A,B)$ defined in \eqref{Definition D mixed}.

\begin{lemma}
\label{Lemma D mixed}
Let $d \geq 2$ and $n \geq d$ with $(d,n) \neq (2,2)$. Assume that $B^{5/6} \leq A \leq B^2$. Then we have 
\begin{equation*}
D_{d,n}^{\mathrm{mix}}(A,B) = \iota_{d,n} A^{N_{d,n}-2} E_{d,n}(B) \left( 1 + O \left(\frac1{(\log A)^{1/2}} \right) \right).
\end{equation*}
\end{lemma}

\begin{proof}
Recall that the definitions of the lattices $\Lambda_{\nu_{d,n}(\mathbf{x})}$ and
$\Lambda_{\nu_{d,n}(\mathbf{y})}^{(W)}$ and the region $\mathcal{C}_{\nu_{d,n}(\mathbf{y})}^{(\alpha)}$ were respectively given in \eqref{Definition lattice}, \eqref{Definition local lattice} and \eqref{Definition Cgamma}. We introduce the lattice
\begin{equation*}
\Gamma_{\mathbf{x}, \mathbf{y}}^{\mathrm{mix}}(W) =
\Lambda_{\nu_{d,n}(\mathbf{x})} \cap \Lambda_{\nu_{d,n}(\mathbf{y})}^{(W)},
\end{equation*}
and the region
\begin{equation*}
\mathcal{T}_{\mathbf{y}}^{\mathrm{mix}}(A,\alpha) =
\mathcal{B}_{N_{d,n}}(A) \cap \mathcal{C}_{\nu_{d,n}(\mathbf{y})}^{(\alpha)}.
\end{equation*}
Recall the respective definitions \eqref{Definition Sk star} and \eqref{Definition Sk} of the sums
$\mathcal{S}_k^{\ast} \left(\Lambda; \mathcal{R} \right)$ and $\mathcal{S}_k \left(\Lambda; \mathcal{R} \right)$ for any given lattice $\Lambda \subset \mathbb{Z}^{N_{d,n}}$, any bounded region $\mathcal{R} \subset \mathbb{R}^{N_{d,n}}$ and any integer $k \geq 0$. Recall also the definition \eqref{Definition Omega} of the set $\Omega_{d,n}(B)$. We see that
\begin{equation*}
D_{d,n}^{\mathrm{mix}}(A,B) = \frac{\alpha W}{8} \sum_{(\mathbf{x}, \mathbf{y}) \in \Omega_{d,n}(B)}
\frac{\mathcal{S}^{\ast}_1 \left( \Gamma_{\mathbf{x}, \mathbf{y}}^{\mathrm{mix}}(W); \mathcal{T}_{\mathbf{y}}^{\mathrm{mix}}(A,\alpha) \right)}{||\nu_{d,n}(\mathbf{y})||}.
\end{equation*}
Recall the definition \eqref{Definition mu(x)} of the quantity $\mu(\mathbf{x})$ and set
\begin{equation}
\label{Definition Sigma1 mixed}
\Sigma_1^{\mathrm{mix}}(A,B) =
\frac{\alpha W}{8} \sum_{\substack{(\mathbf{x}, \mathbf{y}) \in \Omega_{d,n}(B) \\ \mu(\mathbf{x}) \leq A/W}}
\frac{\mathcal{S}^{\ast}_1 \left( \Gamma_{\mathbf{x}, \mathbf{y}}^{\mathrm{mix}}(W); \mathcal{T}_{\mathbf{y}}^{\mathrm{mix}}(A,\alpha) \right)}{||\nu_{d,n}(\mathbf{y})||},
\end{equation}
and
\begin{equation}
\label{Definition Sigma2 mixed}
\Sigma_2^{\mathrm{mix}}(A,B) = D_{d,n}^{\mathrm{mix}}(A,B) - \Sigma_1^{\mathrm{mix}}(A,B).
\end{equation}

We start by dealing with the sum $\Sigma_1^{\mathrm{mix}}(A,B)$. A M\"{o}bius inversion gives
\begin{equation*}
\mathcal{S}^{\ast}_1 \left( \Gamma_{\mathbf{x}, \mathbf{y}}^{\mathrm{mix}}(W); \mathcal{T}_{\mathbf{y}}^{\mathrm{mix}}(A,\alpha) \right) = \sum_{\ell \leq A} \frac{\mu(\ell)}{\ell}
\mathcal{S}_1 \left( \Gamma_{\mathbf{x}, \mathbf{y}}^{\mathrm{mix}} \left(\frac{W}{\gcd(\ell,W)} \right);
\mathcal{T}_{\mathbf{y}}^{\mathrm{mix}} \left( \frac{A}{\ell}, \alpha \right) \right).
\end{equation*}
For any real $u \geq 1$, it is clear that we have
\begin{equation*}
\mathcal{S}_1 \left( \Gamma_{\mathbf{x}, \mathbf{y}}^{\mathrm{mix}} \left(\frac{W}{\gcd(\ell,W)} \right);
\mathcal{T}_{\mathbf{y}}^{\mathrm{mix}} \left( u, \alpha \right) \right) \leq
\mathcal{S}_1 \left( \Lambda_{\nu_{d,n}(\mathbf{x})}; \mathcal{B}_{N_{d,n}}(u) \right).
\end{equation*}
Breaking the size of $\mathbf{a}$ into dyadic intervals, we see that
\begin{equation*}
\mathcal{S}_1 \left( \Lambda_{\nu_{d,n}(\mathbf{x})}; \mathcal{B}_{N_{d,n}}(u) \right) \ll
\sum_{U \ll u} \frac1{U} \mathcal{S}_0 \left( \Lambda_{\nu_{d,n}(\mathbf{x})}; \mathcal{B}_{N_{d,n}}(U) \right).
\end{equation*}
Recall that we have the inequality \eqref{Replacement key lemma one vector}. Therefore, we see that if
$U \geq \mu(\mathbf{x})$ then we are in position to apply Lemma~\ref{Lemma lattice pivotal} with $I=1$ and
$\gamma=1/2$, say. On the other hand, if $U < \mu(\mathbf{x})$ then we can apply the first part of
Lemma~\ref{Lemma R0} with $M=1/2$ and $R_0=N_{d,n}-2$. It follows that
\begin{equation*}
\mathcal{S}_1 \left( \Lambda_{\nu_{d,n}(\mathbf{x})}; \mathcal{B}_{N_{d,n}}(u) \right) \ll
\sum_{U \ll u} \frac1{U}
\left( \frac{U^{N_{d,n}-1}}{||\mathbf{x}||^d} + \frac{U \mu(\mathbf{x})^{N_{d,n}-2}}{||\mathbf{x}||^d} + 1 \right).
\end{equation*}
Note that we have used the fact that Lemma~\ref{Lemma determinant global} gives
$\det(\Lambda_{\nu_{d,n}(\mathbf{x})}) = ||\nu_{d,n}(\mathbf{x})|| \gg ||\mathbf{x}||^d$. As a result, for
$u \geq \mu(\mathbf{x})$ we obtain
\begin{equation}
\label{Upper bound trivial u}
\mathcal{S}_1 \left( \Gamma_{\mathbf{x}, \mathbf{y}}^{\mathrm{mix}} \left(\frac{W}{\gcd(\ell,W)} \right);
\mathcal{T}_{\mathbf{y}}^{\mathrm{mix}} \left( u, \alpha \right) \right) \ll
\frac{u^{N_{d,n}-2}}{||\mathbf{x}||^d} \log u + 1.
\end{equation}
Moreover, writing that
\begin{equation*}
\mathcal{S}_0 \left( \Gamma_{\mathbf{x}, \mathbf{y}}^{\mathrm{mix}} \left(\frac{W}{\gcd(\ell,W)} \right);
\mathcal{T}_{\mathbf{y}}^{\mathrm{mix}} \left( u, \alpha \right) \right) \leq
\mathcal{S}_0 \left( \Lambda_{\nu_{d,n}(\mathbf{x})}; \mathcal{B}_{N_{d,n}}(u) \right),
\end{equation*}
and applying Lemma~\ref{Lemma lattice pivotal} with $I=1$ and $\gamma=1/2$, we deduce that for
$u \geq \mu(\mathbf{x})$ we have
\begin{equation}
\label{Upper bound S0}
\mathcal{S}_0 \left( \Gamma_{\mathbf{x}, \mathbf{y}}^{\mathrm{mix}} \left(\frac{W}{\gcd(\ell,W)} \right);
\mathcal{T}_{\mathbf{y}}^{\mathrm{mix}} \left( u, \alpha \right) \right) \ll
\frac{u^{N_{d,n}-1}}{||\mathbf{x}||^d}.
\end{equation}
Using the upper bound \eqref{Upper bound trivial u} with $u = W \mu(\mathbf{x})$, we see that
\begin{align*}
\sum_{A/W\mu(\mathbf{x})< \ell \leq A} \frac{\mu(\ell)}{\ell}
\mathcal{S}_1 \left( \Gamma_{\mathbf{x}, \mathbf{y}}^{\mathrm{mix}} \left(\frac{W}{\gcd(\ell,W)} \right);
\mathcal{T}_{\mathbf{y}}^{\mathrm{mix}} \left( \frac{A}{\ell}, \alpha \right) \right) \ll & \
\left( \frac{\left(W \mu(\mathbf{x})\right)^{N_{d,n}-2}}{||\mathbf{x}||^d} + 1 \right) \\
& \times (\log A)^2.
\end{align*}
Therefore, using the upper bound \eqref{Upper bound trivial u} once again with $u = W \mu(\mathbf{x})$ we get
\begin{align}
\label{Estimate l small mixed}
\begin{split}
\mathcal{S}^{\ast}_1 \left( \Gamma_{\mathbf{x}, \mathbf{y}}^{\mathrm{mix}}(W); \mathcal{T}_{\mathbf{y}}^{\mathrm{mix}}(A,\alpha) \right) = & \
\sum_{\ell \leq A/W\mu(\mathbf{x})} \frac{\mu(\ell)}{\ell} S_{\mathbf{x}, \mathbf{y}}^{\mathrm{mix}}(A, B; \ell) \\
& + O \left( \left( \frac{\left(W \mu(\mathbf{x})\right)^{N_{d,n}-2}}{||\mathbf{x}||^d} + 1 \right) (\log A)^2 \right),
\end{split}
\end{align}
where
\begin{equation*}
S_{\mathbf{x}, \mathbf{y}}^{\mathrm{mix}}(A, B; \ell) =
\mathcal{S}_1 \left( \Gamma_{\mathbf{x}, \mathbf{y}}^{\mathrm{mix}} \left(\frac{W}{\gcd(\ell,W)} \right);
\mathcal{T}_{\mathbf{y}}^{\mathrm{mix}} \left( \frac{A}{\ell}, \alpha \right) \smallsetminus
\mathcal{T}_{\mathbf{y}}^{\mathrm{mix}} \left( W \mu(\mathbf{x}), \alpha \right) \right).
\end{equation*}
Next, an application of partial summation yields
\begin{align*}
S_{\mathbf{x}, \mathbf{y}}^{\mathrm{mix}}(A, B; \ell) = & \
\frac{\ell}{A} \mathcal{S}_0 \left( \Gamma_{\mathbf{x}, \mathbf{y}}^{\mathrm{mix}} \left(\frac{W}{\gcd(\ell,W)} \right);
\mathcal{T}_{\mathbf{y}}^{\mathrm{mix}} \left( \frac{A}{\ell}, \alpha \right) \right) \\
& \! \! + \int_{W \mu(\mathbf{x})}^{A/\ell} \!
\mathcal{S}_0 \! \left( \Gamma_{\mathbf{x}, \mathbf{y}}^{\mathrm{mix}} \left(\frac{W}{\gcd(\ell,W)} \right);
\mathcal{T}_{\mathbf{y}}^{\mathrm{mix}} \left( t, \alpha \right) \right) \! \frac{\mathrm{d} t}{t^2}
+ O \! \left( \frac{(W \mu(\mathbf{x}))^{N_{d,n}-2}}{||\mathbf{x}||^d} \right) \! .
\end{align*}
Note that we have used the upper bound \eqref{Upper bound S0} with $u = W \mu(\mathbf{x})$. In addition,
Lemma~\ref{Key lemma one vector} states that the ball $\mathcal{B}_{N_{d,n}}(\mu(\mathbf{x}))$ contains $N_{d,n}-1$ linearly independent vectors of the lattice $\Lambda_{\nu_{d,n}(\mathbf{x})}$. Therefore, multiplying these vectors by $W$ we deduce that
\begin{equation*}
\lambda_{N_{d,n}-1} \left( \Gamma_{\mathbf{x}, \mathbf{y}}^{\mathrm{mix}} \left(\frac{W}{\gcd(\ell,W)} \right) \right) \leq
W \mu(\mathbf{x}).
\end{equation*}
Lemma~\ref{Lemma lattice pivotal} thus implies that for any $t \in [W \mu(\mathbf{x}), A/\ell]$, we have
\begin{align*}
\mathcal{S}_0 \left( \Gamma_{\mathbf{x}, \mathbf{y}}^{\mathrm{mix}} \left(\frac{W}{\gcd(\ell,W)} \right);
\mathcal{T}_{\mathbf{y}}^{\mathrm{mix}} \left( t, \alpha \right) \right) = & \
t^{N_{d,n}-1} \det \left( \Gamma_{\mathbf{x}, \mathbf{y}}^{\mathrm{mix}} \left( \frac{W}{\gcd(\ell,W)}\right) \right)^{-1} \\
& \times \left( \vol \left( \mathfrak{T}_{\mathbf{x}, \mathbf{y}}(\alpha) \right) + O \left( \frac{W \mu(\mathbf{x})}{t} \right) \right),
\end{align*}
where we have set
\begin{equation*}
\mathfrak{T}_{\mathbf{x}, \mathbf{y}}(\alpha) =
\Span_{\mathbb{R}} \left(\Lambda_{\nu_{d,n}(\mathbf{x})} \right) \cap
\mathcal{T}_{\mathbf{y}}^{\mathrm{mix}} ( 1, \alpha).
\end{equation*}
It follows that
\begin{align}
\label{Estimate after partial summation mixed}
\begin{split}
S_{\mathbf{x}, \mathbf{y}}^{\mathrm{mix}}(A, B; \ell) = & \
\frac{N_{d,n}-1}{N_{d,n}-2} \frac{A^{N_{d,n}-2}}{\ell^{N_{d,n}-2}} \cdot
\det \left( \Gamma_{\mathbf{x}, \mathbf{y}}^{\mathrm{mix}} \left( \frac{W}{\gcd(\ell,W)}\right) \right)^{-1} \\
& \times \left( \vol \left( \mathfrak{T}_{\mathbf{x}, \mathbf{y}}(\alpha) \right)
+ O \left( \frac{W \mu(\mathbf{x}) \ell}{A} \right) \right)
+ O \left( \frac{(W \mu(\mathbf{x}))^{N_{d,n}-2}}{||\mathbf{x}||^d} \right).
\end{split}
\end{align}
Recalling the definition \eqref{Definition volume mixed} of $\mathcal{I}(\mathbf{w}, \mathbf{z})$, we note that we have 
\begin{equation*}
\vol \left( \mathfrak{T}_{\mathbf{x}, \mathbf{y}}(\alpha) \right) = 
\mathcal{I} \left(2 \alpha \frac{\nu_{d,n}(\mathbf{x})}{||\nu_{d,n}(\mathbf{x})||},
2 \alpha \frac{\nu_{d,n}(\mathbf{y})}{||\nu_{d,n}(\mathbf{y})||} \right).
\end{equation*}
Recall the definition \eqref{Definition Delta} of the quantity $\Delta(\mathbf{x}, \mathbf{y})$. Using
Lemma~\ref{Lemma volume mixed} and the equality \eqref{Equality lattice} we see that
\begin{equation}
\label{Estimate volume mixed}
\vol \left( \mathfrak{T}_{\mathbf{x}, \mathbf{y}}(\alpha) \right) =
\frac{N_{d,n}-2}{N_{d,n}-1} V_{N_{d,n}-2} \frac{\Delta(\mathbf{x}, \mathbf{y})}{\alpha}
\left( 1 + O \left( \min \left\{ 1, \frac{\Delta(\mathbf{x}, \mathbf{y})^2}{\alpha^2} \right\} \right) \right).
\end{equation}
Moreover, Lemmas~\ref{Lemma determinant mixed/local} and \ref{Lemma G} give
\begin{equation*}
\det \left( \Gamma_{\mathbf{x}, \mathbf{y}}^{\mathrm{mix}} \left(\frac{W}{\gcd(\ell,W)} \right) \right) =
\frac{W ||\nu_{d,n}(\mathbf{x})||}{\gcd(\ell,W)} \cdot
\gcd\left(\mathcal{G}(\mathbf{x}, \mathbf{y}), \frac{W}{\gcd(\ell,W)} \right)^{-1}.
\end{equation*}
Recall the definition \eqref{Definition radical} of the radical of the integer $W$. We note that if $\ell$ is a squarefree integer then
\begin{equation}
\label{Estimate gcd}
\gcd\left(\mathcal{G}(\mathbf{x}, \mathbf{y}), \frac{W}{\gcd(\ell,W)} \right) = \mathcal{G}(\mathbf{x}, \mathbf{y})
\left( 1 + O \left( \boldsymbol{1}_{\mathcal{G}(\mathbf{x}, \mathbf{y}) \nmid W/\rad(W)} \right) \right).
\end{equation}
Hence 
\begin{equation}
\label{Estimate det mixed}
\det \left( \Gamma_{\mathbf{x}, \mathbf{y}}^{\mathrm{mix}} \left(\frac{W}{\gcd(\ell,W)} \right) \right) =
\frac{W||\nu_{d,n}(\mathbf{x})||}{\gcd(\ell,W) \mathcal{G}(\mathbf{x}, \mathbf{y})}
\left( 1 + O \left( \boldsymbol{1}_{\mathcal{G}(\mathbf{x}, \mathbf{y}) \nmid W/\rad(W)} \right) \right).
\end{equation}
Recall the definition \eqref{Definition localised error} of the quantity $\mathcal{E}_{\mathbf{x},\mathbf{y}}(B)$. Combining the estimates \eqref{Estimate after partial summation mixed}, \eqref{Estimate volume mixed} and \eqref{Estimate det mixed} and using the lower bound $\Delta(\mathbf{x}, \mathbf{y}) \geq 1$, we obtain
\begin{align*}
S_{\mathbf{x}, \mathbf{y}}^{\mathrm{mix}}(A, B; \ell) = & \
V_{N_{d,n}-2} \frac{A^{N_{d,n}-2}}{\alpha W} \frac{\gcd(\ell,W)}{\ell^{N_{d,n}-2}}
\frac{\Delta(\mathbf{x}, \mathbf{y}) \mathcal{G}(\mathbf{x}, \mathbf{y})}{||\nu_{d,n}(\mathbf{x})||} \\
& \times \left( 1 + O \left( \mathcal{E}_{\mathbf{x},\mathbf{y}}(B) + \frac{\alpha W \mu(\mathbf{x}) \ell}{A} \right) \right)
+ O \left( \frac{(W \mu(\mathbf{x}))^{N_{d,n}-2}}{||\mathbf{x}||^d} \right).
\end{align*}
Recall the definition \eqref{Definition alpha} of $\alpha$ and the upper bound \eqref{Upper bound W} for $W$ and note that the assumption $B \leq A^{6/5}$ gives $\alpha \ll \log A$ and $W \ll A^{5 / \log \log A}$. Using these facts and the inequalities $\Delta(\mathbf{x}, \mathbf{y}), \mathcal{G}(\mathbf{x}, \mathbf{y}) \geq 1$ and
$W \mu(\mathbf{x}) \leq A$, we see that the estimate \eqref{Estimate l small mixed} implies in particular that
\begin{align*}
\mathcal{S}^{\ast}_1 \left( \Gamma_{\mathbf{x}, \mathbf{y}}^{\mathrm{mix}}(W); \mathcal{T}_{\mathbf{y}}^{\mathrm{mix}}(A,\alpha) \right) = & \
V_{N_{d,n}-2} \frac{A^{N_{d,n}-2}}{\alpha W}
\frac{\Delta(\mathbf{x}, \mathbf{y}) \mathcal{G}(\mathbf{x}, \mathbf{y})}{||\nu_{d,n}(\mathbf{x})||} \\
& \times \left( \sum_{\ell \leq A/W\mu(\mathbf{x})} \mu(\ell) \frac{\gcd(\ell,W)}{\ell^{N_{d,n}-1}}
+ O \left( \mathcal{E}_{\mathbf{x},\mathbf{y}}(B) + \frac{\mu(\mathbf{x})}{A^{3/4}} \right) \right).
\end{align*}
Now it is clear that
\begin{equation*}
\sum_{\ell \leq A/W\mu(\mathbf{x})} \mu(\ell) \frac{\gcd(\ell, W)}{\ell^{N_{d,n}-1}} =
\sum_{\ell \leq w} \mu(\ell) \frac{\gcd(\ell, W)}{\ell^{N_{d,n}-1}}
+ O \left( \left( \frac{W\mu(\mathbf{x})}{A} \right)^{N_{d,n}-3} + \frac1{w^{N_{d,n}-3}} \right).
\end{equation*}
But the definition \eqref{Definition W} of $W$ shows that for any squarefree integer $\ell \leq w$ we have
$\gcd(\ell,W) = \ell$. We thus get
\begin{equation}
\label{Estimate zeta}
\sum_{\ell \leq A/W\mu(\mathbf{x})} \mu(\ell) \frac{\gcd(\ell, W)}{\ell^{N_{d,n}-1}} =
\frac1{\zeta(N_{d,n}-2)} + O \left( \left( \frac{W\mu(\mathbf{x})}{A} \right)^{N_{d,n}-3} + \frac1{w^{N_{d,n}-3}} \right).
\end{equation}
Finally, we note that the equality \eqref{Identity determinant} yields
\begin{equation}
\label{Equality Delta}
\frac{\Delta(\mathbf{x}, \mathbf{y}) \mathcal{G}(\mathbf{x}, \mathbf{y})}
{||\nu_{d,n}(\mathbf{x})|| \cdot ||\nu_{d,n}(\mathbf{y})||} =
\frac1{\det(\Lambda_{\nu_{d,n}(\mathbf{x})} \cap \Lambda_{\nu_{d,n}(\mathbf{y})})}.
\end{equation}
Recall that the quantity $F_{d,n}(B)$ was defined in \eqref{Definition F}. It follows from Lemma~\ref{Lemma E} and the equality \eqref{Definition Sigma1 mixed} that
\begin{equation}
\label{Estimate intermediate Sigma1}
\Sigma_1^{\mathrm{mix}}(A,B) = \iota_{d,n} A^{N_{d,n}-2} E_{d,n}(B)
\left( 1 + O \left( \frac1{w^{N_{d,n}-3}} + \frac{F_{d,n}(B)}{B^2} + \frac{G_{d,n}(B)}{B^2} \right) \right),
\end{equation}
where we have set
\begin{equation*}
G_{d,n}(B) = \frac1{A^{3/4}} \sum_{(\mathbf{x}, \mathbf{y}) \in \Omega_{d,n}(B)}
\frac{\mu(\mathbf{x})}{\det(\Lambda_{\nu_{d,n}(\mathbf{x})} \cap \Lambda_{\nu_{d,n}(\mathbf{y})})}.
\end{equation*}
Note that we have used the obvious fact that
\begin{align*}
\sum_{\substack{(\mathbf{x}, \mathbf{y}) \in \Omega_{d,n}(B) \\ \mu(\mathbf{x}) > A/W}}
\frac1{\det(\Lambda_{\nu_{d,n}(\mathbf{x})} \cap \Lambda_{\nu_{d,n}(\mathbf{y})})} & \leq \frac{W}{A}
\sum_{\substack{(\mathbf{x}, \mathbf{y}) \in \Omega_{d,n}(B) \\ \mu(\mathbf{x}) > A/W}}
\frac{\mu(\mathbf{x})}{\det(\Lambda_{\nu_{d,n}(\mathbf{x})} \cap \Lambda_{\nu_{d,n}(\mathbf{y})})} \\
& \ll G_{d,n}(B).
\end{align*}
An application of Lemma~\ref{Lemma upper and lower bounds det} gives
\begin{equation*}
G_{d,n}(B) \ll \frac1{A^{3/4}} \sum_{(\mathbf{x}, \mathbf{y}) \in \Omega_{d,n}(B)}
\frac1{\mathfrak{d}_2(\mathbf{x}) \mathfrak{d}_2(\mathbf{x},\mathbf{y}) \cdot ||\mathbf{x}||^{d-2} ||\mathbf{y}||^{d-1}}.
\end{equation*}
Recall the respective definitions \eqref{Definition l_r(X)} and \eqref{Definition l_r(X,Y)} of the quantities
$\ell_{2,n}(X;\Delta_0)$ and $\ell_{2,n}(X,Y;\Delta_2)$. We proceed to break the sizes of $||\mathbf{x}||$,
$||\mathbf{y}||$, $\mathfrak{d}_2(\mathbf{x})$ and $\mathfrak{d}_2(\mathbf{x},\mathbf{y})$ into dyadic intervals. Recalling that we have the upper bounds \eqref{Upper bound trivial drx} and \eqref{Upper bound trivial drxy} we deduce that
\begin{equation*}
G_{d,n}(B) \ll \frac1{A^{3/4}} \sum_{X, Y \ll B^{1/(n+1-d)}} \ \sum_{\Delta_0 \ll X} \ \sum_{\Delta_2 \ll XY}
\frac{\min \left\{ Y^{n+1} \ell_{2,n}(X;\Delta_0), \ell_{2,n}(X,Y;\Delta_2) \right\}}{\Delta_0 \Delta_2 X^{d-2} Y^{d-1}}.
\end{equation*}
Applying Lemmas~\ref{Lemma l_r(X)} and \ref{Lemma l_r(X,Y)} we see that
\begin{align*}
\min \left\{ Y^{n+1} \ell_{2,n}(X;\Delta_0), \ell_{2,n}(X,Y;\Delta_2) \right\} & \ll (\log X)
\min \left\{ X^2 Y^{n+1} \Delta_0^n, (XY)^2 \Delta_2^{n-1} \right\} \\
& \ll (\log X) \left( X^2 Y^{n+1} \Delta_0^n \right)^{1/n} \left( (XY)^2 \Delta_2^{n-1} \right)^{1-1/n}.
\end{align*}
We thus derive
\begin{align*}
G_{d,n}(B) & \ll \frac{\log B}{A^{3/4}} \sum_{X, Y \ll B^{1/(n+1-d)}} \ \sum_{\Delta_0 \ll X} \ \sum_{\Delta_2 \ll XY}
\frac{\Delta_2^{n-3+1/n}}{X^{d-4} Y^{d-4+1/n}} \\
& \ll \frac{(\log B)^2}{A^{3/4}} \sum_{X, Y \ll B^{1/(n+1-d)}} X^{n+1-d+1/n} Y^{n+1-d} \\
& \ll \frac{(\log B)^2}{A^{3/4}} B^{2+1/n(n+1-d)}.
\end{align*}
Using the fact that $n(n+1-d) \geq 3$ and the assumption $B \leq A^{6/5}$ we obtain in particular
\begin{equation}
\label{Upper bound G}
G_{d,n}(B) \ll \frac{B^2}{A^{1/3}}.
\end{equation}
Note that the definition \eqref{Definition w} of $w$ implies that $w^{N_{d,n}-3} \gg (\log B)^{1/2}$. Therefore, combining the estimate \eqref{Estimate intermediate Sigma1}, Lemma~\ref{Lemma F} and the upper bound
\eqref{Upper bound G}, and using the assumption $A \leq B^2$ we deduce that
\begin{equation}
\label{Estimate final Sigma1}
\Sigma_1^{\mathrm{mix}}(A,B) = \iota_{d,n} A^{N_{d,n}-2} E_{d,n}(B) \left( 1 + O \left( \frac1{(\log A)^{1/2}} \right) \right).
\end{equation}

We now deal with the quantity $\Sigma_2^{\mathrm{mix}}(A,B)$. Recall the definition \eqref{Definition Xi} of the set
$\Xi_{d,n}(B)$. We start by noting that we trivially have
\begin{align*}
\Sigma_2^{\mathrm{mix}}(A,B) & \leq \frac{\alpha W}{8}
\sum_{\substack{(\mathbf{x}, \mathbf{y}) \in \Omega_{d,n}(B) \\ \mu(\mathbf{x}) > A/W}}
\frac{\mathcal{S}_1^{\ast}( \Lambda_{\nu_{d,n}(\mathbf{x})}; \mathcal{B}_{N_{d,n}}(A))}{||\nu_{d,n}(\mathbf{y})||} \\
& \ll \alpha WB \sum_{\substack{\mathbf{x} \in \Xi_{d,n}(B) \\ \mu(\mathbf{x}) > A/W}} \# \left( \Lambda_{\nu_{d,n}(\mathbf{x})} \cap \mathcal{B}_{N_{d,n}}(A) \right).
\end{align*}
Lemma~\ref{Key lemma one vector} implies in particular that
\begin{equation*}
\lambda_{N_{d,n}-1} \left( \Lambda_{\nu_{d,n}(\mathbf{x})} \right) \leq W \mu(\mathbf{x}).
\end{equation*}
Moreover we have $A < W \mu(\mathbf{x})$ so we can apply the first part of Lemma~\ref{Lemma R0} with $M=1/2$ and $R_0=n-1$. This yields
\begin{align*}
\# \left(\Lambda_{\nu_{d,n}(\mathbf{x})} \cap \mathcal{B}_{N_{d,n}}(A) \right) & \ll
\frac{A^{N_{d,n}-n} \left( W \mu(\mathbf{x}) \right)^{n-1}}{\det(\Lambda_{\nu_{d,n}(\mathbf{x})})} + A^{N_{d,n}-n-1} \\
& \ll W^{n-1} A^{N_{d,n}-n} \frac{||\mathbf{x}||^{n-d-1}}{\mathfrak{d}_2(\mathbf{x})^{n-1}} + A^{N_{d,n}-n-1}.
\end{align*}
Note that we have used the fact that Lemma~\ref{Lemma determinant global} gives
$\det(\Lambda_{\nu_{d,n}(\mathbf{x})}) = ||\nu_{d,n}(\mathbf{x})|| \gg ||\mathbf{x}||^d$. Breaking the sizes of
$||\mathbf{x}||$ and $\mathfrak{d}_2(\mathbf{x})$ into dyadic intervals we see that we have
\begin{equation*}
\Sigma_2^{\mathrm{mix}}(A,B) \ll \alpha W^n A^{N_{d,n}-2} B \sum_{\substack{X \ll B^{1/(n+1-d)} \\ \Delta_0 \ll WX/A}} 
\left( \frac{X^{n-d-1}}{A^{n-2} \Delta_0^{n-1}} + \frac1{A^{n-1}} \right) \ell_{2,n}(X;\Delta_0).
\end{equation*}
It follows from Lemma~\ref{Lemma l_r(X)} that
\begin{align*}
\Sigma_2^{\mathrm{mix}}(A,B) & \ll
\alpha W^n A^{N_{d,n}-2} B (\log B) \sum_{X \ll B^{1/(n+1-d)}} \ \sum_{\Delta_0 \ll WX/A} 
\left( \frac{X^{n+1-d} \Delta_0}{A^{n-2}} + \frac{X^2 \Delta_0^n}{A^{n-1}} \right) \\
& \ll \alpha W^n A^{N_{d,n}-2} B (\log B) \sum_{X \ll B^{1/(n+1-d)}} \left( \frac{WX^{n+2-d}}{A^{n-1}} + \frac{W^nX^{n+2}}{A^{2n-1}} \right) \\
& \ll \alpha W^{2n} A^{N_{d,n}-2} B^2 \left( \frac{B^{1/(n+1-d)}}{A^{n-1}} + \frac{B^{(d+1)/(n+1-d)}}{A^{2n-1}} \right) \log B.
\end{align*}
Since $(d+1)/(n+1-d) \leq n+1$, the assumption $B \leq A^{6/5}$ gives
\begin{equation*}
 \frac{B^{(d+1)/(n+1-d)}}{A^{2n-1}} \leq \frac1{A^{(4n-11)/5}}.
\end{equation*}
Therefore, using the upper bounds $\alpha \ll \log A$ and $W \ll A^{5 / \log \log A}$ and the fact that $n \geq 3$, we see that Lemma~\ref{Lemma E} implies in particular that
\begin{equation}
\label{Upper bound final Sigma2}
\Sigma_2^{\mathrm{mix}}(A,B) \ll A^{N_{d,n}-2} E_{d,n}(B) \cdot \frac1{A^{1/10}}.
\end{equation}
Putting together the equality \eqref{Definition Sigma2 mixed}, the estimate \eqref{Estimate final Sigma1} and the upper bound \eqref{Upper bound final Sigma2} completes the proof.
\end{proof}

Recall the definition \eqref{Definition iota} of $\iota_{d,n}$. Our final task in this section is to prove an estimate for the quantity $D_{d,n}^{\mathrm{loc}}(A,B)$ defined in \eqref{Definition D loc}.

\begin{lemma}
\label{Lemma D loc}
Let $d \geq 2$ and $n \geq d$ with $(d,n) \neq (2,2)$. Assume that $B^{1/2} \leq A \leq B^2$. Then we have 
\begin{equation*}
D_{d,n}^{\mathrm{loc}}(A,B) = \iota_{d,n} A^{N_{d,n}-2} E_{d,n}(B) \left( 1 + O \left(\frac1{(\log A)^{1/2}} \right) \right).
\end{equation*}
\end{lemma}

\begin{proof}
Recall that the definitions of the lattice $\Lambda_{\nu_{d,n}(\mathbf{y})}^{(W)}$ and the region
$\mathcal{C}_{\nu_{d,n}(\mathbf{y})}^{(\alpha)}$ were respectively given in \eqref{Definition local lattice} and \eqref{Definition Cgamma}. We define the lattice
\begin{equation*}
\Gamma_{\mathbf{x}, \mathbf{y}}^{\mathrm{loc}}(W) =
\Lambda_{\nu_{d,n}(\mathbf{x})}^{(W)} \cap \Lambda_{\nu_{d,n}(\mathbf{y})}^{(W)},
\end{equation*}
and the region
\begin{equation*}
\mathcal{T}_{\mathbf{x}, \mathbf{y}}^{\mathrm{loc}}(A,\alpha) = \mathcal{B}_{N_{d,n}}(A)
\cap \mathcal{C}_{\nu_{d,n}(\mathbf{x})}^{(\alpha)} \cap \mathcal{C}_{\nu_{d,n}(\mathbf{y})}^{(\alpha)}.
\end{equation*}
Recall that the definitions of the sums $\mathcal{S}_k^{\ast} \left(\Lambda; \mathcal{R} \right)$ and
$\mathcal{S}_k \left(\Lambda; \mathcal{R} \right)$ for any given lattice $\Lambda \subset \mathbb{Z}^{N_{d,n}}$, any bounded region $\mathcal{R} \subset \mathbb{R}^{N_{d,n}}$ and any integer $k \geq 0$, were respectively given in \eqref{Definition Sk star} and \eqref{Definition Sk}. Recalling the definition \eqref{Definition Omega} of the set
$\Omega_{d,n}(B)$ we see that
\begin{equation}
\label{Equality D loc}
D_{d,n}^{\mathrm{loc}}(A,B) = \frac{\alpha^2 W^2}{8} \sum_{(\mathbf{x}, \mathbf{y}) \in \Omega_{d,n}(B)}
\frac{\mathcal{S}^{\ast}_2 \left( \Gamma_{\mathbf{x}, \mathbf{y}}^{\mathrm{loc}}(W); \mathcal{T}_{\mathbf{x}, \mathbf{y}}^{\mathrm{loc}}(A,\alpha) \right)}{||\nu_{d,n}(\mathbf{x})|| \cdot ||\nu_{d,n}(\mathbf{y})||}.
\end{equation}
A M\"{o}bius inversion gives
\begin{equation*}
\mathcal{S}^{\ast}_2 \left( \Gamma_{\mathbf{x}, \mathbf{y}}^{\mathrm{loc}}(W); \mathcal{T}_{\mathbf{x}, \mathbf{y}}^{\mathrm{loc}}(A,\alpha) \right) = \sum_{\ell \leq A} \frac{\mu(\ell)}{\ell^2} \mathcal{S}_2 \left( \Gamma_{\mathbf{x}, \mathbf{y}}^{\mathrm{loc}} \left(\frac{W}{\gcd(\ell,W)} \right);
\mathcal{T}_{\mathbf{x}, \mathbf{y}}^{\mathrm{loc}} \left( \frac{A}{\ell}, \alpha \right) \right).
\end{equation*}
For any integer $k \geq 0$ and any real $u \geq 1$ we clearly have
\begin{equation*}
\mathcal{S}_k \left( \Gamma_{\mathbf{x}, \mathbf{y}}^{\mathrm{loc}} \left(\frac{W}{\gcd(\ell,W)} \right);
\mathcal{T}_{\mathbf{x}, \mathbf{y}}^{\mathrm{loc}} \left( u, \alpha \right) \right) \leq
\mathcal{S}_k \left( \mathbb{Z}^{N_{d,n}}; \mathcal{B}_{N_{d,n}}(u) \right).
\end{equation*}
It thus follows that if $k \in \{0, \dots, N_{d,n}-1\}$ then
\begin{equation}
\label{Upper bound trivial t}
\mathcal{S}_k \left( \Gamma_{\mathbf{x}, \mathbf{y}}^{\mathrm{loc}} \left(\frac{W}{\gcd(\ell,W)} \right);
\mathcal{T}_{\mathbf{x}, \mathbf{y}}^{\mathrm{loc}} \left( u, \alpha \right) \right) \ll u^{N_{d,n}-k}.
\end{equation}
Using the upper bound \eqref{Upper bound trivial t} with $k=2$ and $u = A/\ell$ we deduce that
\begin{equation*}
\sum_{\ell > A/W} \frac{\mu(\ell)}{\ell^2} \mathcal{S}_2 \left( \Gamma_{\mathbf{x}, \mathbf{y}}^{\mathrm{loc}} \left(\frac{W}{\gcd(\ell,W)} \right); \mathcal{T}_{\mathbf{x}, \mathbf{y}}^{\mathrm{loc}} \left( \frac{A}{\ell}, \alpha \right) \right) \ll \frac{W^{N_{d,n}-1}}{A}.
\end{equation*}
Therefore, using the upper bound \eqref{Upper bound trivial t} once again with $k=2$ and $u = W$ we get
\begin{equation}
\label{Estimate l small}
\mathcal{S}^{\ast}_2 \left( \Gamma_{\mathbf{x}, \mathbf{y}}^{\mathrm{loc}}(W); \mathcal{T}_{\mathbf{x}, \mathbf{y}}^{\mathrm{loc}}(A,\alpha) \right) = \sum_{\ell \leq A/W} \frac{\mu(\ell)}{\ell^2} S_{\mathbf{x}, \mathbf{y}}^{\mathrm{loc}}(A, B; \ell) + O \left(W^{N_{d,n}-2} \right),
\end{equation}
where
\begin{equation*}
S_{\mathbf{x}, \mathbf{y}}^{\mathrm{loc}}(A, B; \ell) = \mathcal{S}_2 \left( \Gamma_{\mathbf{x}, \mathbf{y}}^{\mathrm{loc}} \left(\frac{W}{\gcd(\ell,W)} \right); \mathcal{T}_{\mathbf{x}, \mathbf{y}}^{\mathrm{loc}} \left( \frac{A}{\ell}, \alpha \right) \smallsetminus \mathcal{T}_{\mathbf{x}, \mathbf{y}}^{\mathrm{loc}} \left( W, \alpha \right) \right) .
\end{equation*}

Next, an application of partial summation yields
\begin{align*}
S_{\mathbf{x}, \mathbf{y}}^{\mathrm{loc}}(A, B; \ell) = & \
\frac{\ell^2}{A^2} \mathcal{S}_0 \left(\Gamma_{\mathbf{x}, \mathbf{y}}^{\mathrm{loc}} \left(\frac{W}{\gcd(\ell,W)} \right);
\mathcal{T}_{\mathbf{x}, \mathbf{y}}^{\mathrm{loc}} \left( \frac{A}{\ell}, \alpha \right) \right) \\
& + 2 \int_{W}^{A/\ell} \mathcal{S}_0 \left(\Gamma_{\mathbf{x}, \mathbf{y}}^{\mathrm{loc}} \left(\frac{W}{\gcd(\ell,W)} \right); \mathcal{T}_{\mathbf{x}, \mathbf{y}}^{\mathrm{loc}} \left( t, \alpha \right) \right)
\frac{\mathrm{d} t}{t^3} + O \left(W^{N_{d,n}-2} \right).
\end{align*}
Note that we have made use of the upper bound \eqref{Upper bound trivial t} with $k=0$ and $u = W$.
In addition, it is clear that we have
\begin{equation*}
\lambda_{N_{d,n}} \left( \Gamma_{\mathbf{x}, \mathbf{y}}^{\mathrm{loc}} \left(\frac{W}{\gcd(\ell,W)} \right) \right) \leq W.
\end{equation*}
Therefore, Lemma~\ref{Lemma lattice pivotal} shows that for any $t \in [W, A /\ell]$, we have
\begin{align*}
\mathcal{S}_0 \left(\Gamma_{\mathbf{x}, \mathbf{y}}^{\mathrm{loc}} \left(\frac{W}{\gcd(\ell,W)} \right);
\mathcal{T}_{\mathbf{x}, \mathbf{y}}^{\mathrm{loc}} \left( t, \alpha \right) \right) = & \
t^{N_{d,n}} \det \left( \Gamma_{\mathbf{x}, \mathbf{y}}^{\mathrm{loc}} \left( \frac{W}{\gcd(\ell,W)} \right) \right)^{-1} \\
& \times \left( \vol \left(\mathcal{T}_{\mathbf{x}, \mathbf{y}}^{\mathrm{loc}} ( 1, \alpha ) \right) + O \left( \frac{W}{t} \right) \right).
\end{align*}
We thus derive
\begin{align}
\label{Estimate after partial summation loc}
\begin{split}
S_{\mathbf{x}, \mathbf{y}}^{\mathrm{loc}}(A, B; \ell) = & \
\frac{N_{d,n}}{N_{d,n}-2} \frac{A^{N_{d,n}-2}}{\ell^{N_{d,n}-2}} \cdot
\det \left( \Gamma_{\mathbf{x}, \mathbf{y}}^{\mathrm{loc}} \left( \frac{W}{\gcd(\ell,W)} \right) \right)^{-1} \\
& \times \left( \vol \left(\mathcal{T}_{\mathbf{x}, \mathbf{y}}^{\mathrm{loc}} ( 1, \alpha ) \right)
+ O \left( \frac{W \ell}{A} \right) \right) + O \left( W^{N_{d,n}-2} \right).
\end{split}
\end{align}
Recalling the definition \eqref{Definition volume loc} of $\mathcal{J}(\mathbf{w}, \mathbf{z})$, we see that we have 
\begin{equation*}
\vol \left( \mathcal{T}_{\mathbf{x}, \mathbf{y}}^{\mathrm{loc}} ( 1, \alpha) \right) = 
\mathcal{J} \left(2 \alpha \frac{\nu_{d,n}(\mathbf{x})}{||\nu_{d,n}(\mathbf{x})||},
2 \alpha \frac{\nu_{d,n}(\mathbf{y})}{||\nu_{d,n}(\mathbf{y})||} \right).
\end{equation*}
Recall the definition \eqref{Definition Delta} of the quantity $\Delta(\mathbf{x}, \mathbf{y})$. Using
Lemma~\ref{Lemma volume loc} and the equality \eqref{Equality lattice} we deduce that
\begin{equation}
\label{Estimate volume loc}
\vol \left( \mathcal{T}_{\mathbf{x}, \mathbf{y}}^{\mathrm{loc}} ( 1, \alpha) \right) =
\frac{N_{d,n}-2}{N_{d,n}} V_{N_{d,n}-2} \frac{\Delta(\mathbf{x}, \mathbf{y})}{\alpha^2}
\left( 1 + O \left( \min \left\{ 1, \frac{\Delta(\mathbf{x}, \mathbf{y})^2}{\alpha^2} \right\} \right) \right).
\end{equation}
In addition, Lemmas~\ref{Lemma determinant mixed/local} and \ref{Lemma G} and the estimate \eqref{Estimate gcd} give
\begin{align}
\nonumber
\det \left( \Gamma_{\mathbf{x}, \mathbf{y}}^{\mathrm{loc}} \left(\frac{W}{\gcd(\ell,W)} \right) \right) & =
\frac{W^2}{\gcd(\ell,W)^2} \cdot \gcd\left(\mathcal{G}(\mathbf{x}, \mathbf{y}), \frac{W}{\gcd(\ell,W)} \right)^{-1} \\
\label{Estimate det loc}
& = \frac{W^2}{\gcd(\ell,W)^2 \mathcal{G}(\mathbf{x}, \mathbf{y})}
\left( 1 + O \left( \boldsymbol{1}_{\mathcal{G}(\mathbf{x}, \mathbf{y}) \nmid W/\rad(W)} \right) \right).
\end{align}
Recall the definition \eqref{Definition localised error} of the quantity $\mathcal{E}_{\mathbf{x},\mathbf{y}}(B) $. Putting together the estimates \eqref{Estimate after partial summation loc}, \eqref{Estimate volume loc} and
\eqref{Estimate det loc} and using the lower bound $\Delta(\mathbf{x}, \mathbf{y}) \geq 1$, we obtain
\begin{align*}
S_{\mathbf{x}, \mathbf{y}}^{\mathrm{loc}}(A, B; \ell) = & \ 
V_{N_{d,n}-2} \frac{A^{N_{d,n}-2}}{\alpha^2 W^2} 
\frac{\gcd(\ell,W)^2}{\ell^{N_{d,n}-2}} \Delta(\mathbf{x}, \mathbf{y}) \mathcal{G}(\mathbf{x}, \mathbf{y}) \\
& \times \left( 1 + O \left( \mathcal{E}_{\mathbf{x},\mathbf{y}}(B) + \frac{\alpha^2 W \ell}{A} \right) \right)
+ O \left( W^{N_{d,n}-2} \right).
\end{align*}
Recall the upper bound \eqref{Upper bound W} for $W$ and note that the assumption $B \leq A^2$ gives
$W \ll A^{8 / \log \log A}$. Therefore, using the inequalities
$\Delta(\mathbf{x}, \mathbf{y}), \mathcal{G}(\mathbf{x}, \mathbf{y}) \geq 1$, we deduce from the estimate \eqref{Estimate l small} that
\begin{align*}
\mathcal{S}^{\ast}_2 \left( \Gamma_{\mathbf{x}, \mathbf{y}}^{\mathrm{loc}}(W); \mathcal{T}_{\mathbf{x}, \mathbf{y}}^{\mathrm{loc}}(A,\alpha) \right) = & \ V_{N_{d,n}-2} \frac{A^{N_{d,n}-2}}{\alpha^2 W^2}
\Delta(\mathbf{x}, \mathbf{y}) \mathcal{G}(\mathbf{x}, \mathbf{y}) \\
\times & \left( \sum_{\ell \leq A/W} \mu(\ell) \frac{\gcd(\ell,W)^2}{\ell^{N_{d,n}}} 
+ O \left( \mathcal{E}_{\mathbf{x},\mathbf{y}}(B) + \frac{\alpha^2 W}{A} \right) \right).
\end{align*}
Arguing as in the proof of the estimate \eqref{Estimate zeta}, we see that 
\begin{equation*}
\sum_{\ell \leq A/W} \mu(\ell) \frac{\gcd(\ell, W)^2}{\ell^{N_{d,n}}} = \frac1{\zeta(N_{d,n}-2)}
+ O \left( \left( \frac{W}{A} \right)^{N_{d,n}-3} + \frac1{w^{N_{d,n}-3}} \right).
\end{equation*}
Recall the definition \eqref{Definition F} of the quantity $F_{d,n}(B)$. We remark that the respective definitions \eqref{Definition alpha} and \eqref{Definition w} of $\alpha$ and $w$ show that $w^{N_{d,n}-3} \ll A/\alpha^2W$. As a result, Lemma~\ref{Lemma E} and the equalities \eqref{Equality Delta} and \eqref{Equality D loc} yield
\begin{equation*}
D_{d,n}^{\mathrm{loc}}(A,B) = \iota_{d,n} A^{N_{d,n}-2} E_{d,n}(B)
\left( 1 + O \left( \frac1{w^{N_{d,n}-3}} + \frac{F_{d,n}(B)}{B^2} \right) \right).
\end{equation*}
We complete the proof by applying Lemma~\ref{Lemma F} and by using the assumption $A \leq B^2$ and the fact that $w^{N_{d,n}-3} \gg (\log A)^{1/2}$.
\end{proof}

\subsection{Proof of the key variance upper bound}

\label{Section proof variance}

We now combine the tools developed in Sections~\ref{Section first} and \ref{Section second} in order to establish Proposition~\ref{Proposition variance}.

\begin{proof}[Proof of Proposition~\ref{Proposition variance}]
Recall the respective definitions \eqref{Definition Delta mixed} and \eqref{Definition Delta loc} of the two quantities
$\Delta_V^{\mathrm{mix}}(B)$ and $\Delta_V^{\mathrm{loc}}(B)$. It is convenient to set
\begin{equation*}
K(A,B) = \sum_{V \in \mathbb{V}_{d,n}(A)} \left( N_V(B) + \Delta_V^{\mathrm{mix}}(B) + \Delta_V^{\mathrm{loc}}(B) \right).
\end{equation*}
Expanding the square, we see that
\begin{equation*}
\sum_{V \in \mathbb{V}_{d,n}(A)} \! \! \left( N_V(B) - N_V^{\mathrm{loc}}(B) \right)^2 \! =
D_{d,n}(A,B) - 2 D_{d,n}^{\mathrm{mix}}(A,B) + D_{d,n}^{\mathrm{loc}}(A,B) +O(K(A,B)).
\end{equation*}
It follows from the lower bound \eqref{Lower bound Vdn} and Lemmas~\ref{Lemma E}, \ref{Lemma D}, \ref{Lemma D44}, \ref{Lemma D mixed} and \ref{Lemma D loc} that
\begin{equation*}
\frac1{\# \mathbb{V}_{d,n}(A)}
\left( D_{d,n}(A,B) - 2 D_{d,n}^{\mathrm{mix}}(A,B) + D_{d,n}^{\mathrm{loc}}(A,B) \right) \ll
\frac{B^2}{A^2} \cdot \frac1{(\log A)^{1/2}}.
\end{equation*}
We thus derive
\begin{equation}
\label{Upper bound with K}
\frac1{\# \mathbb{V}_{d,n}(A)} \sum_{V \in \mathbb{V}_{d,n}(A)} \left( N_V(B) - N_V^{\mathrm{loc}}(B) \right)^2
\ll \frac{B^2}{A^2} \cdot \frac1{(\log A)^{1/2}} + \frac{K(A,B)}{\# \mathbb{V}_{d,n}(A)}.
\end{equation}
Recall the definition \eqref{Definition alpha} of $\alpha$ and the upper bound \eqref{Upper bound W} for $W$. We trivially have
\begin{equation*}
\Delta_V^{\mathrm{mix}}(B) \leq \frac{\alpha W}{||\mathbf{a}_V||} N_V(B).
\end{equation*}
Hence, using partial summation it follows from Lemma~\ref{Lemma first moment} that 
\begin{equation}
\label{Upper bound Delta mixed}
\frac1{\#\mathbb{V}_{d,n}(A)} \sum_{V \in \mathbb{V}_{d,n}(A)} \Delta_V^{\mathrm{mix}}(B) \ll
\frac{B^{1+5/\log \log B}}{A^2}.
\end{equation}
Moreover, using the trivial upper bound
\begin{equation*}
\Delta_V^{\mathrm{loc}}(B) \ll \frac{\alpha^2 W^2}{||\mathbf{a}_V||^2}
\sum_{\mathbf{x} \in \Xi_{d,n}(B)} \frac1{||\nu_{d,n}(\mathbf{x})||},
\end{equation*}
we obtain
\begin{equation}
\label{Upper bound Delta loc}
\frac1{\#\mathbb{V}_{d,n}(A)} \sum_{V \in \mathbb{V}_{d,n}(A)} \Delta_V^{\mathrm{loc}}(B) \ll \frac{B^{1+9/\log \log B}}{A^2}.
\end{equation}
Applying Lemma~\ref{Lemma first moment} and using the upper bounds \eqref{Upper bound Delta mixed} and
\eqref{Upper bound Delta loc} together with the assumption $B \ll A (\log A)^{1/2}$, we derive
\begin{equation}
\label{Upper bound K}
\frac{K(A,B)}{\#\mathbb{V}_{d,n}(A)} \ll \frac{B}{A}.
\end{equation}
Putting together the upper bounds \eqref{Upper bound with K} and \eqref{Upper bound K}, we immediately see that the assumption $B \ll A (\log A)^{1/2}$ allows us to complete the proof.
\end{proof}

\section{The localised counting function is rarely small}

\label{Section localised rarely small}

In Section~\ref{Section factors} we start by introducing certain non-Archimedean and Archimedean factors that will arise throughout the dissection of our localised counting function $N_V^{\mathrm{loc}}(B)$. We then check that
Proposition~\ref{Proposition 2} follows from upper bounds for the number of
$V \in \mathbb{V}_{d,n}^{\mathrm{loc}}(A)$ at which one of these two factors is exceptionally small, as stated in
Propositions~\ref{Proposition non-Archimedean} and \ref{Proposition Archimedean}. We finally turn to the proofs of Propositions~\ref{Proposition non-Archimedean} and \ref{Proposition Archimedean} in
Sections~\ref{Section non-Archimedean} and ~\ref{Section Archimedean}, respectively.

\subsection{The local factors}

\label{Section factors}

Given $N \geq 1$, recall the definition \eqref{Definition Cgamma} of the region $\mathcal{C}_{\mathbf{v}}^{(\gamma)}$ for any real $\gamma > 0$ and any $\mathbf{v} \in \mathbb{R}^N$. For $\mathbf{a} \in \mathbb{R}^{N_{d,n}}$ and
$\gamma > 0$, we introduce the Archimedean factor
\begin{equation}
\label{Definition tau}
\tau(\mathbf{a};\gamma) = \gamma \cdot \vol \left( \left\{ \mathbf{u} \in \mathcal{B}_{n+1}(1) :
\mathbf{a} \in \mathcal{C}_{\nu_{d,n}(\mathbf{u})}^{(\gamma)} \right\} \right).
\end{equation}
In addition, given $Q \geq 1$ and $\mathbf{b} \in (\mathbb{Z}/Q \mathbb{Z})^N$ we let
$\gcd(Q, \mathbf{b})$ denote the greatest common divisor of $Q$ and the coordinates of the vector $\mathbf{b}$. In analogy with the Archimedean setting it is convenient for our purpose to define
\begin{equation}
\label{Definition non-Archimedean ball}
\mathfrak{R}_N(Q) = \left\{ \mathbf{b} \in (\mathbb{Z}/Q\mathbb{Z})^N : \gcd(Q, \mathbf{b}) = 1 \right\}.
\end{equation}
Recall the definition \eqref{Definition local lattice} of the lattice $\Lambda_{\mathbf{c}}^{(Q)}$ for given $Q \geq 1$ and $\mathbf{c} \in \mathbb{Z}^N$. For $\mathbf{a} \in \mathbb{Z}^{N_{d,n}}$ and $Q \geq 1$, we introduce the
non-Archimedean factor
\begin{equation}
\label{Definition sigma}
\sigma(\mathbf{a};Q) = \frac1{Q^n} \cdot \# \left\{ \mathbf{b} \in \mathfrak{R}_{n+1}(Q) :
\mathbf{a} \in \Lambda_{\nu_{d,n}(\mathbf{b})}^{(Q)} \right\}.
\end{equation}
We note that for any vector $\mathbf{b} \in \mathfrak{R}_{n+1}(Q)$ the lattice $\Lambda_{\nu_{d,n}(\mathbf{b})}^{(Q)}$ is well-defined.

Recall the respective definitions \eqref{Definition alpha} and \eqref{Definition W} of the quantity $\alpha$ and the integer $W$. Given $V \in \mathbb{V}_{d,n}$ we set
\begin{equation}
\label{Definition singular integral}
\mathfrak{J}_V(B) = \tau(\mathbf{a}_V;\alpha),
\end{equation}
and
\begin{equation}
\label{Definition singular series}
\mathfrak{S}_V(B) = \sigma(\mathbf{a}_V;W).
\end{equation}
One may check that $\mathfrak{J}_V(B)$ converges to the usual singular integral for the problem at hand as $B$ tends to $\infty$. Similarly, using the Chinese remainder theorem it is possible to show that $\mathfrak{S}_V(B)$ converges to the singular series as $B$ tends to $\infty$. We shall use neither of these facts in our work, however.

Recall the definition \eqref{Definition local counting} of our localised counting function $N_V^{\mathrm{loc}}(B)$. We prove the following upper bound for the product of the local factors.

\begin{lemma}
\label{Lemma product factors}
Let $d \geq 2$ and $n \geq d$. For any $V \in \mathbb{V}_{d,n}(A)$, we have
\begin{equation*}
\mathfrak{S}_V(B) \cdot \mathfrak{J}_V(B) \ll \frac{A}{B} \cdot N_V^{\mathrm{loc}}(B) + \frac1{B^{1/n}}.
\end{equation*}
\end{lemma}

\begin{proof}
Recall the definition \eqref{Definition Xi} of the set $\Xi_{d,n}(B)$. We start by noting that for any
$\mathbf{x} \in \Xi_{d,n}(B)$ we have $||\nu_{d,n}(\mathbf{x})|| \ll B^{d/(n+1-d)}$. Using the fact that
$||\mathbf{a}_V|| \leq A$ for $V \in \mathbb{V}_{d,n}(A)$, we deduce that
\begin{equation}
\label{Lower bound localised}
\sum_{\substack{\mathbf{x} \in \Xi_{d,n}(B) \\
\mathbf{a}_V \in \Lambda_{\nu_{d,n}(\mathbf{x})}^{(W)} \cap \mathcal{C}_{\nu_{d,n}(\mathbf{x})}^{(\alpha)}}} 1 \ll
\frac{AB^{d/(n+1-d)}}{\alpha W} \cdot N_V^{\mathrm{loc}}(B).
\end{equation}
Breaking the summation into residue classes modulo $W$, we obtain
\begin{equation}
\label{Equality congruence}
\sum_{\substack{\mathbf{x} \in \Xi_{d,n}(B) \\
\mathbf{a}_V \in \Lambda_{\nu_{d,n}(\mathbf{x})}^{(W)} \cap \mathcal{C}_{\nu_{d,n}(\mathbf{x})}^{(\alpha)}}} 1 =
\sum_{\substack{\mathbf{b} \in \mathfrak{R}_{n+1}(W) \\ \mathbf{a}_V \in \Lambda_{\nu_{d,n}(\mathbf{b})}^{(W)}}}
\# \left\{ \mathbf{x} \in \Xi_{d,n}(B) :
\begin{array}{l l}
\mathbf{x} \equiv \mathbf{b} \bmod{W} \\
\mathbf{a}_V \in \mathcal{C}_{\nu_{d,n}(\mathbf{x})}^{(\alpha)}
\end{array}
\right\}.
\end{equation}
We proceed to use a M\"obius inversion to handle the condition that the vectors $\mathbf{x}$ are primitive. Note that for any non-zero real number $t$ and any vector $\mathbf{z} \in \mathbb{Z}^{n+1}$, we have 
\begin{equation}
\label{Equality Calpha invariant}
\mathcal{C}_{\nu_{d,n}(t \mathbf{z})}^{(\alpha)} = \mathcal{C}_{\nu_{d,n}(\mathbf{z})}^{(\alpha)}.
\end{equation}
Given $\mathbf{b} \in \mathfrak{R}_{n+1}(W)$, it follows that
\begin{equation}
\label{Equality after Mobius}
\# \left\{ \mathbf{x} \in \Xi_{d,n}(B) :
\begin{array}{l l}
\mathbf{x} \equiv \mathbf{b} \bmod{W} \\
\mathbf{a}_V \in \mathcal{C}_{\nu_{d,n}(\mathbf{x})}^{(\alpha)}
\end{array}
\right\}
= \sum_{\substack{k \leq B^{1/(n+1-d)} \\ \gcd(k,W)=1}} \mu(k) M_k(V;B),
\end{equation}
where
\begin{equation*}
M_k(V;B) = \# \left\{ \mathbf{z} \in \mathbb{Z}^{n+1} \smallsetminus \{\boldsymbol{0}\} :
\begin{array}{l l}
||\mathbf{z}|| \leq B^{1/(n+1-d)}/k \\
k \mathbf{z} \equiv \mathbf{b} \bmod{W} \\
\mathbf{a}_V \in \mathcal{C}_{\nu_{d,n}(\mathbf{z})}^{(\alpha)}
\end{array}
\right\}.
\end{equation*}
A trivial application of the lattice point counting result \cite[Theorem~$1.3$]{MR3264671} shows that
\begin{equation*}
M_k(V;B) = \frac1{W^{n+1}} \cdot \vol \left( \left\{ \mathbf{v} \in \mathbb{R}^{n+1} :
\begin{array}{l l}
||\mathbf{v}|| \leq B^{1/(n+1-d)}/k \\
\mathbf{a}_V \in \mathcal{C}_{\nu_{d,n}(\mathbf{v})}^{(\alpha)}
\end{array}
\right\} \right)
+ O \left( \frac{B^{n/(n+1-d)}}{k^n} \right).
\end{equation*}
Recall the definition \eqref{Definition singular integral} of the Archimedean factor $\mathfrak{J}_V(B)$. Making the change of variables $\mathbf{v} = B^{1/(n+1-d)} \mathbf{u} / k$ and using the equality \eqref{Equality Calpha invariant} again, we see that
\begin{equation*}
M_k(V;B) = \frac{B^{(n+1)/(n+1-d)}}{W^{n+1} k^{n+1}} \cdot \frac{\mathfrak{J}_V(B)}{\alpha} +
O \left( \frac{B^{n/(n+1-d)}}{k^n} \right).
\end{equation*}
In addition, the respective definitions \eqref{Definition W} and \eqref{Definition w} of the integer $W$ and the quantity $w$ easily yield
\begin{equation*}
\sum_{\substack{k \leq B^{1/(n+1-d)} \\ \gcd(k,W)=1}} \frac{\mu(k)}{k^{n+1}} = 1 + O \left( \frac1{w^n} \right).
\end{equation*}
As a result, we deduce from the equality \eqref{Equality after Mobius} that
\begin{equation*}
\frac{B^{(n+1)/(n+1-d)}}{W^{n+1}} \cdot \frac{\mathfrak{J}_V(B)}{\alpha} \ll
\# \left\{ \mathbf{x} \in \Xi_{d,n}(B) :
\begin{array}{l l}
\mathbf{x} \equiv \mathbf{b} \bmod{W} \\
\mathbf{a}_V \in \mathcal{C}_{\nu_{d,n}(\mathbf{x})}^{(\alpha)}
\end{array}
\right\}
+ B^{n/(n+1-d)}.
\end{equation*}
We now sum this upper bound over the vectors $\mathbf{b} \in \mathfrak{R}_{n+1}(W)$ such that
$\mathbf{a}_V \in \Lambda_{\nu_{d,n}(\mathbf{b})}^{(W)}$. Combining the equality \eqref{Equality congruence} and the upper bound \eqref{Lower bound localised} and recalling the definition \eqref{Definition singular series} of the
non-Archimedean factor $\mathfrak{S}_V(B)$, we derive
\begin{equation*}
B^{(n+1)/(n+1-d)} \cdot \frac{\mathfrak{S}_V(B)}{W} \cdot \frac{\mathfrak{J}_V(B)}{\alpha} \ll
\frac{AB^{d/(n+1-d)}}{\alpha W} \cdot N_V^{\mathrm{loc}}(B) + W^{n+1} B^{n/(n+1-d)}.
\end{equation*}
Recalling the definition \eqref{Definition alpha} of $\alpha$ and the upper bound \eqref{Upper bound W} for $W$, we see that this completes the proof.
\end{proof}

Recall the definition \eqref{Definition Vdnloc} of the space $\mathbb{V}_{d,n}^{\mathrm{loc}}(A)$ of hypersurfaces of height at most $A$ which are everywhere locally soluble. We now state upper bounds for the frequencies of occurrence of particularly small values of the local factors $\mathfrak{S}_V(B)$ and $\mathfrak{J}_V(B)$ as $V$ runs over the set
$\mathbb{V}_{d,n}^{\mathrm{loc}}(A)$, under the assumption that the ratio $B/A$ does not grow too rapidly as $A$ tends to $\infty$. It is worth highlighting the fact that these two results hold without any restriction on $d \geq 2$ and
$n \geq 3$.

The following statement is concerned with the non-Archimedean factor.

\begin{proposition}
\label{Proposition non-Archimedean}
Let $d \geq 2$ and $n \geq 3$. Let $\phi : \mathbb{R}_{>0} \to \mathbb{R}_{>1}$ be such that $\phi(A) \leq A$ and let $C>0$. Then we have
\begin{equation*}
\frac1{\# \mathbb{V}_{d,n}^{\mathrm{loc}}(A)} \cdot
\# \left\{V \in \mathbb{V}_{d,n}^{\mathrm{loc}}(A) : \mathfrak{S}_V(A\phi(A)) < \frac{C}{\phi(A)^{1/6}} \right\} \ll \frac1{\phi(A)^{1/24n}},
\end{equation*}
where the implied constant may depend on $C$.
\end{proposition}

The next result deals with the Archimedean factor.

\begin{proposition}
\label{Proposition Archimedean}
Let $d \geq 2$ and $n \geq 3$. Let $\phi : \mathbb{R}_{>0} \to \mathbb{R}_{>1}$ be such that $\phi(A) \leq A$ and let $C>0$. Then we have
\begin{equation*}
\frac1{\#\mathbb{V}_{d,n}^{\mathrm{loc}}(A)} \cdot
\# \left\{ V \in \mathbb{V}_{d,n}^{\mathrm{loc}}(A) : \mathfrak{J}_V(A\phi(A)) < \frac{C}{\phi(A)^{1/6}} \right\} \ll \frac1{\phi(A)^{1/6n}},
\end{equation*}
where the implied constant may depend on $C$.
\end{proposition}

Propositions~\ref{Proposition non-Archimedean} and \ref{Proposition Archimedean} will respectively be established in Sections~\ref{Section non-Archimedean} and ~\ref{Section Archimedean}. We now have everything in place to provide the proof of Proposition~\ref{Proposition 2}.

\begin{proof}[Proof of Proposition~\ref{Proposition 2}]
It is convenient to set
\begin{equation*}
\mathscr{Q}_{\phi}(A) = \frac1{\# \mathbb{V}_{d,n}^{\mathrm{loc}}(A)} \cdot
\# \left\{ V \in \mathbb{V}_{d,n}^{\mathrm{loc}}(A) : N_V^{\mathrm{loc}}(A \phi(A)) \leq \phi(A)^{2/3} \right\}.
\end{equation*}
Lemma \ref{Lemma product factors} and the assumption $\phi(A) \leq A^{3/n}$ imply that there exists a constant
$c > 0$ depending at most on $d$ and $n$ such that
\begin{equation*}
\mathscr{Q}_{\phi}(A) \leq \frac1{\#\mathbb{V}_{d,n}^{\mathrm{loc}}(A)} \cdot
\# \left\{ V \in \mathbb{V}_{d,n}^{\mathrm{loc}}(A) : \mathfrak{S}_V(A \phi(A)) \cdot \mathfrak{J}_V(A \phi(A)) \leq
\frac{c}{\phi(A)^{1/3}} \right\}.
\end{equation*}
Therefore, we also have
\begin{equation*}
\mathscr{Q}_{\phi}(A) \leq \frac1{\#\mathbb{V}_{d,n}^{\mathrm{loc}}(A)} \cdot
\# \left\{ V \in \mathbb{V}_{d,n}^{\mathrm{loc}}(A) :
\min \left\{ \mathfrak{S}_V(A \phi(A)), \mathfrak{J}_V(A \phi(A)) \right\} \leq \frac{c^{1/2}}{\phi(A)^{1/6}} \right\}.
\end{equation*}
Our assumption $\phi(A) \leq A^{3/n}$ allows us to apply Propositions \ref{Proposition non-Archimedean} and
\ref{Proposition Archimedean} to conclude that 
\begin{equation*}
\mathscr{Q}_{\phi}(A) \ll \frac1{\phi(A)^{1/24n}},
\end{equation*}
which completes the proof of Proposition~\ref{Proposition 2}.
\end{proof}

\subsection{The non-Archimedean factor is rarely small}

\label{Section non-Archimedean}

The purpose of this section is to provide the proof of Proposition~\ref{Proposition non-Archimedean}. Given
$N, Q \geq 1$, recall the respective definitions \eqref{Definition non-Archimedean ball} and \eqref{Definition sigma} of the set $\mathfrak{R}_N(Q)$ and the quantity $\sigma(\mathbf{a};Q)$ for $\mathbf{a} \in \mathbb{Z}^{N_{d,n}}$. We note that given a prime number $p$ and $r \geq 1$, we clearly have
\begin{equation}
\label{Equality cardinality B}
\# \mathfrak{R}_N(p^r) = p^{rN} \left( 1 - \frac1{p^N} \right).
\end{equation}
In addition, it follows from the Chinese remainder theorem that for any $\mathbf{a} \in \mathfrak{R}_{N_{d,n}}(Q)$ we have
\begin{equation}
\label{Equality multiplicativity}
\sigma(\mathbf{a};Q) = \prod_{p^r \| Q} \sigma(\mathbf{a};p^r),
\end{equation}
where the notation $p^r \| Q$ means that $p$ is a prime number dividing $Q$ and $r$ is the $p$-adic valuation of $Q$. The following pair of results will allow us to handle situations in which one of the factors $\sigma(\mathbf{a};p^r)$ is somewhat large despite the fact that $\sigma(\mathbf{a};Q)$ is assumed to be small. We shall start by estimating the first moment of the quantity $\sigma(\mathbf{a};p^r)$ as $\mathbf{a}$ runs over the set $\mathfrak{R}_{N_{d,n}}(p^r)$.

\begin{lemma}
\label{Lemma first moment sigma}
Let $d \geq 2$ and $n \geq 3$. Let also $p$ be a prime number and $r \geq 1$. We have
\begin{equation*}
\frac1{\# \mathfrak{R}_{N_{d,n}}(p^r)} \sum_{\mathbf{a} \in \mathfrak{R}_{N_{d,n}}(p^r)} \sigma(\mathbf{a};p^r) = 1 + O \left( \frac1{p^{n+1}} \right).
\end{equation*}
\end{lemma}

\begin{proof}
Inverting the order of summation we obtain
\begin{equation*}
\sum_{\mathbf{a} \in \mathfrak{R}_{N_{d,n}}(p^r)} \sigma(\mathbf{a};p^r) =
\frac1{p^{rn}} \sum_{\mathbf{b} \in \mathfrak{R}_{n+1}(p^r)}
\# \left\{ \mathbf{a} \in \mathfrak{R}_{N_{d,n}}(p^r) : \mathbf{a} \in \Lambda_{\nu_{d,n}(\mathbf{b})}^{(p^r)} \right\}.
\end{equation*}
Recall the definition \eqref{Definition local lattice} of the lattice $\Lambda_{\nu_{d,n}(\mathbf{b})}^{(p^r)}$ and observe that if $p\nmid \mathbf{b}$ then $p\nmid \nu_{d,n}(\mathbf{b})$. We deduce that
\begin{align*}
\# \left\{ \mathbf{a} \in \mathfrak{R}_{N_{d,n}}(p^r) : \mathbf{a} \in \Lambda_{\nu_{d,n}(\mathbf{b})}^{(p^r)} \right\}
& = p^{r(N_{d,n}-1)} - p^{(r-1)(N_{d,n}-1)} \\
& = p^{r(N_{d,n}-1)} \left( 1 - \frac1{p^{N_{d,n}-1}} \right).
\end{align*}
We thus get
\begin{equation*}
\frac1{\# \mathfrak{R}_{N_{d,n}}(p^r)} \sum_{\mathbf{a} \in \mathfrak{R}_{N_{d,n}}(p^r)} \sigma(\mathbf{a};p^r) =
\frac{\# \mathfrak{R}_{n+1}(p^r)}{p^{r(n+1)}} \cdot \frac{p^{rN_{d,n}}}{\# \mathfrak{R}_{N_{d,n}}(p^r)}
\left( 1 - \frac1{p^{N_{d,n}-1}} \right).
\end{equation*}
We see that two applications of the equality \eqref{Equality cardinality B} allow us to complete the proof.
\end{proof}

We now establish an upper bound for the variance of the quantity $\sigma(\mathbf{a};p^r)$ as $\mathbf{a}$ runs over the set $\mathfrak{R}_{N_{d,n}}(p^r)$.

\begin{lemma}
\label{Lemma second moment sigma}
Let $d \geq 2$ and $n \geq 3$. Let also $p$ be a prime number and $r \geq 1$. We have
\begin{equation*}
\frac1{\# \mathfrak{R}_{N_{d,n}}(p^r)} \sum_{\mathbf{a} \in \mathfrak{R}_{N_{d,n}}(p^r)}
\left( \sigma(\mathbf{a};p^r) - 1 \right)^2 \ll \frac1{p^{n-1}}.
\end{equation*}
\end{lemma}

\begin{proof}
We start by estimating the second moment of $\sigma(\mathbf{a};p^r)$ as $\mathbf{a}$ runs over
$\mathfrak{R}_{N_{d,n}}(p^r)$. We have
\begin{equation}
\label{Definition L}
\sum_{\mathbf{a} \in \mathfrak{R}_{N_{d,n}}(p^r)} \sigma(\mathbf{a};p^r)^2 = \frac1{p^{2rn}}
\sum_{\mathbf{b}_1, \mathbf{b}_2 \in \mathfrak{R}_{n+1}(p^r)} \# L(\mathbf{b}_1,\mathbf{b}_2;p^r),
\end{equation}
where
\begin{equation*}
L(\mathbf{b}_1,\mathbf{b}_2;p^r) =
\left\{ \mathbf{a} \in \mathfrak{R}_{N_{d,n}}(p^r) : \mathbf{a} \in \Lambda_{\nu_{d,n}(\mathbf{b}_1)}^{(p^r)} \cap
\Lambda_{\nu_{d,n}(\mathbf{b}_2)}^{(p^r)} \right\}.
\end{equation*}

We first investigate the cardinality of the set $L(\mathbf{b}_1,\mathbf{b}_2;p^r)$ under the assumption that there does not exist $g \in (\mathbb{Z}/p\mathbb{Z})^{\times}$ such that $\mathbf{b}_1 = g \mathbf{b}_2$. In this case we select any primitive representatives $\mathbf{c}_1, \mathbf{c}_2 \in \mathbb{Z}^{n+1}$ of $\mathbf{b}_1$ and $\mathbf{b}_2$ respectively, and we note that $\mathbf{c}_1$ and $\mathbf{c}_2$ are linearly independent. Appealing to the equality \eqref{Equality after Smith} we deduce that there exists $f \in \mathbb{Z}$ such that the cardinality of the set $L(\mathbf{b}_1,\mathbf{b}_2;p^r)$ is equal to
\begin{equation*}
\# \left\{ \mathbf{z} \in \mathfrak{R}_{N_{d,n}}(p^r) :
\begin{pmatrix}
1 & 0 & 0 & \cdots & 0 \\
f & \mathcal{G}(\nu_{d,n}(\mathbf{c}_1), \nu_{d,n}(\mathbf{c}_2)) & 0 & \cdots & 0 
\end{pmatrix}
\mathbf{z} \equiv \boldsymbol{0} \bmod{p^r} \right\}.
\end{equation*}
Using Lemma~\ref{Lemma G} and the fact that
$\gcd(\mathcal{G}(\mathbf{c}_1,\mathbf{c}_2),p^r) = \gcd(\mathcal{G}(\mathbf{b}_1,\mathbf{b}_2), p^r)$, we obtain
\begin{equation*}
\# L(\mathbf{b}_1,\mathbf{b}_2;p^r) = p^{r(N_{d,n}-2)} \gcd(\mathcal{G}(\mathbf{b}_1,\mathbf{b_2}), p^r) - p^{(r-1)(N_{d,n}-2)} \gcd(\mathcal{G}(\mathbf{b}_1,\mathbf{b_2}), p^{r-1}).
\end{equation*}
In the case where there exists $g \in (\mathbb{Z}/p\mathbb{Z})^{\times}$ such that $\mathbf{b}_1 = g\mathbf{b}_2$, we have
\begin{equation*}
\# L(\mathbf{b}_1,\mathbf{b}_2;p^r) = \#
\left\{ \mathbf{a} \in \mathfrak{R}_{N_{d,n}}(p^r) : \mathbf{a} \in \Lambda_{\nu_{d,n}(\mathbf{b}_1)}^{(p^r)} \right\},
\end{equation*}
and since $p \nmid \mathbf{b}_1$ we get
\begin{equation*}
\# L(\mathbf{b}_1,\mathbf{b}_2;p^r) = p^{r(N_{d,n}-1)} - p^{(r-1)(N_{d,n}-1)}.
\end{equation*}
As a result, we have proved in particular that in both cases we have
\begin{equation}
\label{Estimate L}
\# L(\mathbf{b}_1,\mathbf{b}_2;p^r) = p^{r(N_{d,n}-2)} \gcd(\mathcal{G}(\mathbf{b}_1,\mathbf{b_2}),p^r)
\left( 1 + O \left( \frac1{p^{N_{d,n}-2}} \right) \right).
\end{equation}

Next, we note that
\begin{equation*}
\sum_{\mathbf{b}_1, \mathbf{b}_2 \in \mathfrak{R}_{n+1}(p^r)} \gcd(\mathcal{G}(\mathbf{b}_1,\mathbf{b_2}),p^r) =
\# \mathfrak{R}_{n+1}(p^r)^2 + O \left( \sum_{e=1}^r p^e \cdot \# \mathfrak{F}^{(e)}(p^r) \right),
\end{equation*}
where, for $e \in \{1, \dots, r\}$, we have introduced the set
\begin{equation*}
\mathfrak{F}^{(e)}(p^r) = \left\{ (\mathbf{b}_1,\mathbf{b}_2) \in \mathfrak{R}_{n+1}(p^r)^2 :
\gcd(\mathcal{G}(\mathbf{b}_1,\mathbf{b_2}),p^r) = p^e \right\}.
\end{equation*}
Furthermore, for given $\mathbf{b}_1 \in \mathfrak{R}_{n+1}(p^r)$, a little thought reveals that the number of
$\mathbf{b}_2 \in \mathfrak{R}_{n+1}(p^r)$ such that $(\mathbf{b}_1,\mathbf{b}_2) \in \mathfrak{F}^{(e)}(p^r)$ is at most $p^{r(n+1)-en}$. This yields 
\begin{equation*}
\# \mathfrak{F}^{(e)}(p^r) \ll p^{2r(n+1)-en}.
\end{equation*}
Using the equality \eqref{Equality cardinality B}, we deduce
\begin{equation}
\label{Estimate sum G final}
\sum_{\mathbf{b}_1, \mathbf{b}_2 \in \mathfrak{R}_{n+1}(p^r)} \gcd(\mathcal{G}(\mathbf{b}_1,\mathbf{b_2}),p^r) =
p^{2r(n+1)} \left( 1 + O \left( \frac1{p^{n-1}} \right) \right).
\end{equation}

Putting together the equality \eqref{Definition L} and the estimates \eqref{Estimate L} and \eqref{Estimate sum G final}, we derive
\begin{equation*}
\sum_{\mathbf{a} \in \mathfrak{R}_{N_{d,n}}(p^r)} \sigma(\mathbf{a};p^r)^2 = p^{rN_{d,n}} \left( 1 + O \left( \frac1{p^{n-1}} \right) \right).
\end{equation*}
In view of the equality \eqref{Equality cardinality B}, this can eventually be rewritten as
\begin{equation}
\label{Estimate second moment sigma}
\frac1{ \# \mathfrak{R}_{N_{d,n}}(p^r)} \sum_{\mathbf{a} \in \mathfrak{R}_{N_{d,n}}(p^r)} \sigma(\mathbf{a};p^r)^2 =
1 + O \left( \frac1{p^{n-1}} \right).
\end{equation}
It is now immediate to check that we may complete the proof by combining Lemma~\ref{Lemma first moment sigma} and the estimate \eqref{Estimate second moment sigma}.
\end{proof}

Given an integer $Q \geq 1$, for any vector $\mathbf{a} \in (\mathbb{Z}/Q\mathbb{Z})^{N_{d,n}}$ we let
$f_{\mathbf{a}}$ denote the form of degree $d$ in $n+1$ variables which has coefficient vector $\mathbf{a}$. The set of vectors $\mathbf{a} \in \mathfrak{R}_{N_{d,n}}(Q)$ such that the form $f_{\mathbf{a}}$ has a non-trivial point modulo $p^{v_p(Q)}$ for any prime divisor $p$ of $Q$ will play a major role in the proof of
Proposition~\ref{Proposition non-Archimedean}. For $Q \geq 1$ we thus introduce the set
\begin{equation}
\label{Definition Floc}
\mathbb{F}_{d,n}^{\mathrm{loc}}(Q) = \left\{ \mathbf{a} \in \mathfrak{R}_{N_{d,n}}(Q) :
\forall p^r \| Q \ \exists \mathbf{x} \in \mathfrak{R}_{n+1}(p^r) \ f_{\mathbf{a}}(\mathbf{x}) \equiv 0 \bmod{p^r} \right\}.
\end{equation}
In addition, given a prime number $p$ and integers $r, N \geq 1$ we define the $p$-adic valuation $v_p(\mathbf{v})$ of a vector $\mathbf{v} \in (\mathbb{Z}/p^r\mathbb{Z})^N$ as the largest integer $e \in \{0, \dots, r\}$ such that we have
$\mathbf{v} \equiv \boldsymbol{0} \bmod{p^e}$. For $e \in \{0, \dots, r\}$, we also define the set
\begin{equation}
\label{Definition Re}
\mathfrak{R}_{N_{d,n}}^{(e)}(p^r) = \left\{ \mathbf{a} \in \mathfrak{R}_{N_{d,n}}(p^r) :
\exists \mathbf{x} \in \mathfrak{R}_{n+1}(p^r) \
\begin{array}{l l}
f_{\mathbf{a}}(\mathbf{x}) \equiv 0 \bmod{p^r} \\
v_p(\nabla f_{\mathbf{a}}(\mathbf{x}))=e
\end{array}
\right\}.
\end{equation}
It will be very important for our purpose to note that
\begin{equation}
\label{Partition non-Archimedean}
\mathbb{F}_{d,n}^{\mathrm{loc}}(p^r) = \bigcup_{e=0}^{r} \mathfrak{R}_{N_{d,n}}^{(e)}(p^r).
\end{equation}

The following two results provide us with the suitable tools to deal with small values of the quantity
$\sigma(\mathbf{a};p^r)$ for $\mathbf{a} \in \mathfrak{R}_{N_{d,n}}(p^r)$. We start by proving an upper bound for the cardinality of the set $\mathfrak{R}_{N_{d,n}}^{(e)}(p^r)$.

\begin{lemma}
\label{Lemma upper bound Re}
Let $d \geq 2$ and $n \geq 3$. Let also $p$ be a prime number and $r \geq 1$. For $e \in \{1, \dots, r\}$, we have
\begin{equation*}
\mathfrak{R}_{N_{d,n}}^{(e)}(p^r) \leq 2 p^{rN_{d,n}-e}.
\end{equation*}
\end{lemma}

\begin{proof}
For $\mathbf{a} \in \mathfrak{R}_{N_{d,n}}(p^r)$, we let
\begin{equation*}
\mathfrak{D}_{\mathbf{a}}(p^e) = \left\{ \mathbf{x} \in \mathfrak{R}_{n+1}(p^e) :
\begin{array}{l l}
f_{\mathbf{a}}(\mathbf{x}) \equiv 0 \bmod{p^e} \\
\nabla f_{\mathbf{a}}(\mathbf{x}) \equiv \boldsymbol{0} \bmod{p^e}
\end{array}
\right\},
\end{equation*}
and we observe that
\begin{equation*}
\# \mathfrak{R}_{N_{d,n}}^{(e)}(p^r) \leq
\# \left\{ \mathbf{a} \in \mathfrak{R}_{N_{d,n}}(p^r) : \mathfrak{D}_{\mathbf{a}}(p^e) \neq \emptyset \right\}.
\end{equation*}
Given $\mathbf{x} \in \mathfrak{D}_{\mathbf{a}}(p^e)$, we see that for any $g \in \mathbb{Z}/p^e\mathbb{Z}$ satisfying
$p \nmid g$ we have $g \mathbf{x} \in \mathfrak{D}_{\mathbf{a}}(p^e)$. Moreover these $p^{e-1}(p-1)$ vectors are distinct. In this way we see that
\begin{equation*}
\# \mathfrak{R}_{N_{d,n}}^{(e)}(p^r) \leq
\sum_{\mathbf{a} \in \mathfrak{R}_{N_{d,n}}(p^r)} \frac{\# \mathfrak{D}_{\mathbf{a}}(p^e)}{p^{e-1}(p-1)}.
\end{equation*}
We deduce that
\begin{equation}
\label{Upper bound after inversion}
\# \mathfrak{R}_{N_{d,n}}^{(e)}(p^r) \leq \frac{2}{p^e} \sum_{\mathbf{x} \in \mathfrak{R}_{n+1}(p^e)}
\# \left\{ \mathbf{a} \in \mathfrak{R}_{N_{d,n}}(p^r) :
\begin{array}{l l}
f_{\mathbf{a}}(\mathbf{x}) \equiv 0 \bmod{p^e} \\
\nabla f_{\mathbf{a}}(\mathbf{x}) \equiv \boldsymbol{0} \bmod{p^e}
\end{array}
\right\}.
\end{equation}
Given $\mathbf{x} \in \mathfrak{R}_{n+1}(p^e)$ we assume without loss of generality that $p \nmid x_0$. For
$\mathbf{a} \in \mathfrak{R}_{N_{d,n}}(p^r)$ and $i \in \{0, \dots, n \}$, we let
$c_{\mathbf{a}}^{(i)} \in \mathbb{Z}/p^r\mathbb{Z}$ be the coordinate of $\mathbf{a}$ corresponding to the monomial $x_0^{d-1} x_i$ and we set
\begin{equation}
\label{Definition ga}
g_{\mathbf{a}}(\mathbf{x}) = f_{\mathbf{a}}(\mathbf{x}) - x_0^{d-1}
\left( c_{\mathbf{a}}^{(0)} x_0 + c_{\mathbf{a}}^{(1)} x_1 + \dots + c_{\mathbf{a}}^{(n)} x_n \right).
\end{equation}
Therefore we have
\begin{equation}
\label{Equality nabla fa}
\nabla f_{\mathbf{a}}(\mathbf{x}) =
\begin{pmatrix}
x_0^{d-2} \left( d c_{\mathbf{a}}^{(0)} x_0 + (d-1) \left( c_{\mathbf{a}}^{(1)} x_1 + \dots + c_{\mathbf{a}}^{(n)} x_n \right) \right) \\
c_{\mathbf{a}}^{(1)} x_0^{d-1} \\
\vdots \\
c_{\mathbf{a}}^{(n)} x_0^{d-1}
\end{pmatrix}
+ \nabla g_{\mathbf{a}}(\mathbf{x}).
\end{equation}
Setting $\mathbf{c}_{\mathbf{a}} = \left( c_{\mathbf{a}}^{(0)}, \dots, c_{\mathbf{a}}^{(n)} \right)$ and estimating first the number of $c_{\mathbf{a}}^{(0)} \in \mathbb{Z}/p^r\mathbb{Z}$, we deduce from the assumption $p \nmid x_0$ that the cardinality of the set
\begin{equation*}
\left\{ \mathbf{c}_{\mathbf{a}} \in (\mathbb{Z}/p^r\mathbb{Z})^{n+1} :
\begin{array}{l l}
c_{\mathbf{a}}^{(0)} x_0^d \equiv - x_0^{d-1} \left( c_{\mathbf{a}}^{(1)} x_1 + \dots + c_{\mathbf{a}}^{(n)} x_n \right) -g_{\mathbf{a}}(\mathbf{x}) \bmod{p^e} \\
c_{\mathbf{a}}^{(i)} x_0^{d-1} \equiv - \dfrac{\partial g_{\mathbf{a}}}{\partial x_i}(\mathbf{x}) \bmod{p^e},
\ i \in \{1, \dots, n\}
\end{array}
\right\}
\end{equation*}
is equal to $p^{(r-e)(n+1)}$. We thus derive
\begin{equation*}
\# \left\{ \mathbf{a} \in \mathfrak{R}_{N_{d,n}}(p^r) :
\begin{array}{l l}
f_{\mathbf{a}}(\mathbf{x}) \equiv 0 \bmod{p^e} \\
\nabla f_{\mathbf{a}}(\mathbf{x}) \equiv \boldsymbol{0} \bmod{p^e}
\end{array}
\right\} \leq p^{r(N_{d,n}-n-1)+(r-e)(n+1)}.
\end{equation*}
Recalling the upper bound \eqref{Upper bound after inversion}, we see that an application of the trivial inequality
$\# \mathfrak{R}_{n+1}(p^e) \leq p^{e(n+1)}$ completes the proof.
\end{proof}

We now establish a lower bound for the quantity $\sigma(\mathbf{a};p^r)$ for
$\mathbf{a} \in \mathfrak{R}_{N_{d,n}}^{(e)}(p^r)$.

\begin{lemma}
\label{Lemma lower bound sigma}
Let $d \geq 2$ and $n \geq 3$. Let also $p$ be a prime number and $r \geq 1$. For $e \in \{0, \dots, r\}$ and
$\mathbf{a} \in \mathfrak{R}_{N_{d,n}}^{(e)}(p^r)$, we have
\begin{equation*}
\sigma(\mathbf{a};p^r) \geq \frac1{p^{(e+1) n}}.
\end{equation*}
\end{lemma}

\begin{proof}
Since by assumption $\mathbf{a} \in \mathfrak{R}_{N_{d,n}}^{(e)}(p^r)$ we may select
$\mathbf{x} \in \mathfrak{R}_{n+1}(p^r)$ satisfying the conditions $f_{\mathbf{a}}(\mathbf{x}) \equiv 0 \bmod{p^r}$ and
$v_p(\nabla f_{\mathbf{a}}(\mathbf{x}))=e$. In the case where $e=r$ the existence of $\mathbf{x}$ implies that
$\sigma(\mathbf{a};p^r) \geq 1/p^{rn}$, and the desired lower bound follows. We thus assume that
$e \in \{0, \dots, r-1 \}$ and we note that we trivially have
\begin{equation*}
\sigma(\mathbf{a};p^r) \geq \frac{\# \mathfrak{C}_{\mathbf{a},\mathbf{x}}(p^{e+1};p^r)}{p^{rn}},
\end{equation*}
where, for $c \in \{1, \dots, r \}$, we have set
\begin{equation*}
\mathfrak{C}_{\mathbf{a},\mathbf{x}}(p^c;p^r) = \left\{ \mathbf{b} \in \mathfrak{R}_{n+1}(p^r) :
\begin{array}{l l}
\mathbf{b} \equiv \mathbf{x} \bmod{p^c} \\
f_{\mathbf{a}}(\mathbf{b}) \equiv 0 \bmod{p^r}
\end{array}
\right\}.
\end{equation*}

We first handle the case where $e \in \{\lfloor r/2 \rfloor, \dots, r-1 \}$. For any
$\mathbf{b} \in \mathfrak{R}_{n+1}(p^r)$ such that $\mathbf{b} \equiv \mathbf{x} \bmod{p^{e+1}}$, we see that we have
\begin{equation*}
f_{\mathbf{a}}(\mathbf{b}) \equiv f_{\mathbf{a}}(\mathbf{x}) + \langle \nabla f_{\mathbf{a}}(\mathbf{x}), \mathbf{b}-\mathbf{x} \rangle \bmod{p^{2e+2}}.
\end{equation*}
Therefore, the conditions $f_{\mathbf{a}}(\mathbf{x}) \equiv 0 \bmod{p^r}$ and
$v_p(\nabla f_{\mathbf{a}}(\mathbf{x}))=e$ together with the facts that $\mathbf{b} \equiv \mathbf{x} \bmod{p^{e+1}}$ and $r \leq 2e+1$ imply that $f_{\mathbf{a}}(\mathbf{b}) \equiv 0 \bmod{p^{r}}$. It follows that
\begin{equation*}
\mathfrak{C}_{\mathbf{a},\mathbf{x}}(p^{e+1};p^r) = \left\{ \mathbf{b} \in \mathfrak{R}_{n+1}(p^r) :
\mathbf{b} \equiv \mathbf{x} \bmod{p^{e+1}} \right\}.
\end{equation*}
Hence we have $\# \mathfrak{C}_{\mathbf{a},\mathbf{x}}(p^{e+1};p^r) = p^{(r-e-1)(n+1)}$ and thus
\begin{equation*}
\sigma(\mathbf{a};p^r) \geq \frac{p^{r-e-1}}{p^{(e+1)n}},
\end{equation*}
which is satisfactory since $r-e-1 \geq 0$.

Finally, in the case where $e \in \{0, \dots, \lfloor r/2 \rfloor-1 \}$ we use Hensel's lemma (in the form of
\cite[Lemma~$3.3$]{MR3742196} with $\# \mathfrak{C}_{\mathbf{a},\mathbf{x}}(p^{e+1};p^s) = \# R_e(p^s,0;p^{e+1})$ for example) to deduce that for any $s \geq 2e+2$, we have
\begin{equation*}
\frac{\# \mathfrak{C}_{\mathbf{a},\mathbf{x}}(p^{e+1};p^s)}{p^{sn}} =
\frac{\# \mathfrak{C}_{\mathbf{a},\mathbf{x}}(p^{e+1};p^{s-1})}{p^{(s-1)n}}.
\end{equation*}
Applying this equality $r-(2e+1)$ times, we derive
\begin{equation*}
\sigma(\mathbf{a};p^r) \geq \frac{\# \mathfrak{C}_{\mathbf{a},\mathbf{x}}(p^{e+1};p^{2e+1})}{p^{(2e+1)n}},
\end{equation*}
which completes the proof since we have $\# \mathfrak{C}_{\mathbf{a},\mathbf{x}}(p^{e+1};p^{2e+1})=p^{e(n+1)}$ by the previous case.
\end{proof}

We now have the tools at hand to establish Proposition~\ref{Proposition non-Archimedean}.

\begin{proof}[Proof of Proposition~\ref{Proposition non-Archimedean}]
Recall the respective definitions \eqref{Definition singular series}, \eqref{Definition W} and \eqref{Definition w} of the non-Archimedean factor $\mathfrak{S}_V(B)$, the integer $W$ and the quantity $w$, where we take $B = A \phi(A)$. It is convenient to set
\begin{equation*}
\mathscr{F}_{\phi}(A) = \frac1{\# \mathbb{V}_{d,n}^{\mathrm{loc}}(A)} \cdot
\# \left\{V \in \mathbb{V}_{d,n}^{\mathrm{loc}}(A) : \mathfrak{S}_V(A\phi(A)) < \frac{C}{\phi(A)^{1/6}} \right\}.
\end{equation*}
Recall also the definition \eqref{Definition Floc} of the set $\mathbb{F}_{d,n}^{\mathrm{loc}}(Q)$ for given $Q \geq 1$. Breaking the summation over $\mathbf{a} \in \mathfrak{R}_{N_{d,n}}(W)$ into residue classes modulo $W$, we see that
\begin{equation*}
\mathscr{F}_{\phi}(A) = \frac1{\# \mathbb{V}_{d,n}^{\mathrm{loc}}(A)}
\sum_{\substack{\mathbf{a} \in \mathbb{F}_{d,n}^{\mathrm{loc}}(W) \\ \sigma(\mathbf{a};W) < C/\phi(A)^{1/6}}}
\# \left\{ V \in \mathbb{V}_{d,n}^{\mathrm{loc}}(A) : \mathbf{a}_V \equiv \mathbf{a} \bmod{W} \right\}.
\end{equation*}
The upper bound \eqref{Upper bound W} and the assumption $\phi(A) \leq A$ ensure that $W \ll A$. We thus have
\begin{equation*}
\# \left\{ V \in \mathbb{V}_{d,n}^{\mathrm{loc}}(A) : \mathbf{a}_V\equiv \mathbf{a}\bmod{W} \right\} \ll
\left( \frac{A}{W} \right)^{N_{d,n}}.
\end{equation*}
Since the lower bound \eqref{Lower bound loc} implies in particular that
$\#\mathbb{V}_{d,n}^{\mathrm{loc}}(A) \gg A^{N_{d,n}}$, we deduce
\begin{equation*}
\mathscr{F}_{\phi}(A) \ll \frac1{W^{N_{d,n}}} \cdot
\# \left\{ \mathbf{a} \in \mathbb{F}_{d,n}^{\mathrm{loc}}(W) : \sigma(\mathbf{a};W) < \frac{C}{\phi(A)^{1/6}} \right\}.
\end{equation*}

In addition, we remark that it follows from the equality \eqref{Equality multiplicativity} that $\sigma(\mathbf{a};W) > 0$ whenever $\mathbf{a} \in \mathbb{F}_{d,n}^{\mathrm{loc}}(W)$. We now let $\kappa \in (0,1/n)$ and we use the standard trick
\begin{equation*}
\#\left\{ \mathbf{a} \in \mathbb{F}_{d,n}^{\mathrm{loc}}(W) : \sigma(\mathbf{a};W) < \frac{C}{\phi(A)^{1/6}} \right\} \leq
\sum_{\mathbf{a} \in \mathbb{F}_{d,n}^{\mathrm{loc}}(W)} \left(\frac{C}{\phi(A)^{1/6} \sigma(\mathbf{a};W)}\right)^{\kappa}.
\end{equation*}
Therefore, we deduce from the equality \eqref{Equality multiplicativity} that
\begin{equation*}
\mathscr{F}_{\phi}(A) \ll \frac1{W^{N_{d,n}} \phi(A)^{\kappa/6}}
\sum_{\mathbf{a} \in \mathbb{F}_{d,n}^{\mathrm{loc}}(W)} \ \prod_{p^r \| W} \frac1{\sigma(\mathbf{a};p^r)^{\kappa}}.
\end{equation*}
As a result, an application of the Chinese remainder theorem yields
\begin{equation*}
\mathscr{F}_{\phi}(A) \ll \frac1{W^{N_{d,n}} \phi(A)^{\kappa/6}} \ \prod_{p^r \| W} \ 
\sum_{\mathbf{a} \in \mathbb{F}_{d,n}^{\mathrm{loc}}(p^r)} \frac1{\sigma(\mathbf{a};p^r)^{\kappa}}.
\end{equation*}
In order to estimate the sum over $\mathbf{a} \in \mathbb{F}_{d,n}^{\mathrm{loc}}(p^r)$ we need to argue differently depending on whether or not $\sigma(\mathbf{a};p^r)$ is particularly small. We thus let
\begin{equation*}
\Sigma_{>}^{(\kappa)}(p^r) =
\sum_{\substack{\mathbf{a} \in \mathbb{F}_{d,n}^{\mathrm{loc}}(p^r) \\ \sigma(\mathbf{a};p^r) > 1/2}} \frac1{\sigma(\mathbf{a};p^r)^{\kappa}},
\end{equation*}
and
\begin{equation*}
\Sigma_{\leq}^{(\kappa)}(p^r) =
\sum_{\substack{\mathbf{a} \in \mathbb{F}_{d,n}^{\mathrm{loc}}(p^r) \\ \sigma(\mathbf{a};p^r) \leq 1/2}} \frac1{\sigma(\mathbf{a};p^r)^{\kappa}},
\end{equation*}
so that
\begin{equation}
\label{Upper bound separation cases}
\mathscr{F}_{\phi}(A) \ll \frac1{W^{N_{d,n}} \phi(A)^{\kappa/6}} \ \prod_{p^r \| W}
\left( \Sigma_{>}^{(\kappa)}(p^r) + \Sigma_{\leq}^{(\kappa)}(p^r) \right).
\end{equation}

We start by handling the sum $\Sigma_{>}^{(\kappa)}(p^r)$. In order to do so, we employ the estimate 
\begin{equation*}
\frac1{\sigma(\mathbf{a};p^r)^{\kappa}} =
1 - \kappa \left( \sigma(\mathbf{a};p^r) - 1 \right) + O \left( \left( \sigma(\mathbf{a};p^r) -1 \right)^2 \right),
\end{equation*}
where the implied constant depends at most on $\kappa$. We obtain
\begin{align*}
\Sigma_{>}^{(\kappa)}(p^r) & =
\sum_{\substack{\mathbf{a} \in \mathbb{F}_{d,n}^{\mathrm{loc}}(p^r) \\ \sigma(\mathbf{a};p^r) > 1/2}} 1 -
\kappa \sum_{\substack{\mathbf{a} \in \mathbb{F}_{d,n}^{\mathrm{loc}}(p^r) \\ \sigma(\mathbf{a};p^r) > 1/2}}
 \left( \sigma(\mathbf{a};p^r) - 1 \right) +
O \left( \sum_{\mathbf{a} \in \mathfrak{R}_{N_{d,n}}(p^r)} \left( \sigma(\mathbf{a};p^r) - 1 \right)^2 \right) \\
& \leq \sum_{\mathbf{a} \in \mathfrak{R}_{N_{d,n}}(p^r)} 1 -
\kappa \sum_{\mathbf{a} \in \mathfrak{R}_{N_{d,n}}(p^r)} \left( \sigma(\mathbf{a};p^r) - 1 \right) +
O \left( \sum_{\mathbf{a} \in \mathfrak{R}_{N_{d,n}}(p^r)} \left( \sigma(\mathbf{a};p^r) - 1 \right)^2 \right).
\end{align*}
On appealing to Lemmas \ref{Lemma first moment sigma} and \ref{Lemma second moment sigma} and to the equality \eqref{Equality cardinality B}, we therefore conclude that 
\begin{equation}
\label{Upper bound Sigma geq}
\Sigma_{>}^{(\kappa)}(p^r) \leq p^{r N_{d,n}} \left( 1 + O \left( \frac1{p^{n-1}} \right) \right).
\end{equation}

We now consider the sum $\Sigma_{\leq}^{(\kappa)}(p^r)$. Recall the definition \eqref{Definition Re} of the set
$\mathfrak{R}_{N_{d,n}}^{(e)}(p^r)$ for given $e \in \{0, \dots, r \}$. It follows from the equality
\eqref{Partition non-Archimedean} that
\begin{equation}
\label{Definition Skappa}
\Sigma_{\leq}^{(\kappa)}(p^r) \leq \sum_{e=0}^r S^{(\kappa)}(e;p^r),
\end{equation}
where, for $e \in \{0, \dots, r\}$, we have set
\begin{equation*}
S^{(\kappa)}(e;p^r) = \sum_{\substack{\mathbf{a} \in \mathfrak{R}_{N_{d,n}}^{(e)}(p^r) \\ \sigma(\mathbf{a};p^r) \leq 1/2}} \frac1{\sigma(\mathbf{a};p^r)^{\kappa}}.
\end{equation*}

We first handle the case where $e \in \{0, 1 \}$. Applying Lemma~\ref{Lemma lower bound sigma}, we see that
\begin{equation*}
S^{(\kappa)}(e;p^r) \leq p^{\kappa(e+1)n}
\sum_{\substack{\mathbf{a} \in \mathfrak{R}_{N_{d,n}}(p^r) \\ \sigma(\mathbf{a};p^r) \leq 1/2}} 1.
\end{equation*}
Using the fact that $1 \leq 4 (\sigma(\mathbf{a};p^r) - 1)^2$ whenever $\sigma(\mathbf{a};p^r) \leq 1/2$, we get
\begin{equation*}
S^{(\kappa)}(e;p^r) \leq 4 p^{\kappa(e+1)n}
\sum_{\mathbf{a} \in \mathfrak{R}_{N_{d,n}}(p^r)} \left(\sigma(\mathbf{a};p^r) - 1 \right)^2.
\end{equation*}
Therefore, Lemma~\ref{Lemma second moment sigma} gives
\begin{equation*}
S^{(\kappa)}(e;p^r) \ll p^{rN_{d,n}-n+1+\kappa (e+1) n},
\end{equation*}
from which it eventually follows that
\begin{equation}
\label{Upper bound Skappa 1}
S^{(\kappa)}(0;p^r) + S^{(\kappa)}(1;p^r) \ll p^{rN_{d,n} -n+1 +2 \kappa n}.
\end{equation}

We now treat the case where $e \in \{2, \dots, r\}$. Dropping the condition $\sigma(\mathbf{a};p^r) \leq 1/2$ and applying Lemma~\ref{Lemma lower bound sigma} we obtain
\begin{equation*}
S^{(\kappa)}(e;p^r) \leq p^{\kappa (e+1) n} \cdot \# \mathfrak{R}_{N_{d,n}}^{(e)}(p^r). 
\end{equation*}
Since $\kappa < 1/n$, we deduce from Lemma~\ref{Lemma upper bound Re} that
\begin{equation}
\label{Upper bound Skappa 2}
\sum_{e=2}^r S^{(\kappa)}(e;p^r) \ll p^{rN_{d,n}-2+3\kappa n}.
\end{equation}

Combining the inequality \eqref{Definition Skappa} and the upper bounds \eqref{Upper bound Skappa 1} and \eqref{Upper bound Skappa 2}, we see that
\begin{equation*}
\Sigma_{\leq}^{(\kappa)}(p^r) \ll p^{rN_{d,n}} \left( \frac1{p^{n-1-2 \kappa n}} + \frac1{p^{2-3\kappa n}} \right).
\end{equation*}
Moreover, we have $n \geq 3$ by assumption so the choice $\kappa = 1/4n$ yields
\begin{equation}
\label{Upper bound Sigma leq}
\Sigma_{\leq}^{(\kappa)}(p^r) \ll p^{rN_{d,n} - 5/4}.
\end{equation}

Putting together the upper bounds \eqref{Upper bound separation cases}, \eqref{Upper bound Sigma geq} and
\eqref{Upper bound Sigma leq}, we see that we have proved that
\begin{equation*}
\mathscr{F}_{\phi}(A) \ll \frac1{\phi(A)^{1/24n}} \ \prod_{p^r \| W} \left( 1 + O \left( \frac1{p^{5/4}} \right) \right).
\end{equation*}
The product in the right-hand side is convergent so we see that this completes the proof of
Proposition~\ref{Proposition non-Archimedean}.
\end{proof}

\subsection{The Archimedean factor is rarely small}

\label{Section Archimedean}

Our goal in this section is to prove Proposition~\ref{Proposition Archimedean}. We shall follow the traces of our argument in the non-Archimedean setting. To begin with, we recall that the purpose of
Lemmas~\ref{Lemma first moment sigma} and \ref{Lemma second moment sigma} was to allow us to handle several non-Archimedean places simultaneously, which is of course irrelevant here so we will not need analogues of these results. However, we will establish direct analogues of Lemmas~\ref{Lemma upper bound Re} and
\ref{Lemma lower bound sigma}, and we will also need some preparatory work in order to apply these results.

For any vector $\mathbf{a} \in \mathbb{R}^{N_{d,n}}$ we let $f_{\mathbf{a}}$ denote the form of degree $d$ in $n+1$ variables which has coefficient vector $\mathbf{a}$. The set of non-zero vectors $\mathbf{a} \in \mathbb{R}^{N_{d,n}}$ such that the form $f_{\mathbf{a}}$ has a non-trivial real point will be of primary importance in the proof of
Proposition~\ref{Proposition Archimedean}. We thus define the set
\begin{equation}
\label{Definition Iloc}
\mathbb{I}_{d,n}^{\mathrm{loc}} =
\left\{ \mathbf{a} \in \mathbb{R}^{N_{d,n}} \smallsetminus \{ \boldsymbol{0} \} :
\exists \mathbf{x} \in \mathbb{S}^n \ f_{\mathbf{a}}(\mathbf{x}) = 0 \right\},
\end{equation}
where, for given $N \geq 1$, we have introduced the $N$-dimensional hypersphere
\begin{equation*}
\mathbb{S}^N = \left\{ \mathbf{x} \in \mathbb{R}^{N+1} : ||\mathbf{x}|| = 1 \right\}.
\end{equation*}

Our first task will be to establish an upper bound for the number of integral vectors
$\mathbf{a} \in \mathbb{I}_{d,n}^{\mathrm{loc}}$ having norm at most $A$ and lying close to the boundary of the region
$\mathbb{I}_{d,n}^{\mathrm{loc}}$. In order to do so, for $\mathbf{a} \in \mathbb{Z}^{N_{d,n}}$ we define the neighbourhood
\begin{equation*}
\mathcal{N}(\mathbf{a}) = \left\{ \mathbf{y} \in \mathbb{R}^{N_{d,n}} :
\mathbf{y} - \mathbf{a} \in \mathcal{B}_{N_{d,n}}(1) \right\},
\end{equation*}
and we set
\begin{equation*}
\mathscr{U}_{d,n}(A) =
\left\{ \mathbf{a} \in \mathbb{Z}^{N_{d,n}} \cap \mathcal{B}_{N_{d,n}}(A) \cap \mathbb{I}_{d,n}^{\mathrm{loc}} : \mathcal{N}(\mathbf{a}) \not \subset \mathbb{I}_{d,n}^{\mathrm{loc}} \right\}.
\end{equation*}
Heuristically, given an integer $N \geq 1$ and a real hypersurface embedded in $\mathbb{R}^N$ it is natural to expect that the number of integral vectors of norm at most $A$ and whose distance to the hypersurface is at most $1$ should have order of magnitude $A^{N-1}$. The following result shows that these elementary heuristics apply in our setting.

\begin{lemma}
\label{Lemma upper bound boundary}
Let $d \geq 2$ and $n \geq 3$. We have
\begin{equation*}
\# \mathscr{U}_{d,n}(A) \ll A^{N_{d,n}-1}.
\end{equation*}
\end{lemma}

\begin{proof}
Given $\mathbf{a} \in \mathscr{U}_{d,n}(A)$ we let
$\mathbf{b} \in \mathcal{N}(\mathbf{a}) \smallsetminus \mathbb{I}_{d,n}^{\mathrm{loc}}$ and we define
\begin{equation*}
M_{\mathbf{a}} =
\max \left\{ t \in (0,1] : \mathbf{a} + t (\mathbf{b}- \mathbf{a}) \in \mathbb{I}_{d,n}^{\mathrm{loc}} \right\}.
\end{equation*}
We also set $\mathbf{c} = \mathbf{a} + M_{\mathbf{a}} (\mathbf{b}- \mathbf{a})$ and we check that for any
$\mathbf{x} \in \mathbb{S}^n$ satisfying $f_{\mathbf{c}}(\mathbf{x}) = 0$ we have
$\nabla f_{\mathbf{c}}(\mathbf{x}) = \boldsymbol{0}$. Indeed, for $\rho \in (0,1/A^2)$ and
$\mathbf{y} \in \mathcal{B}_{n+1}(\rho)$ we have
\begin{align*}
f_{\mathbf{c} + \rho^2 (\mathbf{b}-\mathbf{a})}(\mathbf{x} + \mathbf{y}) & =
f_{\mathbf{c} + \rho^2 (\mathbf{b}-\mathbf{a})}(\mathbf{x}) +
\langle \nabla f_{\mathbf{c} + \rho^2 (\mathbf{b}-\mathbf{a})}(\mathbf{x}), \mathbf{y} \rangle +
O(A \rho^2) \\
& = f_{\mathbf{c}}(\mathbf{x}) + \langle \nabla f_{\mathbf{c}}(\mathbf{x}), \mathbf{y} \rangle + O(\rho^{3/2}) \\
& = \langle \nabla f_{\mathbf{c}}(\mathbf{x}), \mathbf{y} \rangle + O(\rho^{3/2}).
\end{align*}
Let us assume that $\nabla f_{\mathbf{c}}(\mathbf{x}) \neq \boldsymbol{0}$ and let $\mathbf{y}_0 \in \mathbb{S}^n$ satisfying $\langle \nabla f_{\mathbf{c}}(\mathbf{x}), \mathbf{y} \rangle \neq 0$. For $|u| \leq \rho$ we thus have
\begin{equation*}
f_{\mathbf{c} + \rho^2 (\mathbf{b}-\mathbf{a})}(\mathbf{x} + u \mathbf{y}_0) =
u \cdot \langle \nabla f_{\mathbf{c}}(\mathbf{x}), \mathbf{y}_0 \rangle + O(\rho^{3/2}).
\end{equation*}
We now see that if $\rho$ is chosen sufficiently small then the intermediate value theorem shows that there exists
$u_0 \in \mathbb{R}$ such that $f_{\mathbf{c} + \rho^2 (\mathbf{b}-\mathbf{a})}(\mathbf{x} + u_0 \mathbf{y}_0) = 0$, which contradicts the maximality of $M_{\mathbf{a}}$. We have thus proved that
\begin{equation*}
\# \mathscr{U}_{d,n}(A) \ll
\# \left\{ \mathbf{a} \in \mathbb{Z}^{N_{d,n}} \cap \mathcal{B}_{N_{d,n}}(A) :
\exists \mathbf{c} \in \mathcal{N}(\mathbf{a}) \ \exists \mathbf{x} \in \mathbb{S}^n \ \nabla f_{\mathbf{c}}(\mathbf{x})=\boldsymbol{0} \right\}.
\end{equation*}

Next, we note that given $\mathbf{a} \in \mathcal{B}_{N_{d,n}}(A)$, if $\mathbf{c} \in \mathcal{N}(\mathbf{a})$ and
$\mathbf{x} \in \mathbb{S}^n$ satisfy $\nabla f_{\mathbf{c}}(\mathbf{x})=\boldsymbol{0}$ then for any
$\mathbf{y} \in \mathbb{R}^{n+1}$ such that $||\mathbf{y}-\mathbf{x}|| \leq 1/A$ we have
$||\nabla f_{\mathbf{a}}(\mathbf{y})|| \ll 1$. Indeed, the triangle inequality gives
\begin{align*}
||\nabla f_{\mathbf{a}}(\mathbf{y})|| & \leq ||\nabla f_{\mathbf{a}-\mathbf{c}}(\mathbf{y})|| +
||\nabla f_{\mathbf{c}}(\mathbf{y}) - \nabla f_{\mathbf{c}}(\mathbf{x})|| + ||\nabla f_{\mathbf{c}}(\mathbf{x})||\\
& \ll ||\mathbf{a}-\mathbf{c}|| \cdot ||\mathbf{y}||^{d-1} +
||\mathbf{c}|| \cdot ||\mathbf{y}-\mathbf{x}|| \cdot \max \left\{ ||\mathbf{x}||, ||\mathbf{y}|| \right\}^{d-2}.
\end{align*}
We thus get $||\nabla f_{\mathbf{a}}(\mathbf{y})|| \ll 1$ as wished. Since we have in addition
\begin{equation*} 
\vol \left( \left\{ \mathbf{y} \in \mathbb{R}^{n+1} : ||\mathbf{y} - \mathbf{x}|| \leq \frac1{A} \right\} \right) \gg
\frac1{A^{n+1}},
\end{equation*}
it follows that
\begin{align}
\nonumber
\# \mathscr{U}_{d,n}(A) & \ll A^{n+1} \sum_{\mathbf{a} \in \mathbb{Z}^{N_{d,n}} \cap \mathcal{B}_{N_{d,n}}(A)}
\vol \left( \left\{ \mathbf{y} \in \mathbb{R}^{n+1} :
\begin{array}{l l}
1-1/A \leq ||\mathbf{y}|| \leq 1+ 1/A \\
||\nabla f_{\mathbf{a}}(\mathbf{y})|| \ll 1
\end{array}
\right\} \right) \\
\label{Upper bound shell}
& \ll A^{n+1} \int_{\mathcal{H}_{n+1}(A)} \# \left\{ \mathbf{a} \in \mathbb{Z}^{N_{d,n}} \cap \mathcal{B}_{N_{d,n}}(A) :
||\nabla f_{\mathbf{a}}(\mathbf{y})|| \ll 1 \right\} \mathrm{d} \mathbf{y},
\end{align}
where we have introduced the hyperspherical shell
\begin{equation*}
\mathcal{H}_{n+1}(A) = \mathcal{B}_{n+1} \left(1+\frac1{A} \right) \smallsetminus \mathcal{B}_{n+1} \left(1-\frac1{A} \right).
\end{equation*}
Given $\mathbf{y} \in \mathcal{H}_{n+1}(A)$ we may clearly assume without loss of generality that $|y_0| \geq 1/n$. For
$\mathbf{a} \in \mathbb{Z}^{N_{d,n}} \cap \mathcal{B}_{N_{d,n}}(A)$ and $i \in \{0, \dots, n \}$, we let
$c_{\mathbf{a}}^{(i)} \in \mathbb{Z}$ be the coordinate of $\mathbf{a}$ corresponding to the monomial $x_0^{d-1} x_i$. Recall the definition \eqref{Definition ga} of $g_{\mathbf{a}}$ and the equality \eqref{Equality nabla fa}. Setting
$\mathbf{c}_{\mathbf{a}} = \left( c_{\mathbf{a}}^{(0)}, \dots, c_{\mathbf{a}}^{(n)} \right)$ and estimating first the number of $c_{\mathbf{a}}^{(0)} \in \mathbb{Z}$, we deduce from the assumption $|y_0| \geq 1/n$ that
\begin{equation*}
\# \left\{ \mathbf{c}_{\mathbf{a}} \in \mathbb{Z}^{n+1} :
\begin{array}{l l}
y_0^{d-2} \left( d c_{\mathbf{a}}^{(0)} y_0 + (d-1) \left( c_{\mathbf{a}}^{(1)} y_1 + \dots + c_{\mathbf{a}}^{(n)} y_n \right) \right) + \dfrac{\partial g_{\mathbf{a}}}{\partial x_0}(\mathbf{y}) \ll 1 \\
c_{\mathbf{a}}^{(i)} y_0^{d-1} + \dfrac{\partial g_{\mathbf{a}}}{\partial x_i}(\mathbf{y}) \ll 1,
\ i \in \{1, \dots, n\}
\end{array} \!
\right\} \ll 1.
\end{equation*}
This yields 
\begin{equation*}
\# \left\{ \mathbf{a} \in \mathbb{Z}^{N_{d,n}} \cap \mathcal{B}_{N_{d,n}}(A) :
||\nabla f_{\mathbf{a}}(\mathbf{y})|| \ll 1 \right\} \ll A^{N_{d,n}-n-1}.
\end{equation*}
Recalling the upper bound \eqref{Upper bound shell} and noting that we clearly have
\begin{equation*}
\vol \left( \mathcal{H}_{n+1}(A) \right) \ll \frac1{A},
\end{equation*}
we see that this finishes the proof.
\end{proof}

For $\lambda > 0$, we introduce the set
\begin{equation}
\label{Definition Blambda}
\mathcal{B}_{N_{d,n}}^{(\lambda)} = \left\{ \mathbf{a} \in \mathcal{B}_{N_{d,n}}(1) :
\exists \mathbf{x} \in \mathbb{S}^n \
\begin{array}{l l}
f_{\mathbf{a}}(\mathbf{x}) = 0 \\
\lambda ||\mathbf{a}|| < ||\nabla f_{\mathbf{a}} (\mathbf{x})|| \leq 2 \lambda ||\mathbf{a}||
\end{array}
\right\}.
\end{equation}
It will be crucial for our purpose to note that
\begin{equation}
\label{Partition Archimedean}
\mathcal{B}_{N_{d,n}}(1) \cap \mathbb{I}_{d,n}^{\mathrm{loc}} =
\bigcup_{\ell=1}^{\infty} \mathcal{B}_{N_{d,n}}^{(M_{d,n}/2^{\ell})},
\end{equation}
where we have set
\begin{equation}
\label{Definition Mdn}
M_{d,n} =
\max \left\{ ||\nabla f_{\mathbf{a}} (\mathbf{x})|| : (\mathbf{a}, \mathbf{x}) \in \mathbb{S}^{N_{d,n}-1} \times \mathbb{S}^n \right\}.
\end{equation}

The following result is the Archimedean analogue of Lemma~\ref{Lemma upper bound Re} and gives an upper bound for the volume of the set $\mathcal{B}_{N_{d,n}}^{(\lambda)}$.

\begin{lemma}
\label{Lemma upper bound Blambda}
Let $d \geq 2$ and $n \geq 3$. For $\lambda \in (0, M_{d,n})$, we have
\begin{equation*}
\vol \left( \mathcal{B}_{N_{d,n}}^{(\lambda)} \right) \ll \lambda^2.
\end{equation*}
\end{lemma}

\begin{proof}
For $\mathbf{a} \in \mathcal{B}_{N_{d,n}}(1)$ we let
\begin{equation*}
\mathcal{D}_{\mathbf{a}}(\lambda) = \left\{ \mathbf{x} \in \mathbb{S}^n :
\begin{array}{l l}
|f_{\mathbf{a}}(\mathbf{x})| \leq \lambda^2 \\
||\nabla f_{\mathbf{a}} (\mathbf{x})|| \leq 2 \lambda
\end{array}
\right\},
\end{equation*}
and we observe that
\begin{equation*}
\vol \left( \mathcal{B}_{N_{d,n}}^{(\lambda)} \right) \leq
\vol \left( \left\{ \mathbf{a} \in \mathcal{B}_{N_{d,n}}(1) :
\mathcal{D}_{\mathbf{a}}(\lambda) \neq \emptyset \right\} \right).
\end{equation*}
Given $\mathbf{x} \in \mathcal{D}_{\mathbf{a}}(\lambda/2)$, it follows from the estimates
\begin{equation*}
f_{\mathbf{a}}(\mathbf{y}) = f_{\mathbf{a}}(\mathbf{x}) + \langle \nabla f_{\mathbf{a}}(\mathbf{x}), \mathbf{y} - \mathbf{x} \rangle + O \left( ||\mathbf{y} - \mathbf{x}||^2 \right),
\end{equation*}
and
\begin{equation*}
||\nabla f_{\mathbf{a}}(\mathbf{y})|| = ||\nabla f_{\mathbf{a}}(\mathbf{x})|| + O \left( ||\mathbf{y} - \mathbf{x}|| \right),
\end{equation*}
that there exists an absolute constant $K>0$ such that if $\mathbf{y} \in \mathbb{S}^n$ and
$||\mathbf{x}-\mathbf{y}|| \leq K \lambda$ then $\mathbf{y} \in \mathcal{D}_{\mathbf{a}}(\lambda)$. Since we have
\begin{equation*}
\vol \left( \left\{ \mathbf{y} \in \mathbb{S}^n : ||\mathbf{x}-\mathbf{y}|| \leq K \lambda \right\} \right) \gg \lambda^n,
\end{equation*}
we deduce that
\begin{equation*} 
\vol \left( \mathcal{B}_{N_{d,n}}^{(\lambda)} \right) \ll \int_{\mathcal{B}_{N_{d,n}}(1)}
\frac{\vol \left( \mathcal{D}_{\mathbf{a}}(\lambda) \right)}{\lambda^n} \mathrm{d} \mathbf{a}.
\end{equation*}
Therefore we have
\begin{equation}
\label{Upper bound Blambda}
\vol \left( \mathcal{B}_{N_{d,n}}^{(\lambda)} \right) \ll \frac1{\lambda^n} \int_{\mathbb{S}^n}
\vol \left( \left\{ \mathbf{a} \in \mathcal{B}_{N_{d,n}}(1) :
\begin{array}{l l}
|f_{\mathbf{a}}(\mathbf{x})| \leq \lambda^2 \\
||\nabla f_{\mathbf{a}} (\mathbf{x})|| \leq 2 \lambda
\end{array}
\right\} \right) \mathrm{d} \mathbf{x}.
\end{equation}
Given $\mathbf{x} \in \mathbb{S}^n$ we may clearly assume without loss of generality that $|x_0| \geq 1/n$. For
$\mathbf{a} \in \mathcal{B}_{N_{d,n}}(1)$ and $i \in \{0, \dots, n \}$, we let $c_{\mathbf{a}}^{(i)} \in \mathbb{R}$ be the coordinate of $\mathbf{a}$ corresponding to the monomial $x_0^{d-1} x_i$. Recall the definition \eqref{Definition ga} of $g_{\mathbf{a}}$ and the equality \eqref{Equality nabla fa}. Setting
$\mathbf{c}_{\mathbf{a}} = \left( c_{\mathbf{a}}^{(0)}, \dots, c_{\mathbf{a}}^{(n)} \right)$ and estimating first the number of $c_{\mathbf{a}}^{(0)} \in \mathbb{R}$, we deduce from the assumption $|x_0| \geq 1/n$ that
\begin{equation*}
\vol \left( \left\{ \mathbf{c}_{\mathbf{a}} \in \mathbb{R}^{n+1} :
\begin{array}{l l}
\left| c_{\mathbf{a}}^{(0)} x_0^d + x_0^{d-1} \left( c_{\mathbf{a}}^{(1)} x_1 + \dots + c_{\mathbf{a}}^{(n)} x_n \right) + g_{\mathbf{a}}(\mathbf{x}) \right| \leq \lambda^2 \\
\left| c_{\mathbf{a}}^{(i)} x_0^{d-1} + \dfrac{\partial g_{\mathbf{a}}}{\partial x_i}(\mathbf{x}) \right| \leq \lambda,
\ i \in \{1, \dots, n\}
\end{array}
\right\} \right) \ll \lambda^{n+2}.
\end{equation*}
It follows that 
\begin{equation*}
\vol \left( \left\{ \mathbf{a} \in \mathcal{B}_{N_{d,n}}(1) :
\begin{array}{l l}
|f_{\mathbf{a}}(\mathbf{x})| \leq \lambda^2 \\
||\nabla f_{\mathbf{a}} (\mathbf{x})|| \leq 2 \lambda
\end{array}
\right\} \right) \ll \lambda^{n+2}.
\end{equation*}
Recalling the equality \eqref{Upper bound Blambda} we see that this completes the proof.
\end{proof}

Given $N \geq 1$, recall the definition \eqref{Definition tau} of the quantity $\tau(\mathbf{a};\gamma)$ for
$\mathbf{a} \in \mathbb{R}^{N_{d,n}}$ and $\gamma > 0$. The following result is the Archimedean analogue of
Lemma~\ref{Lemma lower bound sigma} and provides us with a lower bound for the quantity $\tau(\mathbf{a};\gamma)$ for $\mathbf{a} \in \mathcal{B}_{N_{d,n}}^{(\lambda)}$.

\begin{lemma}
\label{Lemma lower bound tau}
Let $d \geq 2$ and $n \geq 3$. Let also $\gamma > 0$. For $\lambda \in (0, M_{d,n})$ and
$\mathbf{a} \in \mathcal{B}_{N_{d,n}}^{(\lambda)}$, we have
\begin{equation*}
\tau(\mathbf{a}; \gamma) \gg \lambda^{n+1} \cdot \min \left\{ \gamma, \frac1{\lambda^2} \right\}.
\end{equation*}
\end{lemma}

\begin{proof}
We may assume without loss of generality that $||\mathbf{a}|| = 1$. In addition, since by assumption
$\mathbf{a} \in \mathcal{B}_{N_{d,n}}^{(\lambda)}$ we may select $\mathbf{x} \in \mathbb{R}^{n+1}$ such that
$||\mathbf{x}|| = 1/2$ and satisfying the conditions $f_{\mathbf{a}}(\mathbf{x}) = 0$ and
\begin{equation}
\label{Inequalities nabla}
\frac{\lambda}{2^d} < ||\nabla f_{\mathbf{a}}(\mathbf{x})|| \leq \frac{\lambda}{2^{d-1}}.
\end{equation}
We have
\begin{equation*}
\tau(\mathbf{a};\gamma) = \gamma \cdot \vol \left( \left\{ \mathbf{u} \in \mathcal{B}_{n+1}(1) :
|f_{\mathbf{a}}(\mathbf{u})| \leq \frac{||\nu_{d,n}(\mathbf{u})||}{2\gamma} \right\} \right).
\end{equation*}
Since $||\nu_{d,n}(\mathbf{u})|| \geq ||\mathbf{u}||^d$, we see that
\begin{equation*}
\tau(\mathbf{a};\gamma) \geq \gamma \cdot \vol \left( \left\{ \mathbf{u} \in \mathcal{B}_{n+1}(1) \smallsetminus \mathcal{B}_{n+1} \left(\frac1{4} \right) :
|f_{\mathbf{a}}(\mathbf{u})| \leq \frac1{2^{2d+1} \gamma} \right\} \right).
\end{equation*}
We note that if $||\mathbf{u} - \mathbf{x}|| \leq \lambda/4M_{d,n}$ then $1/4 < ||\mathbf{u}|| < 3/4$ and thus
\begin{equation*}
\tau(\mathbf{a};\gamma) \geq \gamma \cdot
\vol \left( \left\{ \mathbf{v} \in \mathcal{B}_{n+1}\left( \frac{\lambda}{4 M_{d,n}} \right) :
|f_{\mathbf{a}}(\mathbf{x} + \mathbf{v})| \leq \frac1{2^{2d+1} \gamma} \right\} \right).
\end{equation*}
It follows that
\begin{equation}
\label{Lower bound tau}
\tau(\mathbf{a};\gamma) \geq \gamma \int_{\mathcal{B}_n(\lambda/8M_{d,n})}
\mathcal{W}_{\mathbf{a}, \mathbf{x}}(\mathbf{w}; \lambda, \gamma) \mathrm{d} \mathbf{w},
\end{equation}
where $\mathbf{w}=(v_1, \dots, v_n)$ and
\begin{equation*}
\mathcal{W}_{\mathbf{a}, \mathbf{x}}(\mathbf{w}; \lambda, \gamma) =
\vol \left( \left\{ v_0 \in [-\lambda/8M_{d,n}, \lambda/8M_{d,n}] :
|f_{\mathbf{a}}(\mathbf{x} + (v_0,\mathbf{w}))| \leq \frac1{2^{2d+1} \gamma} \right\} \right).
\end{equation*}

The inequalities \eqref{Inequalities nabla} imply that we can clearly assume without loss of generality that 
\begin{equation}
\label{Inequalities partial derivative}
\frac{\lambda}{2^dn} < \left| \frac{\partial f_{\mathbf{a}}}{\partial x_0}(\mathbf{x}) \right| \leq \frac{\lambda}{2^{d-1}}.
\end{equation}
We can thus make the change of variables
\begin{equation*}
v_0 = - \left( \frac{\partial f_{\mathbf{a}}}{\partial x_0}(\mathbf{x}) \right)^{-1}
\left( \frac{\partial f_{\mathbf{a}}}{\partial x_1}(\mathbf{x}) v_1 + \cdots + \frac{\partial f_{\mathbf{a}}}{\partial x_n}(\mathbf{x}) v_n - w_0 \right).
\end{equation*}
The upper bound \eqref{Inequalities nabla} shows that there exists an absolute constant $L>0$ such that if
$\mathbf{w} \in \mathcal{B}_n(L \lambda)$ and $|w_0| \leq L \lambda^2$ then
$|v_0| \leq \lambda/8M_{d,n}$. Using the upper bound \eqref{Inequalities partial derivative} and the assumption $f_{\mathbf{a}}(\mathbf{x}) = 0$ we thus deduce that for $\mathbf{w} \in \mathcal{B}_n(L \lambda)$, we have
\begin{equation*}
\mathcal{W}_{\mathbf{a}, \mathbf{x}}(\mathbf{w}; \lambda, \gamma) \gg \frac1{\lambda} \cdot
\vol \left( \left\{ w_0 \in [-L \lambda^2, L \lambda^2] :
|w_0 + P_{\mathbf{a}, \mathbf{x}}(w_0, \mathbf{w})| \leq \frac1{2^{2d+1} \gamma} \right\} \right),
\end{equation*}
where $P_{\mathbf{a}, \mathbf{x}}(w_0, \mathbf{w})$ is a polynomial free of constant and linear terms in
$(w_0, \mathbf{w})$. As a result, if $L>0$ is chosen sufficiently small and if
$\mathbf{w} \in \mathcal{B}_n(L \lambda)$ then for $w_0 \in [L \lambda^2/2, L \lambda^2]$ we have
\begin{equation*}
w_0 + P_{\mathbf{a}, \mathbf{x}}(w_0, \mathbf{w}) \geq \frac{w_0}{2},
\end{equation*}
and for $w_0 \in [-L \lambda^2, - L \lambda^2/2]$ we have
\begin{equation*}
w_0 + P_{\mathbf{a}, \mathbf{x}}(w_0, \mathbf{w}) \leq - \frac{w_0}{2}.
\end{equation*}
Given $\mathbf{w} \in \mathcal{B}_n(L \lambda)$, it follows from the intermediate value theorem that there exists
$\omega_{\mathbf{a}, \mathbf{x}}(\mathbf{w}) \in (- L \lambda^2/2, L \lambda^2/2)$ such that
\begin{equation*}
\omega_{\mathbf{a}, \mathbf{x}}(\mathbf{w}) +
P_{\mathbf{a}, \mathbf{x}}(\omega_{\mathbf{a}, \mathbf{x}}(\mathbf{w}), \mathbf{w}) = 0.
\end{equation*}
In addition, if $L>0$ is small enough then for $\mathbf{w} \in \mathcal{B}_n(L \lambda)$ and
$|w_0| \leq L \lambda^2/2$ we have
\begin{equation*}
\left| \frac{\partial P_{\mathbf{a}, \mathbf{x}}}{\partial w_0}(w_0, \mathbf{w}) \right| \leq 1,
\end{equation*}
so the mean value inequality gives
\begin{equation*}
|w_0 + P_{\mathbf{a}, \mathbf{x}}(w_0, \mathbf{w})| \leq 2 |w_0 - \omega_{\mathbf{a}, \mathbf{x}} (\mathbf{w})|.
\end{equation*}
We have thus proved that if $L>0$ is sufficiently small then for $\mathbf{w} \in \mathcal{B}_n(L\lambda)$ we have
\begin{align*}
\mathcal{W}_{\mathbf{a}, \mathbf{x}}(\mathbf{w}; \lambda, \gamma) & \gg \frac1{\lambda} \cdot
\vol \left( \left\{ w_0 \in [-L \lambda^2/2, L \lambda^2/2] :
|w_0 - \omega_{\mathbf{a}, \mathbf{x}}(\mathbf{w})| \leq \frac1{2^{2d+2}\gamma} \right\} \right) \\
& \gg \frac1{\lambda} \cdot \min \left\{ \lambda^2, \frac1{\gamma} \right\}.
\end{align*}
Recalling the lower bound \eqref{Lower bound tau} we eventually derive
\begin{equation*}
\tau(\mathbf{a};\gamma) \gg \frac{ \gamma}{\lambda} \cdot \min \left\{ \lambda^2, \frac1{\gamma} \right\} \cdot
\vol \left( \mathcal{B}_n(L \lambda) \right),
\end{equation*}
which completes the proof.
\end{proof}

We are now ready to establish Proposition~\ref{Proposition Archimedean}.

\begin{proof}[Proof of Proposition~\ref{Proposition Archimedean}]
We follow closely the lines of the proof of Proposition~\ref{Proposition non-Archimedean}. Recall the respective definitions \eqref{Definition singular integral} and \eqref{Definition alpha} of the Archimedean factor $\mathfrak{J}_V(B)$ and the quantity $\alpha$, where we take $B = A \phi(A)$. It is convenient to set 
\begin{equation*}
\mathscr{I}_{\phi}(A) = \frac1{\#\mathbb{V}_{d,n}^{\mathrm{loc}}(A)} \cdot
\# \left\{ V \in \mathbb{V}_{d,n}^{\mathrm{loc}}(A) : \mathfrak{J}_V(A\phi(A)) < \frac{C}{\phi(A)^{1/6}} \right\}.
\end{equation*}
Recall also the definition \eqref{Definition Iloc} of the set $\mathbb{I}_{d,n}^{\mathrm{loc}}$. The lower bound \eqref{Lower bound loc} yields in particular $\# \mathbb{V}_{d,n}^{\mathrm{loc}}(A) \gg A^{N_{d,n}}$ so we see that
\begin{equation*}
\mathscr{I}_{\phi}(A) \ll \frac1{A^{N_{d,n}}} \cdot
\# \left\{ \mathbf{a} \in \mathbb{Z}^{N_{d,n}} \cap \mathcal{B}_{N_{d,n}}(A) \cap \mathbb{I}_{d,n}^{\mathrm{loc}} :
\tau(\mathbf{a};\alpha) < \frac{C}{\phi(A)^{1/6}} \right\}.
\end{equation*}
It follows from Lemma~\ref{Lemma upper bound boundary} that
\begin{equation}
\label{Upper bound Iphi}
\mathscr{I}_{\phi}(A) \ll \frac1{A^{N_{d,n}}} \cdot
\# \left\{ \mathbf{a} \in \mathbb{Z}^{N_{d,n}} \cap \mathcal{B}_{N_{d,n}}(A) :
\begin{array}{l l}
\mathcal{N}(\mathbf{a}) \subset \mathbb{I}_{d,n}^{\mathrm{loc}} \\
\tau(\mathbf{a};\alpha) < \dfrac{C}{\phi(A)^{1/6}}
\end{array}
\right\} + \frac1{A}.
\end{equation}

We now show that if $\mathbf{a} \in \mathbb{R}^{N_{d,n}}$ satisfies $||\mathbf{a}|| \geq 8 \alpha$ then for any
$\mathbf{y} \in \mathcal{N}(\mathbf{a})$, we have
\begin{equation}
\label{Upper bound tau y}
\tau(\mathbf{y}; 2\alpha) \leq 2 \cdot \tau(\mathbf{a}; \alpha).
\end{equation}
Let $\mathbf{u} \in \mathcal{B}_{n+1}(1)$ be such that $\mathbf{y} \in \mathcal{C}_{\nu_{d,n}(\mathbf{u})}^{(2 \alpha)}$, that is
\begin{equation*}
| \langle \nu_{d,n}(\mathbf{u}), \mathbf{y} \rangle | \leq \frac{||\nu_{d,n}(\mathbf{u})|| \cdot ||\mathbf{y}||}{4 \alpha}.
\end{equation*}
Since $\mathbf{y} \in \mathcal{N}(\mathbf{a})$ the Cauchy--Schwarz inequality gives
$|\langle \nu_{d,n}(\mathbf{u}), \mathbf{y}-\mathbf{a} \rangle | \leq ||\nu_{d,n}(\mathbf{u})||$. We thus see that
\begin{equation*}
| \langle \nu_{d,n}(\mathbf{u}), \mathbf{a} \rangle | \leq \frac{||\nu_{d,n}(\mathbf{u})|| \cdot ||\mathbf{y}||}{4 \alpha} +
||\nu_{d,n}(\mathbf{u})||.
\end{equation*}
We have assumed that $||\mathbf{a}|| \geq 8 \alpha$ so we have in particular $||\mathbf{y}|| \leq 3||\mathbf{a}||/2$. We deduce that
\begin{equation*}
| \langle \nu_{d,n}(\mathbf{u}), \mathbf{a} \rangle | \leq \frac{3 ||\nu_{d,n}(\mathbf{u})|| \cdot ||\mathbf{a}||}{8 \alpha} +
||\nu_{d,n}(\mathbf{u})||.
\end{equation*}
Our assumption $||\mathbf{a}|| \geq 8 \alpha$ now yields
\begin{equation*}
| \langle \nu_{d,n}(\mathbf{u}), \mathbf{a} \rangle | \leq \frac{||\nu_{d,n}(\mathbf{u})|| \cdot ||\mathbf{a}||}{2 \alpha},
\end{equation*}
which shows that $\mathbf{a} \in \mathcal{C}_{\nu_{d,n}(\mathbf{u})}^{(\alpha)}$, and the upper bound
\eqref{Upper bound tau y} follows.

Recalling the upper bound \eqref{Upper bound Iphi} we note that the inequality \eqref{Upper bound tau y} gives
\begin{align*}
\mathscr{I}_{\phi}(A) & \ll \frac1{A^{N_{d,n}}} \cdot
\# \left\{ \mathbf{a} \in \mathbb{Z}^{N_{d,n}} \cap \left( \mathcal{B}_{N_{d,n}}(A) \smallsetminus \mathcal{B}_{N_{d,n}}(8 \alpha) \right) :
\begin{array}{l l}
\mathcal{N}(\mathbf{a}) \subset \mathbb{I}_{d,n}^{\mathrm{loc}} \\
\tau(\mathbf{a};\alpha) < \dfrac{C}{\phi(A)^{1/6}}
\end{array}
\right\} + \frac1{A} \\
& \ll \frac1{A^{N_{d,n}}} \sum_{\mathbf{a} \in \mathbb{Z}^{N_{d,n}} \cap \mathcal{B}_{N_{d,n}}(A)}
\vol \left( \left\{ \mathbf{y} \in \mathcal{N}(\mathbf{a}) \cap \mathbb{I}_{d,n}^{\mathrm{loc}} :
\tau(\mathbf{y}; 2\alpha) < \frac{2C}{\phi(A)^{1/6}} \right\} \right) + \frac1{A}.
\end{align*}
Swapping the summation over $\mathbf{a}$ and the integration over $\mathbf{y}$ we obtain
\begin{align*}
\mathscr{I}_{\phi}(A) & \ll \frac1{A^{N_{d,n}}} \cdot
\vol \left( \left\{ \mathbf{y} \in \mathcal{B}_{N_{d,n}}(A+1) \cap \mathbb{I}_{d,n}^{\mathrm{loc}} :
\tau(\mathbf{y}; 2\alpha) < \frac{2C}{\phi(A)^{1/6}} \right\} \right) + \frac1{A} \\
& \ll \vol \left( \left\{ \mathbf{y} \in \mathcal{B}_{N_{d,n}}(1) \cap \mathbb{I}_{d,n}^{\mathrm{loc}} :
\tau(\mathbf{y}; 2\alpha) < \frac{2C}{\phi(A)^{1/6}} \right\} \right) + \frac1{A}.
\end{align*}

We are now in position to make use of a trick analogous to the one used in the non-Archimedean setting. It is clear that
$\tau(\mathbf{y}; 2\alpha) > 0$ whenever $\mathbf{y} \in \mathbb{I}_{d,n}^{\mathrm{loc}}$. We let
$\kappa \in (0, 2/n)$ and we write
\begin{equation}
\label{Upper bound Iphi 2}
\mathscr{I}_{\phi}(A) \ll \frac1{\phi(A)^{\kappa/6}} \int_{\mathcal{B}_{N_{d,n}}(1) \cap \mathbb{I}_{d,n}^{\mathrm{loc}}} 
\frac{\mathrm{d} \mathbf{y}}{\tau(\mathbf{y}; 2 \alpha)^{\kappa}} + \frac1{A}.
\end{equation}
Recall the respective definitions \eqref{Definition Blambda} and \eqref{Definition Mdn} of the set
$\mathcal{B}_{N_{d,n}}^{(\lambda)}$ for given $\lambda>0$ and the quantity $M_{d,n}$. We deduce from the equality \eqref{Partition Archimedean} that
\begin{equation*}
\int_{\mathcal{B}_{N_{d,n}}(1) \cap \mathbb{I}_{d,n}^{\mathrm{loc}}} 
\frac{\mathrm{d} \mathbf{y}}{\tau(\mathbf{y}; 2 \alpha)^{\kappa}} \leq
\sum_{\ell=1}^{\infty} \int_{\mathcal{B}_{N_{d,n}}^{(M_{d,n}/2^{\ell})}}
\frac{\mathrm{d} \mathbf{y}}{\tau(\mathbf{y}; 2 \alpha)^{\kappa}}.
\end{equation*}
It thus follows from Lemma~\ref{Lemma lower bound tau} that
\begin{equation*}
\int_{\mathcal{B}_{N_{d,n}}(1) \cap \mathbb{I}_{d,n}^{\mathrm{loc}}} 
\frac{\mathrm{d} \mathbf{y}}{\tau(\mathbf{y}; 2 \alpha)^{\kappa}} \ll
\sum_{\ell=1}^{\infty} \left( \frac{2^{\ell(n+1)}}{\alpha} + 2^{\ell(n-1)} \right)^{\kappa}
\vol \left( \mathcal{B}_{N_{d,n}}^{(M_{d,n}/2^{\ell})} \right).
\end{equation*}
Appealing to Lemma~\ref{Lemma upper bound Blambda} and taking $\kappa=1/n$, we obtain
\begin{equation*}
\int_{\mathcal{B}_{N_{d,n}}(1) \cap \mathbb{I}_{d,n}^{\mathrm{loc}}} 
\frac{\mathrm{d} \mathbf{y}}{\tau(\mathbf{y}; 2 \alpha)^{\kappa}} \ll
\frac1{\alpha^{1/n}} \sum_{\ell=1}^{\infty} \frac1{2^{\ell(n-1)/n}} + \sum_{\ell=1}^{\infty} \frac1{2^{\ell(n+1)/n}}.
\end{equation*}
Recalling the upper bound \eqref{Upper bound Iphi 2}, we therefore conclude that
\begin{equation*}
\mathscr{I}_{\phi}(A) \ll \frac1{\phi(A)^{1/6n}} + \frac1{A}.
\end{equation*}
But by assumption we have $\phi(A) \leq A$ and we thus see that this completes the proof of
Proposition~\ref{Proposition Archimedean}.
\end{proof}

\bibliographystyle{amsplain}
\bibliography{biblio}

\end{document}